\pdfoutput=1
\documentclass{ucetd}

\usepackage{subfigure,epsfig,amsfonts}
\usepackage[square]{natbib}
\usepackage{amsmath}
\usepackage{amssymb}
\usepackage{amsthm}
\usepackage{appendix}
\usepackage{lscape}
\usepackage[nottoc,notlot,notlof]{tocbibind}

\newtheorem{thm}{Theorem}[section]
\newtheorem{prop}[thm]{Proposition}
\newtheorem{lem}[thm]{Lemma}
\let\oldlem\lem
\renewcommand{\lem}{\oldlem\normalfont}
\newtheorem{cor}[thm]{Corollary}
\theoremstyle{definition}
\newtheorem{exmp}{Example}[section]

\theoremstyle{definition}

\newtheorem{assump}[thm]{Assumption}

\theoremstyle{remark}
\newtheorem{remark}[thm]{Remark}

\newcommand{\indep}{\rotatebox[origin=c]{90}{$\models$}}

\title{Exponential Series Approaches for Nonparametric Graphical Models}
\author{Eric Janofsky}
\department{Statistics}
\division{Physical Sciences}
\degree{Doctor of Philosophy}
\date{June 2015}

\dedication{\emph{To grammy}}
\epigraph{Epigraph Text}

\begin{document}
\maketitle

\makecopyright
\makededication

\tableofcontents
\listoffigures
\listoftables

\acknowledgments
First and foremost I thank my advisor John Lafferty. His depth and breadth of knowledge has had a profound influence on me since I first took his statistical machine learning elective in 2012. His assistance, patience and encouragement over the last several years has been invaluable to the completion of this thesis. I also thank my committee members Matthew Stephens and Lek-Heng Lim for their constructive input, as well as the rest of the faculty of the Department of Statistics for the outstanding education I received in the last five years, and before that during my undergraduate studies.

I thank my family for their love and support, especially my wife Liz. She has acted as a source of strength and inspiration throughout this journey, all while pursuing her own doctorate. I couldn't have done it without her.


\abstract

Markov Random Fields (MRFs) or undirected graphical models are parsimonious representations of joint probability distributions. Variables correspond to nodes of a graph, with edges between nodes corresponding to conditional dependencies. For a \emph{pairwise MRF}, the joint density factorizes as a product over edges of the graph. This thesis studies high-dimensional, continuous-valued pairwise MRFs. We are particularly interested in approximating pairwise densities whose logarithm belongs to a Sobolev space. For this problem we propose the method of exponential series \citep{crain1974estimation,barron1991approximation}, which approximates the log density by a finite-dimensional exponential family with the number of sufficient statistics increasing with the sample size.

We consider two approaches to estimating these models. The first is regularized maximum likelihood. This involves optimizing the sum of the log-likelihood of the data and a sparsity-inducing regularizer. We provide consistency and edge selection guarantees for this method. We then propose a variational approximation to the likelihood based on tree-reweighted, nonparametric message passing. This approximation allows for upper bounds on risk estimates, leverages parallelization and is scalable to densities on hundreds of nodes. We show how the regularized variational MLE may be estimated using a proximal gradient algorithm. We demonstrate our method's efficacy in density estimation and model selection in comparison to other approaches in the literature using simulated data and MEG signal data.

We then consider estimation using regularized score matching. This approach uses an alternative scoring rule to the log-likelihood, which obviates the need to compute the normalizing constant of the distribution. For general continuous-valued exponential families, we provide parameter and edge consistency results. As a special case we detail a new approach to sparse precision matrix estimation which has statistical performance competitive with the graphical lasso \citep{yuan2007model} and computational performance competitive with the state-of-the-art glasso algorithm \citep{friedman2008sparse}. We then describe results for model selection in the nonparametric pairwise model using exponential series. The regularized score matching problem is shown to be a convex program; we provide scalable algorithms based on consensus alternating direction method of multipliers (ADMM, \citep{boyd2011distributed}) and coordinate-wise descent.  We use simulations to compare our method to others in the literature as well as the aforementioned TRW estimator.

\mainmatter

\chapter{Introduction}

Density estimation is one of the fundamental tools in statistics and machine learning \citep{silverman1986density}. Nonparametric density estimators are used for prediction, goodness-of-fit testing \citep{fan1994testing,bickel1973some}, generative models for classification \citep{fix1989discriminatory,john1995estimating}, inferring independence and conditional independence, estimating statistical functionals \citep{beirlant1997nonparametric,poczos2012nonparametric}, as well as exploratory data analysis and visualization. In many fields including the biological and social sciences, information technology and machine learning, it is now commonplace to analyze large data sets with complex dependencies between many variables. This includes gene expression levels measured using microarrays, brain activity measurements from fMRI or MEG technology, prices of financial instruments, and the activity of individuals on social media web sites. The use of nonparametric density estimators is limited, however, by the curse of dimensionality: with even moderate sample size, high-dimensional space will invariably have large regions where the data is sparse, leading to uninformative predictions. As such nonparametric estimators typically have poor risk guarantees in high-dimensions. Furthermore, many methods suffer from a computational curse of dimensionality and heuristics or approximations to tune the models become necessary.

In this work we consider nonparametric estimation of \emph{pairwise densities}, a class of densities intimitely tied to undirected graphical models. The undirected graphical model or Markov Random Field \citep{jordan2004graphical,lauritzen1996graphical} is a well-studied framework for representing joint dependence structures of random variables. An undirected graph $G = (V,E)$ consists of a vertex set $V=\{1,\ldots,d\}$ corresponding to the elements of the random vector $X=(X_1,\ldots,X_d)$, and an edge set $E\subseteq V\times V$. Each edge $e\in E$ is an unordered pair of elements $j,k\in V$, $e=(j,k)$. For any subset $A\subseteq V$, we define the subset $\{ X_A : X_i, i\in A\}$. Furthermore, for sets $A,B,C$ we write $X_A \indep X_B \big| X_C$ to mean $X_A$ and $X_B$ are independent conditional on $X_C$. The random vector $X$ is \emph{Markov} with respect to the graph $G=(V,E)$ if for every $j,k\in V$, $X_j \indep X_k \big| (X_l : l \not=j,k)$ if and only if $(i,j)\not\in E$.

The fundamental theorem of undirected graphical models is the Hammersley-Clifford theorem \citep{dobruschin1968description}, which states that if the density of $X$ $p(x)$ is positive, then the following are equivalent: 

\begin{enumerate}
	\item $X$ is Markov with respect to the graph $G$;
	\item The density of $X$, $p(x)$ can be factorized over the cliques of $G$:
\begin{equation}
	p(x) = \prod_{C \in cl(G)} \psi_C(x_C), \label{hclif}
\end{equation}
where $x_C=\{x_i : i\in C\}$.
\end{enumerate}

This thesis considers nonparametric estimation of the pairwise graphical model for continuous-valued data, where the joint density can be further factored into a product of potential functions over edges:
\begin{equation}
	p(x) = \prod_{i\in V} \psi_i(x_i)\prod_{(j,k)\in E} \psi_{jk}(x_j,x_k). \label{pair}
\end{equation}
Pairwise graphical models have been used extensively in modeling discrete data. The Ising model \citep{ising1925beitrag} for $\{ 0,1 \}$-valued variables has the density
\begin{equation*}
	p(x) = \exp\left\{ \sum_{i\in V} \theta_i x_i + \sum_{(j,k)\in E} \theta_{jk}x_j x_k  - Z(\theta) \right\},
\end{equation*}
which can be seen as a pairwise graphical model with $\psi_i = \exp\{\theta_i x_i \} $ and $\psi_{jk}=\exp\{\theta_{jk} x_j x_k \}$ and $Z(\theta)$ a normalizing constant. The Ising model can be generalized to discrete variables with more than two levels, but there are a finite possibilities for pairwise discrete potentials with finite number of levels.

The class of continuous-valued pairwise models is considerably more complex than discrete ones, as the potential functions $\{\psi_i,\psi_{jk}\}$ could be any positive-valued functions such that $p$ integrates to 1. The class of continuous pairwise graphical models includes some familiar models, which we describe below.

\begin{exmp}{Gaussian graphical model}

Let $X\in\mathbb{R}^d$ be a Gaussian-distributed random variable with mean $\mu$ and covariance matrix $\mathbb{E}\big((X-\mu)(X-\mu)^\top\big)=\Sigma\succ 0$. Denote $\Omega=\Sigma^{-1}$. Then $X$ has density
\begin{eqnarray*}
	p(x) &=& \sqrt{\frac{\mid \Omega \mid}{(2\pi)^d }} 
			\exp\bigg\{-\frac{1}{2} (x-\mu)^\top \Omega (x-\mu)\bigg\}  \\
		&\propto& \prod_{i\in V} \exp\bigg\{-\frac{1}{2}\Omega_{ii} x_{i} ^2 + (\Omega \mu)_i x_i \bigg\} \notag\\
		& &	\times\prod_{(i,j)\in E} \exp\bigg\{-\frac{1}{2} \Omega_{ij} x_i x_j\bigg\}. \notag
\end{eqnarray*}
From the factorization above, it can be seen that the Gaussian graphical model is not only a Markov random field, but also belongs to the pairwise class of densities. Following from \eqref{hclif}, observing that the Gaussian density is positive over $\mathbb{R}^d$, two Gaussian variables $X_i,X_j$ are conditionally independent given the others if and only if $(\Sigma^{-1})_{ij} = 0$.
\end{exmp}

\begin{exmp}{Gaussian copula graphical model}

Let $X$ be Gaussian distributed with mean $\mu$ and covariance matrix $\Sigma$, and suppose that for each $i\in V$, $Y_i = g_i(X_i)$ where $g_i$ is some monotonic increasing, smooth function. Suppose further that the $g_i$ are centered and scaled so that $\mathbb{E}[g_i(X_i)]=\mathbb{E}[X_i]$ and $\text{var}[g_i(X_i)]=\text{var}[X_i]$. Write $f=g^{-1}$. Denote the vector of functions $f=(f_1,f_2,\ldots,f_d)$ and their derivatives by $f'=(f'_1,f'_2,\ldots f'_d)$. After applying a change of variables to the Gaussian density, we find that $Y$ has density
\begin{eqnarray*}
	p(y)&=& \sqrt{\frac{\mid \Omega \mid}{(2\pi)^d }}  
			\exp\bigg\{ -\frac{1}{2} (f(y)-\mu)^{\top} \Omega (f(y)-\mu) \bigg\}
			\prod_{i\in V} f_i'(y_i) \\
		&\propto& \prod_{i\in V} 
			f'_i(y_i)\exp\bigg\{-\frac{1}{2}\Omega_{ii} f_i(y_i) ^2 + (\Omega\mu)_i f_i(y_i) \bigg\} \\
		& &	\times\prod_{(i,j)\in E} \exp\bigg\{-\frac{1}{2} \Omega_{ij} f_i(y_i) f_j(y_j)\bigg\}.
\end{eqnarray*}

Thus the Gaussian copula density is also a Markov random field and has a pairwise factorization. Following from \eqref{hclif}, two variables $Y_i,Y_j$ are conditionally independent given the others if and only if $\Omega_{ij} = 0$. Gaussian copulas have been used extensively in finance and risk management \citep{cherubini2004copula} for their ability to model dependence between many variables which are (marginally) non-Gaussian.
\end{exmp}

\begin{exmp}{Forests}\label{forest}

A \emph{tree} $T$ is an undirected graph where each pair of vertices is connected by exactly one simple path. Equivalently, a tree is a connected graph with no cycles. A \emph{forest} is an undirected graph where each pair of vertices is connected by at no more than one simple path. Equivalently, a forest is a graph with no cycles. A \emph{spanning tree} $T$ of a connected graph $G$ is a tree containing the vertices of $G$ and a subset of the edges of $G$. A \emph{spanning forest} of a graph $G$ is a graph consisting of a spanning tree for each connected component of $G$. In the sequel will use the term spanning tree to refer to a spanning tree or forest unambiguously whether or not $G$ is connected.

A tree or forest must have cliques of size no more than two. Thus, from the Hammersley-Clifford theorem, a tree density has the factorization
\begin{equation}
	p(x) = \prod_{i\in V} \psi_i(x_i) \prod_{(i,j)\in E} \psi_{ij}(x_i,x_j).
\end{equation}
It follows that all tree distributions have pairwise densities. In particular, a tree can always be factorized in the form
\begin{equation}
	p(x) = \prod_{i \in V} p_i(x_i) \prod_{(i,j)\in E} \frac{p_{ij}(x_i,x_j)}{p_i(x_i)p_j(x_j)} \label{treefac},
\end{equation}
where $\{p_i\}$ are the set of univariate densities and $\{p_{ij}\}$ are the set of bivariate densities of the joint distribution $p$.

\end{exmp}

\section{Previous Work}
Little work has been done in studying the nonparametric estimation of pairwise densities. \citep{gu2002smoothing,gu1993smoothing2} considered log-ANOVA density estimation, where the log-density can be factored into low-order terms, pairwise densities being a special case. They suggest an estimator for the log-density:
\begin{align}
	\min_\eta \left\{ -\frac{1}{n}\sum_{k=1}^n \eta(X^k) + \log \intop e^\eta + \lambda \Vert \eta \Vert_{\mathcal{H}}^2\right\},
\end{align}
such that $\eta$ has a given pairwise factorization, where $\Vert\cdot\Vert_{\mathcal{H}}$ is a norm in a reproducing kernel Hilbert space (RKHS); the resulting density estimate is proportional to $e^{\eta (x)}$. Due to the representer theorem \citep{kimeldorf1971some}, this becomes a finite-dimensional optimization problem. However, due to the difficulty of computing $\log\intop e^\eta$ it can only be used in low-dimensional problems, such as dimension up to 3. This work also assumes the ANOVA factorization structure is known. \citep{jeon2006effective} proposed a smoothing spline estimator based on minimizing the \emph{Bregman score} \eqref{bregman} for log-ANOVA densities, and applied it to undirected graphical model estimation. This solves the problem
\begin{align}
	\min_\eta \left\{\frac{1}{n}\sum_{k=1}^n e^{-\eta(X^k)} + \intop \eta(y)\rho(y)dy + \lambda\Vert \eta\Vert_{\mathcal{H}}^2\right\},
\end{align}
for some given density $\rho$; the resulting density estimate is proportional to $\rho(x)e^{\eta(x)}$. The authors show that solving this problem only requires calculation of one-dimensional integrals, and is thus more scalable. This approach has some limitations. The issue of computing the normalizing constant remains; this is a problem for inference and for choosing the smoothing parameter $\lambda$ when using cross-validation to minimize the $\text{KL}$ risk. Also, their procedure for graph selection is a heuristic. The theoretical properties of this estimator are not yet known. 

When the graph is assumed to be a forest (Example \ref{forest}), density estimation and structure learning was considered in \citep{forest2011}. Due to the tree entropy factorization \eqref{treeent}, learning a forest density can be done in two steps: estimating the univariate and bivariate marginals, and learning the graph structure. For the first task, they estimate marginals using kernel density estimation. For the second task, they use a nonparametric estimator of mutual information \eqref{mutinfo}, and then estimate the maximum likelihood forest using Kruskal's algorithm (Figure \ref{kruskal}), with weights corresponding to mutual information between edges. They show consistency guarantees for their approach in terms of $\text{KL}$ risk and graph selection.

There has been a large amount of work on parametric pairwise models. For learning Gaussian models with sparse precision matrix, the graphical lasso, or the $L_1$-regularized maximum likelihood is the most popular approach \citep{yuan2007model}; this solves the problem
\begin{align}
	\min_{\Omega\succ 0} \left\{\text{trace}(\Omega\widehat{\Sigma}) - \log\left|\Omega\right| + \lambda\Vert\Omega\Vert_1\right\}, \label{glasso}
\end{align}
$\widehat{\Sigma}$ being the sample covariance matrix $\widehat{\Sigma}=\frac{1}{n}\sum_{k=1}^n (X^k)(X^k)^\top$ and $\Vert\Omega\Vert_1=\sum_{i,j}\left|\Omega_{ij}\right|$ . The glasso algorithm solves the resulting problem using block-coordinate descent \citep{banerjee2008model,friedman2008sparse}. The graphical lasso is known to have good properties in terms of parameter and structure selection consistency \citep{rothman2008sparse,ravikumar2011high}. There exist other approaches for sparse estimation of $\Omega$ such as the the graphical Danzig selector \citep{yuan2010high}, and CLIME \citep{cai2011constrained}. The parallel lasso \citep{meinshausen2006high} infers graph structure in a Gaussian graphical model without direct estimation of the covariance or precision matrix by running a sequence of neighborhood lasso fits in parallel. For a node $i$, they estimate the neighborhood of $i$ by solving
\begin{align}
	\widehat{\beta}^i = \underset{\beta:\beta_i=0}{\text{argmin}} \left\{ \frac{1}{2n}\sum_{k=1}^n \left( X^k_i - \beta^\top X^k \right)^2 + \lambda\Vert\beta\Vert_1\right\}.
\end{align}
The corresponding neighborhood of $i$ is the support of $\widehat{\beta}^i$. They show consistency of neighborhood estimates, which may then be aggregated to form an edge set. The precision matrix can then be fit by estimating the Gaussian likelihood subject to sparsity constraints on the precision matrix. \citep{liu2012high} proposes the SKEPTIC estimator for structure learning of the semiparametric Gaussian copula model. The estimator plugs in a matrix of rank correlations (Kendall's $\tau$ or Spearman's $\rho$) into the graphical lasso \eqref{glasso}; they show this estimator achieves the parametric rates for edge selection and parameter estimation.


This thesis includes several contributions to the literature. In Chapter \ref{esestsec} we introduce the exponential series approximation to pairwise densities. We propose an estimator for pairwise densities based on regularized maximum likelihood estimation of a particular exponential family whose sufficient statistics are basis elements. We use a method for edge selection using convex regularization, and provide risk and model selection guarantees in Section \ref{main}. While the exact problem is in general not tractable, in Chapter \ref{trwsec} we propose a convex variational upper bound on the likelihood based on a nonparametric tree-reweighted relaxation \citep{wainwright2005new,wainwright2003tree}, which can be computed efficiently and in parallel. Our method provides an upper bound on the normalizing constant, guaranteeing an upper bound on risk estimates. The approximation leads to a natural variational maximum likelihood estimator, as well as an approach for marginalization. We train our method using a projected gradient algorithm, which can effortlessly be scaled to relatively sparse graphs on hundreds of nodes. In Section \ref{expersec} we compare our method to several other approaches to large-scale density estimation, including the graphical lasso, mixtures of Gaussians with the EM algorithm, and kernel forest density estimation. We demonstrate our method by estimating the graph from an MEG neuroimaging dataset.

In Chapter 4 we consider a different approach to estimation and graph selection using an alternative scoring rule to the log-likelihood. It is based on minimizing the log-gradient between the model distribution and data distribution, or equivalently minimizing the \emph{Fisher divergence}. This method obviates the need for computing a normalizing constant. We show that the optimization amounts to a second-order cone program, and provide two types of scalable algorithms specially tailored to the problem. Our method, which we denote QUASR for \textbf{Qua}dratic \textbf{S}coring and \textbf{R}egularization, produces parameter and graph selection consistency for general pairwise exponential families with only weak regularity conditions. En route, we derive a new method for sparse precision matrix estimation which performs competitively with the regularized MLE \eqref{glasso}. Finally, we show how this approach produces graph selection guarantees for the pairwise nonparametric model when the sufficient statistics are basis elements.

\section{Notation and Preliminaries}

Throughout this thesis we assume we are given independent and identically distributed data $X^1,X^2,\ldots,X^n$, where $X^k=(X_1^k,\ldots,X_d^k)$, drawn from the density $p(x)$ with respect to a reference measure $\nu(x)$. In the context of density estimation, we assume the unknown log density $f=\log p$ belongs to the Sobolev space of functions on $[0,1]^d$, $W_r^2$, so that for any multi-index $\alpha$ with $\mid\alpha\mid\leq r$,
\begin{equation}
	f^{(\alpha)} := \frac{\partial^{(\alpha)}f}{\partial x_1^{\alpha_1}\cdots \partial x_d^{\alpha_d}}
\end{equation}

has bounded $L_2(\nu)$ norm:
\begin{equation}
	\Vert f^{(\alpha)}\Vert := \intop_{\mathcal{X}} \mid f^{(\alpha)}\mid^2 \nu(dx) < \infty.
\end{equation}
This implies that $p$ is bounded away from zero and infinity and that $p(x)$ has bounded support.

We use the asymptotic notations $O(\cdot)$, $o(\cdot)$, $\Omega(\cdot)$ and $\asymp$. For two functions $f(n),g(n)$, $f=O(g)$ if $f/g \leq c$ for a constant $c>0$ as $n\rightarrow\infty$; $f=o(g)$ if $f/g\rightarrow 0$ as $n\rightarrow\infty$; $f=\Omega(g)$ if $f> c g$ for some $c>0$ as $n\rightarrow\infty$; $f\asymp g$ if $f = c g$ for some $c>0$ as $n\rightarrow\infty$. We say that $f=O_p(g)$ if $f=O(g)$ with probability approaching one as $n\rightarrow\infty$.

Let $\{\phi_k,\phi_{k'l},k,k',l=1,2,\ldots\}$ be a tensor product basis for $L_2[0,1]^2$, so that $\phi_{kl}=\phi_k\phi_l$. We suppose this basis is uniformly bounded and orthonormal. Consider the density $p$  having a pairwise factorization \ref{pair}, so that $f=\log p$ can be expressed with the basis expansion
\begin{equation}
	 f(x)= f_0(x) + \theta_0^* + \sum_{i,j\in V} \sum_{k,l=1}^\infty (\theta^*)_{ij}^{kl}\phi_k(x_i) \phi_l(x_j)  + \sum_{i\in V} \sum_{k=1}^\infty (\theta^*)_i^k \phi_k(x_i) \label{basis},
\end{equation}

$\exp f_0$ is some base measure which has the same pairwise factorization as $p$. We will take $f_0=0$, but our results also apply whenever $f_0$ has the same smoothness assumptions as $f$. $\theta,\phi$ represent the parameters and sufficient statistics vectorized. $\theta_e,\theta_v$ denote the vectors of edge and vertex parameters, respectively. 
For a $f\in W_r^2$, for any $i,j\in V$ and $r_i+r_j=r$, we additionally have that
\begin{align}
	\sum_k (\theta^{^*k}_i)^2 k^{2r} = C_1<\infty, \\
	\sum_{k,l} (\theta^{*kl}_{ij})^2 k^{2r_i}l^{2r_j} = C_2<\infty.
\end{align} 

This implies that $\left|(\theta^*)_i^k\right| = o(k^{-r-1/2})$ and $\left|(\theta^*)_{ij}^{kl}\right|=o(k^{-r_i-1/2}l^{-r_j-1/2})$.

 We denote $\Pi$ to be the set of densities on $[0,1]^d$.
For risk analysis in Chapter 2 we use the relative entropy, also known as the Kullback-Leibler divergence:
\begin{equation}
	\text{KL}(p \mid \widehat{p}) = \intop_{\mathcal{X}} p(x) \log \bigg(\frac{p(x)}{\widehat{p}(x)}\bigg) d\nu(x) = \mathbb{E}_p\bigg[\log\bigg(\frac{p(X)}{\widehat{p}(X)}\bigg)\bigg].
\end{equation}

It can be shown that $\text{KL}(p \mid \widehat{p})\geq 0$ and equals zero only if $p=\widehat{p}$ $\nu$-almost everywhere. Relative entropy is a natural risk measure for density estimation. It is invariant to invertible changes of variables, and shares a natural connection to maximum likelihood. 
Convergence in $\text{KL}$ is strong in the sense that it implies convergence in several other risk measures. In particular, define the $L_1$, Hellinger and Total Variation distances as follows:
\begin{eqnarray}
	D_1(p \mid \widehat{p}) &=& \intop_{\mathcal{X}} \mid p(x)-\widehat{p}(x) \mid d\nu(x), \\
	D_H(p \mid \widehat{p}) &=& \intop_{\mathcal{X}} \big| p(x)^{1/2} - \widehat{p}(x)^{1/2} \big| ^2 d\nu(x), \\
	D_{TV}(p \mid \widehat{p}) &=& \sup_{A\in \mathcal{X}} \bigg| \intop_A p(x) d\nu(x) - \intop_A \widehat{p}(x) d\nu(x) \bigg|,
\end{eqnarray}

By Pinsker's inequality and \citep{reiss1989approximate,kullback1967lower}, we have
\begin{eqnarray}
	\text{KL}(p \mid \widehat{p}) &\geq& \frac{D_1(p \mid \widehat{p})^2}{2}, \\
	\text{KL}(p \mid \widehat{p}) &\geq& D_H(p \mid \widehat{p})^2, \\
	\text{KL}(p \mid \widehat{p}) &\geq& 2 D_{TV}(p \mid \widehat{p})^2.
\end{eqnarray}

%


\chapter{ Exponential Families and the Exponential Series Regularized MLE}\label{esestsec}

\section{Exponential Series Approximation}
Consider an approximation to $f=\log p$ \eqref{basis} by truncating the basis expansion in the following way:
\begin{equation}
	\log p_\theta(x) = \sum_{i=1}^d\sum_{k=1}^{m_1} \theta_i^k\phi_k(x_i) + \sum_{i,j\in V}\sum_{k=1}^{m_2}\sum_{l=1}^{m_2} \theta_{ij}^{kl}\phi_{kl}(x_i,x_j) - Z(\theta), \label{exseries}
\end{equation}

This approximation is an exponential family with sufficient statistics corresponding to the basis functions $\{\phi_k(x_i),i\in V,k\leq m_1\}$ and $\{\phi_k(x_i)\phi_l(x_j), i,j\in V, k,l \leq m_2\}$. $Z(\theta)$ is chosen so the density integrates to one. The idea of representing a density as an exponential expansion was first used for goodness of fit testing in \citep{neyman1937smooth}. Nonparametric estimation of univariate distributions using exponential series has been studied previously in \citep{crain1974estimation,crain1977information,barron1991approximation}. For simplicity we assume that each univariate component is truncated after $m_1$ terms and bivariate components after $m_2$ terms. In practice we could attempt to vary truncation for each variable, though this could be unwieldy for large problems. When not ambiguous we will write $\log p_\theta (x) = \langle \theta, \phi (x) \rangle - Z(\theta)$. Given $n$ i.i.d. samples $X^1,\ldots,X^n$, the \emph{exponential series MLE estimator} finds the regularized maximum likelihood estimator of $\theta$:
\begin{eqnarray}
	\widehat{\theta} &:=& \underset{\theta}{\text{argmin}}\left\{-\frac{1}{n}\sum_{r=1}^n \log p_\theta (X^r) + \lambda \mathcal{R}(\theta)\right\}, \\
				&=&  \underset{\theta}{\text{argmin}} \bigg\{ -\mathcal{L}(\theta) + \lambda\mathcal{R}(\theta)\bigg\} \label{mlsol}.
\end{eqnarray}
$\mathcal{R}$ is a convex regularizer; we discuss our particular choice of regularization in Section \ref{structure}. The resulting density estimate is $p_{\widehat{\theta}}(x)$. The fact that $\theta$ depends on the truncation parameters $m_1,m_2$ and regularization parameter $\lambda$ is left implicit.

Exponential series is a natural approach for the estimation of pairwise densities. The product factorization of pairwise densities can be expressed naturally by exponential series. Indeed, many common parametric graphical models are exponential families. Each truncated series forms an exponential family. We detail exponential families and the regularized maximum likelihood in the sequel.

\section{Exponential Families}
This section presents background on exponential families and their use in graphical modeling. For an in-depth treatment, see \citep{brown1986fundamentals,wainwright2008graphical}.

Consider the random vector $X=(X_1,\ldots,X_d)$ taking values on the support $\mathcal{X}=\bigotimes_{i=1}^d \mathcal{X}_i$. The \emph{exponential family} with sufficient statistics $\phi(x)=(\phi_1(x),\phi_2(x),\ldots,\phi_M(x))$ is the family of probability distributions
\begin{equation}
	\bigg\{ p_\theta: p_\theta(x) = \exp\bigg\{ \langle\theta,\phi(x)\rangle - Z(\theta)\bigg\} , \theta\in\mathcal{P} \bigg\}, \label{exfam}
\end{equation}
where $\langle \theta , \phi(x) \rangle = \sum_{a=1}^M \theta_a \phi_a(x)$.  $Z(\theta)$ is called the \emph{log-partition function}, and is given by
\begin{equation}
	Z(\theta) = \log \int_{\mathcal{X}} \exp\bigg\{ \langle \theta,\phi \rangle \bigg\} \nu(dx),
\end{equation}
and ensures the density $p_\theta(x)$ integrates to one. The family is indexed by the parameters $\theta$, called the \emph{natural parameters} belonging to the space $\mathcal{P} = \{\theta: Z(\theta)<\infty \}$. 

An exponential family is \emph{minimal} if there is no choice of parameters $\theta\not=0$ such that $\langle\theta,\phi(X) \rangle = C$ $\nu$-a.e., where $C$ is a constant. If a family is minimal, there is a bijection between the natural parameter space $\mathcal{P}$ and the densities belonging to the family. An exponential family is \emph{regular} if $\mathcal{P}$ is an open set.

\begin{exmp}[Gaussian exponential family]
	Consider the Gaussian density with mean $\mu$ and covariance $\Sigma\succ 0$. The Gaussian family is an exponential family with sufficient statistics $\{ x , xx^\top \}$ and natural parameters $(\theta_1,\theta_2)=\{\Omega\mu,-\frac{1}{2}\Omega \}$. Since $\Sigma\succ 0$, the natural parameter space is
\begin{equation*}
	\mathcal{P} = \{(\theta_1,\theta_2)\in \mathbb{R}^d \times \mathbb{R}^{d\times d} : \theta_2 \prec 0 \}.
\end{equation*}
The space of negative definite matrices is an open convex set, so it follows that $\mathcal{P}$ is convex and the Gaussian family is regular. Furthermore, linear independence of monomials implies that $x$ and $xx^\top$ are linearly independent, and so the Gaussian family is also minimal.
\end{exmp}

\begin{remark}[Exponential series]
Consider the exponential series family of \eqref{exseries}. When $\{\phi_k,\phi_{kl}\}$ is an orthogonal basis which satisfies the Haar condition, the exponential series family corresponds to a minimal exponential family. Furthermore, since the exponential series are defined have compact support $[0,1]^d$ and the sufficient statistics $\phi$ are bounded above and below, any choice of $\theta\in \mathbb{R}^M$ will produce a valid $Z(\theta)<\infty$; in other words $\mathcal{P}=\mathbb{R}^M$, which is a convex and open set, so the exponential series family is regular.
\end{remark}
%
%
\section{Mean Parametrization}

An exponential family is parametrized by its so-called \emph{natural parameters} $\theta$ \eqref{exfam}. Alternatively, an exponential family may also be characterized by a vector of \emph{mean parameters} $\mu$. The connection between these seemingly disparate entities will be described in Lemma \ref{conn1}.

Let $p$ be any density on $\mathcal{X}$ with respect to $\nu$. We define the mean parameter $\mu_a$ corresponding to the sufficient statistic $\phi_a$ by
\begin{equation}
	\mu_a = \intop p(x) \phi_a(x) d\nu(x) = \mathbb{E}_p[\phi_a].
\end{equation}
Consider the set of vectors that correspond to the moments of some distribution:  $\mathcal{M}=\{\mu: \exists p\in\Pi, \mu_a=\mathbb{E}_p[\phi_a],\forall a=1,\ldots,M\}$. In particular, the elements of $\mathcal{M}$ need not correspond to mean parameters of an exponential family. Furthermore, $\mathcal{M}$ is a convex set. To see this, let $\mu_1,\mu_2\in\mathcal{M}$ be mean parameters corresponding to two distributions $p_1,p_2$ with respect to $\nu$. Then 
\begin{equation*}
	\lambda\mu_1+(1-\lambda)\mu_2 = \mathbb{E}_{\lambda p_1 +(1-\lambda)p_2}[\phi].
\end{equation*}
For discrete variables, one can further show that $\mathcal{M}$ is a convex polytope \citep{wainwright2008graphical}. This fact is exploited in many inference algorithms for discrete graphical models, but this is not true for continuous random variables. Characterizing $\mathcal{M}$ for general continuous sufficient statistics is in general very challenging. It is closely connected to the so-called \emph{moment problem} \citep{landau1987moments} which has been studied since the late 19th century. For polynomial sufficient statistics, $\mathcal{M}$ can be characterized by a sequence of semidefinite constraints on the moments \citep{lasserre2009moments}. This is suggested by the positive semidefinite constraint on the covariance matrix $\Sigma$ for Gaussian densities.

We now state several important facts relating the natural parameters $\theta$ and the mean parameters $\mu$. For proofs, see \citep{wainwright2008graphical}.

\begin{lem}\label{conn1} Suppose $\theta$ corresponds to the natural parameters of an exponential family with sufficient statistics $\phi(x)$ and corresponding mean vector $\mu(\theta)=\mathbb{E}_{p_\theta}[\phi]$. Let $Z(\theta)$ be the corresponding log-partition function, and define its gradient $\nabla Z(\theta): \mathcal{P} \rightarrow \mathcal{M}$. The following hold:

\begin{enumerate} 
	\item $\theta$ and $\mu$ are related by the mapping 
	\begin{equation}
		\nabla Z(\theta) = \mu(\theta);
	\end{equation}
	\item $\nabla^2 Z(\theta) = \text{cov}_\theta[\phi]$;
	\item $Z(\theta)$ is a convex function, and strictly so if the family is minimal, so that $\delta^\top(\text{cov}_\theta[\phi])\delta > 0$ for each $\delta\not=0$  \label{strongconv};
	\item The mapping $\nabla Z(\theta): \mathcal{P} \rightarrow \mathcal{M}$ is one-to-one if and only if the family is minimal;
	\item The mapping $\nabla Z(\theta)$ is onto the interior of $\mathcal{M}$, $\text{int}(\mathcal{M})$ if the family is minimal\label{onto}.
\end{enumerate}
\end{lem}

\begin{remark}
 The exponential series family is minimal when the orthogonal series satisfies the \emph{Haar condition}, so its likelihood is strictly convex, and \eqref{mlsol} is a convex problem so long as the regularizer $\mathcal{R}$ is convex. 
\end{remark}

\section{Duality} \label{dualitysec}

For any function $Z$ taking values $\theta\in\mathcal{P}$, we define the \emph{Fenchel conjugate}, $Z^*$, as follows:
\begin{align}
	Z^*(\mu) &= \text{sup}_{\theta\in\mathcal{P}} \left\{\langle \theta,\mu \rangle - Z(\theta)\right\} \label{fenchel},
\end{align}
When $Z$ corresponds to the log-partition function this equation bears a strong resemblance to the maximum of the log-likelihood \eqref{mlsol}. Indeed when $\mu $ corresponds to empirical mean parameters  $\widehat{\mu}= \frac{1}{n}\sum_{k=1}^n \phi(X^k)$ it is precisely that, though \eqref{fenchel} is well-defined when $\mu$ doesn't correspond to a $\mu\in\mathcal{M}$.

For $\mu\in\text{int}\mathcal{M}$ corresponding to a minimal family, let $\theta(\mu)$ denote the unique natural parameters corresponding to $\mu$. Denote 
\begin{equation}
H(\mu):=H(p_{\theta(\mu)}) = -\mathbb{E}_{p_{\theta(\mu)}}[\log p_{\theta(\mu)}]
\end{equation}
the entropy of the density $p_{\theta(\mu)}$. Furthermore, denote the univariate entropy by 
\begin{equation}
H_{i}\left(\mu\right):=-\mathbb{E}_{p_{\theta(\mu)}}\big[\log p_{i,\theta(\mu)}\left(X_{i}\right) \big],
\end{equation}
and bivariate mutual information 
\begin{equation}
I_{ij}\left(\mu\right):=\mathbb{E}_{p_{\theta(\mu)}}\bigg[\log \left(\frac{p_{ij\theta(\mu)}(X_{i},X_{j})}{p_{i,\theta(\mu)}(X_i)p_{j,\theta(\mu)}(X_j)}\right)\bigg]. \label{mutinfo}
\end{equation}

\begin{thm}
The Fenchel conjugate of the log-partition function $Z$ is given by
		\begin{align} Z^*(\mu) &= 
			\begin{cases}
			-H(\mu), & \mu \in \text{int}\mathcal{M}; \\
			\infty, & \mu\not\in\mathcal{M};
			\end{cases} \label{logpartfenchel}
		\end{align}
			for $\mu\in\mathcal{M}\backslash\text{int}\mathcal{M}$, $Z^*(\mu)$ is given by the limit of $Z^*(\mu^k)$ for any sequence $\{\mu^k\}$, $\mu^k\in\text{int}\mathcal{M}$.
\end{thm}

%
%
%
%
%
%
%
%

\begin{exmp}[Gaussian Entropy]
Let $X\sim N(0,\Sigma)$, $\Sigma\succ 0$, and $\Omega=\Sigma^{-1}$. Denote $p$ the density of $X$. Then
\begin{align}
	-\mathbb{E}_p[\log p] &= \frac{d}{2}\log(2\pi) - \frac{1}{2}\log\left| \Omega\right| + \frac{1}{2}\mathbb{E}_p[X^\top\Omega X].
\end{align}

Now, $X^\top\Omega X=\text{trace}(\Omega XX^\top)$, and since the trace is a linear operator, we may move the expectation inside the trace, giving
\begin{align}
	\mathbb{E}_p[X^\top \Omega X] &= \text{trace}\left(\Omega\mathbb{E}_p[XX^\top]\right) \\
	&= \text{trace}\left(\Omega\Sigma\right) \\
	&= \text{trace}(I_d)=d.
\end{align}

Thus the Gaussian entropy is
\begin{align}
	\frac{d}{2}\left(1+\log(2\pi)\right)-\frac{1}{2}\log\left| \Omega\right|.
\end{align}
\end{exmp}

\begin{exmp}[Tree Entropy and Maximum Likelihood Trees]
	Let $\theta$ be the parameters of a minimal exponential family which is tree-structured: that is, any edge parameters $\theta_{ij}=0$ for $(i,j)\not\in T$, where $T$ is the edge set for a tree. Because of the tree density factorization \eqref{treefac},
	\begin{align}
		H(\mu(\theta)) &= - \mathbb{E}[\log p_{\theta(\mu)}] \\
				&= \sum_{i\in V} H_i(\mu(\theta))  - \sum_{(i,j)\in T}I_{ij}(\mu(\theta)). \label{treeent}
	\end{align}
Thus, for a tree-factored distribution has a simple expression for its entropy in terms of the univariate entropies and bivariate mutual informations.

\end{exmp}

\section{Main Results}\label{main}
\subsection{Sparsity} \label{structure}
For our risk analysis, make the sparsity assumption on $\theta^*$, that $\theta^*\in\tilde{\mathcal{P}}$, where
\begin{equation}
	\tilde{\mathcal{P}}=\tilde{\mathcal{P}}(E) := \bigg\{\theta: \Vert\theta_{ij}\Vert_2=0,\forall (i,j)\not\in E \bigg\} \subseteq\mathcal{P}.
\end{equation}

However, the set $E$ is unknown. Recall that for the pairwise graphical model, $(i,j)\not{\in} E$ when $\theta_{kl}^{ij}=0$ for each $k,l=1,\ldots,\infty$; in other words, when $\Vert \theta^{i,j}\Vert_2=0$. For clarity, we refer to $\theta^*_v,\mu_v$ to be the vectors of vertex parameters, $\theta_e,\mu_e$ to be the vectors of edge parameters, and $\theta_{ij},\mu_{ij}$ to be the vector of parameters corresponding to edge $(i,j)$. To encourage edge sparsity we consider the penalty
\begin{equation}
	\mathcal{R}(\theta_e) = \sum_{i, j\in V,i<j} \Vert \theta_{ij}\Vert_2. 
\end{equation}
	
$\mathcal{R}$ defines a norm over the edge parameters. Furthermore, $\mathcal{R}$ has some special properties which we detail below.

\begin{prop}
The dual norm of $\mathcal{R}$ is
\begin{equation}
	\mathcal{R}^*(\theta_e) = \max_{(i,j)\in E} \Vert \theta_{ij} \Vert_2 .
\end{equation}
\end{prop}

$\mathcal{R}$ is known as the (1,2)-group penalty, and its dual the $(\infty,2)$-group penalty. In our application, the groups correspond to parameters of given edges. Group penalties are best known from their use in the group lasso \citep{yuan2006model}, which is used to encourage group sparsity in regression coefficients.
For a vector of parameters $\theta$ denote its projection onto $\tilde{\mathcal{P}}$ by $\theta_{\tilde{\mathcal{P}}}=\sum_{(i,j)\in E}\Vert\theta_{ij}\Vert$, and its projection onto its orthogonal complement by $\theta_{\tilde{\mathcal{P}}^\perp}=\sum_{(i,j) \in E^c}\Vert\theta_{ij}\Vert$.
\begin{prop}
	$\mathcal{R}$ is decomposable with respect to $\tilde{\mathcal{P}}$. That is,
	\begin{align}
		\mathcal{R}(\delta+\theta) = \mathcal{R}(\delta)+\mathcal{R}(\theta),
	\end{align}
	for each $\delta\in\tilde{\mathcal{P}}^\perp$ and $\theta\in\mathcal{P}$.
\end{prop}
The following proposition characterizes the \emph{subspace compatibility constant} for $\mathcal{R}$, which is necessary in the proofs.
\begin{prop}
\begin{align}
	\sup_{u\in\tilde{\mathcal{P}} \backslash \{0\}}\frac{\mathcal{R}(u)}{\Vert u \Vert} &\leq \sqrt{\left| E\right|}.
\end{align}
\end{prop}

We will state our main theoretical result for the regularized exponential series MLE. A full derivation of the results are in Appendix \ref{proofs}.

We start with three assumptions:
\begin{assump}
\emph{Haar Condition}: Any truncated collection of basis elements$\{\bar{\phi}\}$
 is linearly independent $\nu-a.e.$.
\end{assump}
\begin{assump} \label{assump2}
The univariate basis functions satisfy for each $k$, $\left|\phi_k\right| \leq b(k)=O(k^\alpha)$ for some $\alpha\geq 0$.
\end{assump}

\begin{assump} \label{assump3}
		 For each $i,j\in V$,
	\begin{align}
		\underline{\epsilon} & \leq p_{ij}(x_i,x_j) \leq \bar{\epsilon},
		\end{align}
		for absolute constants $\underline{\epsilon}>0,\bar{\epsilon}<\infty$.
\end{assump}

These assumptions are mild. Many  bases satisfy assumptions (1) and (2), such as the standard polynomial or trigonometric bases. For the orthonormal Legendre basis, $\left|\phi_k\right| \leq \sqrt{2k+1}$, so it satisfies Assumption \ref{assump2} with $\alpha=\frac{1}{2}$. The use of an overcomplete basis creates some statistical difficulties as the resulting truncated exponential family is no longer minimal. Assumption \ref{assump3} is mild for density estimation as we only require boundedness of the bivariate marginals of $p$ rather than of $p$ itself.

The natural first question is whether the problem \eqref{mlsol} has a solution at all, and if so, how many solutions. We begin by showing the existence and uniqueness of \eqref{mlsol}.
\begin{lem} \label{uniqueness}
Suppose $n>m_1$. The solution \eqref{mlsol} exists and is unique with probability one.
\end{lem}

\subsection{Risk Consistency}

We now present consistency of $p_{\widehat{\theta}}$ in terms of the $\text{KL}$ risk.

\begin{thm}\label{riskthm}
Suppose that the regularization parameter is chosen to be 
\begin{align}
\lambda_n\asymp\sqrt{\frac{m_2^{2+4\alpha}\log(m_2d)}{n}},
\end{align}
 and the truncation parameters satisfy
\begin{align}
	m_1= \Omega\left(d^{\frac{1}{r-\alpha-1/2}}\right), \\
	m_2 = \Omega\left(\left| E\right|^{\frac{1}{r-\alpha-1/2}}\right).
\end{align}
 Then the regularized exponential series MLE $\widehat{\theta}_\lambda$ satisfies
	\begin{align}
		\text{KL}(p \mid p_{\widehat{\theta}}) = O_p\left( m_2^{-2r}\left| E\right|+m_1^{-2r}d + \frac{m_2^{2+4\alpha} \log (d m_2)}{n}\left| E\right| + \frac{m_1}{n} d\right). \label{klrate}
	\end{align}
\end{thm}
\begin{cor}
The optimal choice of truncation dimensions $m_1,m_2$ is
\begin{align*}
	&m_1 \asymp \max\left\{ n^{\frac{1}{2r+1}},d^{\frac{1}{r-\alpha-1/2}}\right\}, \\
	&m_2 \asymp \max\left\{n^{\frac{1}{2r+2+4\alpha}},\left| E\right|^{\frac{1}{r-\alpha-1/2}}\right\}.
\end{align*}

Consider typical choices of $r=2$, $\alpha=\frac{1}{2}$ (such as the Legendre basis).
\begin{itemize}
\item
The dimension $d$ and edge cardinality $\left| E \right|$ may scale as
\begin{align}
	d= o(\sqrt{n}), \\
	\left| E\right| = o(n^{\frac{1}{5}}/\log(n)),
\end{align}	
with the risk still approaching zero as $n\rightarrow \infty$. 

\item Suppose that $d=O(n^{\frac{1}{5}})$ and $\left| E\right| =O(n^{\frac{1}{8}})$. Then by choosing
\begin{align}
	m_1\asymp n^{\frac{1}{5}}, \\
	m_2\asymp n^{\frac{1}{8}}, \\
	\lambda\asymp \frac{\log^{1/2}(nd)}{n^\frac{1}{4}},
\end{align}
the risk decreases as
\begin{align}
		\text{KL}(p \mid p_{\widehat{\theta}}) &= O_p\left(\frac{\left|E\right|\log(dn)}{n^{\frac{1}{2}}}+\frac{d}{n^{\frac{4}{5}}}\right).
\end{align}
\end{itemize}

\end{cor}
\begin{remark}
The result of \ref{riskthm} holds uniformly for any set $B$ of pairwise densities $p$ with bounded Sobolev norm. In particular, assuming $\mid E \mid = o(n^{1/2})$ and $d=o(n^{4/5})$,

\begin{align}
	\lim_{t\rightarrow\infty}\lim_{n\rightarrow\infty}\sup_{p\in B} \mathbb{P}\left(\text{KL}(p \mid \widehat{p}) \geq \frac{\mid E\mid \log dn}{n^{\frac{1}{2}}} t + \frac{d}{n^{\frac{4}{5}}} t \right)= 0.
\end{align}
\end{remark}
\begin{remark}
 Theorem \ref{riskthm} shows that the risk of $p_{\widehat{\theta}}$ adapts to the unknown sparsity of $\theta^*$, in that the risk contains a factor of $\left| E\right|$ rather than $d^2$. However, this is not sufficient for model selection consistency, which requires further assumptions. In particular consistency in $\text{KL}$ risk does not require an \emph{incoherence} condition. We consider model selection in the next section.
\end{remark}
\subsection{Model Selection}

Our result for model selection consistency requires more stringent assumptions, in addition to those in the previous section. We denote the vector of truncated parameters by $\bar{\theta}^*$. We index the (infinite) vector of omitted parameters by $T$, so that the vector of parameters are $\theta^*_T$. We use the subscript $E$ to denote the collection of parameters of edges in $E$ (as well as all vertex parameters), and $E^c$ to denote parameters for edges in $E^c$.  Denote the covariance matrix of the sufficient statistics $\phi$ by
\begin{align}
		\Gamma := \text{cov}_{p}[\phi].
\end{align}
	
For two index sets $A,B$ denote $\Gamma_{AB}$ to be the cross-covariance between $\phi_A$ and $\phi_B$, $\text{cov}[\phi_A,\phi_B]$. \begin{assump}
For a constant $\kappa_\Gamma<\infty$,
	\begin{align}
		\Vert\Gamma_{EE}^{-1}\Vert_2 &\leq \frac{\kappa_\Gamma}{m\sqrt{d+\left| E\right|}}, \\
		\kappa_T := \Vert \Gamma_{ET}\Vert_{\infty}.
	\end{align}
\end{assump}
where $\Vert A\Vert_\infty=\max_i \sum_j \left| A_{ij}\right|$ here denotes the matrix $\infty$ norm and $\Vert A\Vert_2$ the matrix operator norm. Additionally, we define the following:
\begin{align}
	\bar{K}_{\theta^\dagger} := \mathbb{E}_{\theta^\dagger} \left[U \cdot(\bar{\phi}-\mathbb{E}_{\theta^\dagger}[\bar{\phi}])\right],
\end{align}

where 
\begin{align}
U=(\Vert \bar{\phi}_E-\mathbb{E}_{\theta^\dagger}[\bar{\phi}_E]\Vert_1+\Vert \phi_T-\mathbb{E}_{\theta^\dagger}[\phi_T]\Vert_1)^2,
\end{align}

and $\theta^\dagger:=\theta^*+z(\widehat{\theta}-\theta^*)$ and $z\in[0,1]$.
\begin{assump}
For some $\kappa_R<\infty$, for all $z\in[0,1]$ and all $\widehat{\theta}$ satisfying

\begin{align}
\widehat{\theta}_{E^c}=0, \\
\Vert\bar{\theta}^*_E-\widehat{\theta}_E\Vert_\infty &\leq 2\kappa_\Gamma\left(\Vert\widehat{\mu}_E-\mu^*_E\Vert_\infty + \lambda_n/m + (\kappa_T+1)\Vert\theta^*_T\Vert_\infty\right),
\end{align}

we have that
\begin{align}
	\Vert\bar{K}_{\theta^\dagger}\Vert_\infty &\leq \kappa_R\max\{b(m_2)^2,b(m_1)\}.
\end{align}
\end{assump}
This assumption may appear opaque so we will elaborate. If $p_{\theta}$ is bounded, the third central moment of the univariate statistic $\phi_k$ is

\begin{align}
	\mathbb{E}_\theta[ (\phi_k-\mathbb{E}_\theta[\phi_k])^3]
		&\leq b(k) \mathbb{E}_\theta[ (\phi_k-\mathbb{E}_\theta[\phi_k])^2] \\
		&\leq \bar{\epsilon}b(k)\intop \phi_k^2 \\
		&\leq \bar{\epsilon}b(k).
\end{align}

This holds similarly for bivariate sufficient statistics. $\bar{K}_{\theta^\dagger}$ is a sum of a third central moment of a sufficient statistic and the third cross-moments of that statistic with the other sufficient statistics. We thus require that the third cross central moments between sufficient statistics decay sufficiently rapidly, so that the sum is on the order as stated. In our proof, this factors in to the \emph{remainder term}, which is the bias from truncating the infinite expansion of the log density.
Finally, we have an \emph{irrepresentable condition}

\begin{assump}
\begin{align}
\max_{(i,j)\in E^c} \Vert \Gamma_{ij,E}\Gamma_{EE}^{-1}\Vert_2 \leq \frac{1-\tau}{\sqrt{d+E}}, && \text{for some } \tau\in (0,1].
\end{align}
\end{assump}
Here $\Vert A\Vert_2$ is the matrix operator norm. This condition is reminiscent of the irrepresentable condition for sparse additive models in \citep{ravikumar2009sparse}, in that it involves the operator norm rather than the matrix $\infty$ norm. It guarantees that no sets of variables in $E$ and $E^c$ are too strongly influenced.
In the following theorem we assume $\kappa_\Gamma,\kappa_T,\kappa_R$ grow as constants, though they are tracked in the supplementary lemmas. Finally we define $\rho^*=\min_{(i,j)\in E} \Vert\bar{\theta}^*\Vert_\infty$ to be the minimum $\infty$ norm of the edge parameters.
\begin{thm}\label{modelselect}
Denote $\widehat{E}$ to be the edge set learned from $\widehat{\theta}_{\lambda_n}$; $\widehat{E}:=\{(i,j): \Vert\widehat{\theta}_{ij}\Vert_2=0\}$. If the truncation dimensions $m_1,m_2$ and regularization parameter $\lambda_n$ satisfy
 \begin{align}
 	&m_2\asymp n^{\frac{1}{2r+4\alpha+1}}, \\
	&m_1\asymp n^{\frac{1}{2r+4\alpha+1}}, \\
	&\lambda_n \asymp \sqrt{\frac{\log(nd)}{n^{\frac{2r-1}{2r+4\alpha+1}}}}
 \end{align}
 and suppose that the number of variables $d$ and $\rho^*$ satisfy
 \begin{align}
 	d=o\left(e^{n^{\frac{2r-1}{2r+4\alpha+1}}}\right), \\
	\frac{1}{\rho^*} = o\left(\sqrt{\frac{n^{\frac{2r+1}{2r+1+4\alpha}}}{\log(nd)}}\right),
 \end{align}
 then
 \begin{align}
 	\mathbb{P}(\widehat{E}=E)\rightarrow 1.
 \end{align}
 \end{thm}
 
\section{Discussion}
 
Proofs and supporting lemmas for this chapter may be found in Chapter 5, but we will briefly discuss the results here. In \citep{barron1991approximation}, the optimal choice of truncation for univariate exponential series approximation was $m_1\asymp n^{\frac{1}{2r+1}}$, and for the bivariate problem we have $m_2\asymp n^{\frac{1}{2r+2}}$. Our truncation is of a lower order. In our proofs, in order for our estimator to adapt to the unknown sparsity of $E$ we require exponential concentration for the sufficient statistics. To do this, we use Hoeffding's inequality, which gives that
\begin{align}
	\Vert\widehat{\mu}_{ij}-\mu_{ij}\Vert = O_p\left(\sqrt{\frac{m_2^{2+4\alpha}}{n}}\right).
\end{align}
\citep{barron1991approximation} use Chebyshev's inequality, gives a tighter bound of $\sqrt{\frac{m_2^2}{n}}$ but doesn't give exponential concentration.

Our model selection results require conditions on the covariance matrix of the sufficient statistics $\Gamma$, particularly an incoherence condition. This is natural; for example, for Gaussian graphical model selection, the same incoherence condition is required on the covariance matrix of the sufficient statistics; in this application the covariance has the simple expression $\Sigma\otimes\Sigma$ \citep{ravikumar2011high}. For the typical choices of $r=2$ and $\alpha=\frac{1}{2}$, the dimension may scale nearly exponentially with the sample size,
	\begin{align}
		d=o\left(e^{n^\frac{3}{7}}\right),
	\end{align}
and $\rho^*$ may scale as
\begin{align}
	\frac{1}{\rho^*} &= o\left(\frac{n^{5/14}}{\log^{1/2}(nd)}\right),
\end{align}
	with $\widehat{E}=E$ with probability approaching one. Observe that the $\text{KL}$ risk of $p_{\widehat{\theta}}$ may diverge rapidly while still having the correct sparsity pattern with high probability. For the parametric Gaussian graphical model the optimal rate is $o(e^n)$, \citep{ravikumar2011high} which is also the optimal rate for forests \citep{liu2012exponential}. The optimal choice of $m_2$ for model selection is larger than that for the risk analysis, so by \emph{oversmoothing} the edge potentials, we get better sample complexity for the model selection problem. This phenomenon was also found for graph selection for forests \citep{liu2012exponential}.
	
\chapter{Tree-Reweighted Variational Likelihood Approximation} \label{trwsec}

First-order optimization procedures for solving the regularized maximum likelihood problem requires evaluation of $Z(\theta)$ and its gradient $\nabla Z(\theta)=\intop_\mathcal{X}p_\theta(x) \phi(x) d\nu(x)$ . For a pairwise density on $d$ nodes, these computations still require a d-dimensional integral even when the graph is sparse. Using the junction tree algorithm \citep{koller2009probabilistic}, it is possible to factorize the joint density into terms which have no more variables than the \emph{treewidth} of the graph, making these calculations simpler. However, the treewidth of a graph may in general be large, and we are interested in procedures which work for general graphs.  

Monte Carlo methods are one popular approach for approximating partition functions \citep{gilbert1992global}. However, it may take a very long time for suitable convergence, and such methods generally don't provide finite-time bounds on the accuracy of the approximation.

Here we pursue a tree-reweighted variational approach \citep{wainwright2005new}, which replaces $Z(\theta)$ by a surrogate $Q(\theta)$. Our method has several key advantages:
\begin{enumerate}
	\item It computes an approximation to $Z(\theta)$ and its gradient $\nabla Z(\theta)$ all in one pass using a parallelizable message passing algorithm;
	\item $Q(\theta)$ is guaranteed to be an upper bound on $Z(\theta)$ for each $\theta$, with the tightness dictated by variational parameters;
	\item It is based on convex optimization of the variational parameters, so there are deterministic stopping criteria for computing $Q(\theta)$;
	\item $Q(\theta)$, like $Z(\theta)$, will be strictly convex in $\theta$, so there are criterion for global convergence of the approximate maximum likelihood.
\end{enumerate}

\section{Problem Formulation}

Recall that a density with a tree graph can be factorized in the form \eqref{treefac}. A density corresponding to a graph $G$ with cycles cannot in general be factorized. Instead we will consider the collection $\mathcal{T}$ of spanning trees of $G$. Consider an exponential family following the graph $G$ having natural parameters $\theta$. For a spanning tree $T\in\mathcal{T}$, consider the vector $\theta^T$  that obeys $T$: $\theta^T_{ij}=0$ if $(i,j)\not\in T$; in shorthand we write $\theta^T_{ij}$ to be the vector of parameters corresponding to edge $(i,j)$.

Now consider writing the parameter value $\theta$ as a convex combination of spanning tree parameters:
\begin{eqnarray*}
\theta & = & \sum_{T\in\mathcal{T}}\alpha_{T}\theta^{T},
\end{eqnarray*}
by the convexity of the log-partition function $Z\left(\theta\right)$,
we have
\begin{eqnarray}
Z\left(\theta\right) & \leq & \sum_{T\in\mathcal{T}}\alpha_{T}Z\left(\theta^{T}\right).
\end{eqnarray}

Now, we may form the tightest upper bound on the log-partition function
by solving
\begin{align}
Q(\theta, \alpha):= &\min_{\left\{ \theta^{T}\right\} _{T\in\mathcal{T}}}\left\{ \sum_{T\in\mathcal{T}}\alpha_{T}Z\left(\theta^{T}\right)\right\} \label{primal} \\
s.t. \qquad &\theta=\sum_{T}\alpha_{T}\theta^{T}. \notag
\end{align}

Observe that since $Z$ is a convex function, \eqref{primal} is convex in $\{\theta^T\}_{T\in\mathcal{T}}$, and the constraints are linear, so the problem is convex. However, the number of spanning trees of a general loopy graph $G$
could be very large, perhaps even super-exponential in the number of edges \citep{cayley1889theorem}, so the number of parameters to minimize over is
in general very large. Contrary to expectation, it is possible to efficiently solve this problem. To begin, we will look at the dual problem to \eqref{primal}.

\section{Solution to the Dual Problem}
For an edge $(i,j)$, write
\begin{equation}
	\alpha_{ij} = \sum_{T\in\mathcal{T}} \alpha_T 1\{(i,j)\in T \}.
\end{equation}
This is the \emph{edge appearance probability}: the probability edge $(i,j)$ is observed when drawing a spanning tree at random according to the distribution $\{\alpha_T\}$. The set of such edge appearance probability vectors $\{\alpha_{ij}\}$ which can be written as convex combinations of tree indicator vectors, is known as the \emph{spanning tree polytope}, which we denote $\mathcal{S}_\mathcal{T}$. Let
 \begin{eqnarray} 
	\tilde{\Pi} &=& \bigg\{\{q_i,q_{ij}\}: \intop_{\mathcal{X}_j} q_{ij}(x_i,x_j)dx_j=q_i(x_i), \intop_{\mathcal{X}_i\times\mathcal{X}_j} q_{ij}=1, \\
	&& \qquad \intop_{\mathcal{X}_i} q_i =1, q_{ij}\geq 0, q_i \geq 0, (i,j)\in E    \bigg\}, \notag \\
	\tilde{\mathcal{M}} &=& \bigg\{\mu: \exists \{q_i,q_{ij}\}\in\tilde{\Pi}:\mathbb{E}_{q_i}[\phi_i]=\mu_i,\mathbb{E}_{q_{ij}}[\phi_{ij}]=\mu_{ij},(i,j)\in E \bigg\}.
\end{eqnarray}
That is, $\tilde{\Pi}$ is the set of univariate and bivariate densities over $E$ which respect marginalization, and $\tilde{\mathcal{M}}$ is the set of mean parameters which can arise from elements of $\tilde{\Pi}$. To derive the dual problem to \eqref{primal}, we first define the Lagrangian
\begin{eqnarray}
	L(\{\theta^T\}_{T\in\mathcal{T}},\tau) &=& \sum_{T\in\mathcal{T}}\alpha_{T}Z\left(\theta^{T}\right) + \tau^\top \big( \theta - \sum_{T\in\mathcal{T}} \alpha_T \theta^T \big) \\
	&=& \langle \tau , \theta \rangle - \sum_{T\in\mathcal{T}} \alpha_T \big( \langle \tau, \theta^T \rangle - Z(\theta^T) \big)
\end{eqnarray}
To minimize $L$ with respect to $\theta^T$ for a given $T\in\mathcal{T}$, we set the associated derivative to zero: $\frac{\partial L}{\partial \theta^T}=0$. Denoting the optimum by $\theta^T_*$, the solution may be written in terms of the \emph{Fenchel conjugate} of $Z$  (section \ref{dualitysec}):
\begin{equation}
	Z_T^*(\tau) := \max_{\theta^T} \left\{\langle\tau,\theta^T\rangle - Z(\theta^T)\right\} =\langle \tau, \theta^T_* \rangle - Z(\theta^T_*).
\end{equation}
Recall that the dual of $Z$ for a tree-structured parametrization \eqref{treeent} takes the form
\begin{align}
	Z_T^*(\tau) &= \sum_{i\in V} H_i(\tau) - \sum_{(i,j)\in T} I_{ij}(\tau), &\tau\in\tilde{\mathcal{M}},
\end{align}
and so the dual to  \eqref{primal} is
\begin{equation}
\max_{\tau\in\tilde{\mathcal{M}}} \bigg\{ \langle\theta,\tau\rangle+\sum_{i\in V}H_{i}\left(\tau\right)-\sum_{\left(i,j\right)\in E}\alpha_{ij}I_{ij}\left(\tau\right) \bigg\}. \label{dual}
\end{equation}

For some distributions, such as discrete pairwise models and Gaussian models, the entropy and mutual information have a closed form and \eqref{dual} can be solved explicitly. Unfortunately, for continuous models there is typically no such expression. In the following section we show that \eqref{dual} is equivalent to a functional optimization problem, which we solve using message passing.

\section{Functional Message Passing}

Observe that the dual problem finds an optimum over the space of mean parameters realizable by distributions in $\tilde{\Pi}$, so \eqref{dual} is equivalent to the following functional optimization:
\begin{align}
\max_{q \in\tilde{\Pi}} & \bigg\{\sum_{i\in V}\bigg(\langle \theta_{i} , \mathbb{E}_{q_{i}}[\phi_i]\rangle -\mathbb{E}_{q_i}[\log q_i]\bigg)\nonumber \\
 & + \sum_{(i,j)\in E} \bigg(\langle \theta_{ij} , \mathbb{E}_{q_{ij}}[\phi_{ij}] \rangle - \alpha_{ij}\mathbb{E}_{q_{ij}}\bigg[\log\bigg(\frac{q_{ij}}{q_i q_j}\bigg)\bigg]\bigg)\bigg\} \label{eq:1}
\end{align}

If $q_{i}^{*},q_{ij}^{*}$ are solutions to \eqref{eq:1},
the solution to the dual problem \eqref{dual} is given by
\begin{align}
\tau_{i}^{*}(\theta) & =  \mathbb{E}_{q_{i}^{*}}[\phi_i], &i\in V, \label{pseudomoments} \\
\tau_{ij}^{*}(\theta) & =  \mathbb{E}_{q_{ij}^{*}}[\phi_{ij}],\qquad&(i,j)\in E.
\end{align}
The optimization \eqref{eq:1} is an optimization of a convex functional
over a space of linear functional constraints $\tilde{\Pi}$. Any solution to the stationary conditions of the associated Lagrangian will thus correspond to a global optimum. We may derive the stationary conditions using standard arguments from calculus of variations. Write the constraints as
\begin{eqnarray}
C_{i}\left(q_{i}\right) & = & 1-\intop q_{i}\left(x_{i}\right)dx_{i}\\
C_{ij}\left(x_{j},q_{ij}\right) & = & q_{j}\left(x_{j}\right)-\intop q_{ij}\left(x_{i},x_{j}\right)dx_{i}, \\
C_{ji}((x_i,q_{ij}) &=& q_i(x_i) - \intop q_{ij}(x_i,x_j)dx_j,
\end{eqnarray}
and let $\eta_{i},\eta_{ij}\left(x_{j}\right),\eta_{ji}(x_i)$ be the Lagrange
multipliers associated with these constraints. The second and third multipliers are real-valued functions. This gives us the stationary conditions
\begin{eqnarray}
\log q_{i}\left(x_{i}\right) & = & \langle \theta_{i},\phi_i(x_i)\rangle +\sum_{r\in N\left(i\right)}\eta_{ri}\left(x_{i}\right)+\eta_{i},\\
\alpha_{ij}\log \frac{q_{ij}(x_{i},x_{j})}{q_i(x_i)q_j(x_j)}) & = & \langle\theta_{ij},\phi_{ij}(x_i,x_j)\rangle-\eta_{ji}\left(x_{i}\right)-\eta_{ij}\left(x_{j}\right),
\end{eqnarray}
we may simplify the second condition to get
\begin{align}
\log q_{ij}\left(x_{i},x_{j}\right) & = \eta_{i}+\eta_{j}+\langle\theta_{ij},\phi_{ij}(x_i,x_j)\rangle/\alpha_{ij}\notag\\
   & \qquad+\langle\theta_{i},\phi_{i}(x_i)\rangle +\langle\theta_{j},\phi_{j}(x_j)\rangle \notag\\
   & \qquad+\sum_{r\in N\left(i\right)\backslash j}\eta_{ri}\left(x_{i}\right)/\alpha_{ij}+\sum_{r\in N\left(j\right)\backslash i}\eta_{rj}\left(x_{j}\right)/\alpha_{ij}. 
\end{align}

For each $\left(i,j\right)\in E$ , we define a message $M_{ij}:\mathcal{X}_{j}\rightarrow\mathbb{R}$,
and $M_{ji}:\mathcal{X}_{i}\rightarrow\mathbb{R}$, by
\begin{eqnarray}
M_{ij}\left(x_{j}\right) & = & e^{\eta_{ij}\left(x_{j}\right)/\alpha_{ij}}, \\
M_{ji}\left(x_{i}\right) & = & e^{\eta_{ji}\left(x_{i}\right)/\alpha_{ij}},
\end{eqnarray}
so that the solution to \eqref{eq:1} takes the form
\begin{eqnarray}
q_{i}\left(x_{i}\right) & \propto & \exp\{\langle\theta_{i},\phi_i\rangle\}\prod_{j\in N\left(i\right)}M_{ji}\left(x_{i}\right)^{\alpha_{ji}},\\
q_{ij}\left(x_{i},x_{j}\right) & \propto & \exp\{\langle \theta_{ij},\phi_{ij}\rangle/\alpha_{ij}+\langle \theta_i,\phi_i\rangle+\langle\theta_j,\phi_j\rangle\}\\
 &  & \times\frac{\prod_{r\in N\left(i\right)\backslash j}M_{ri}\left(x_{i}\right)^{\alpha_{ri}}\prod_{r\in N\left(j\right)\backslash i}M_{rj}\left(x_{j}\right)^{\alpha_{rj}}}{M_{ji}\left(x_{i}\right)^{1-\alpha_{ji}}M_{ij}\left(x_{j}\right)^{1-\alpha_{ij}}}.\nonumber \label{pseudo}
\end{eqnarray}

\begin{figure}
\centering
\includegraphics[scale=.65]{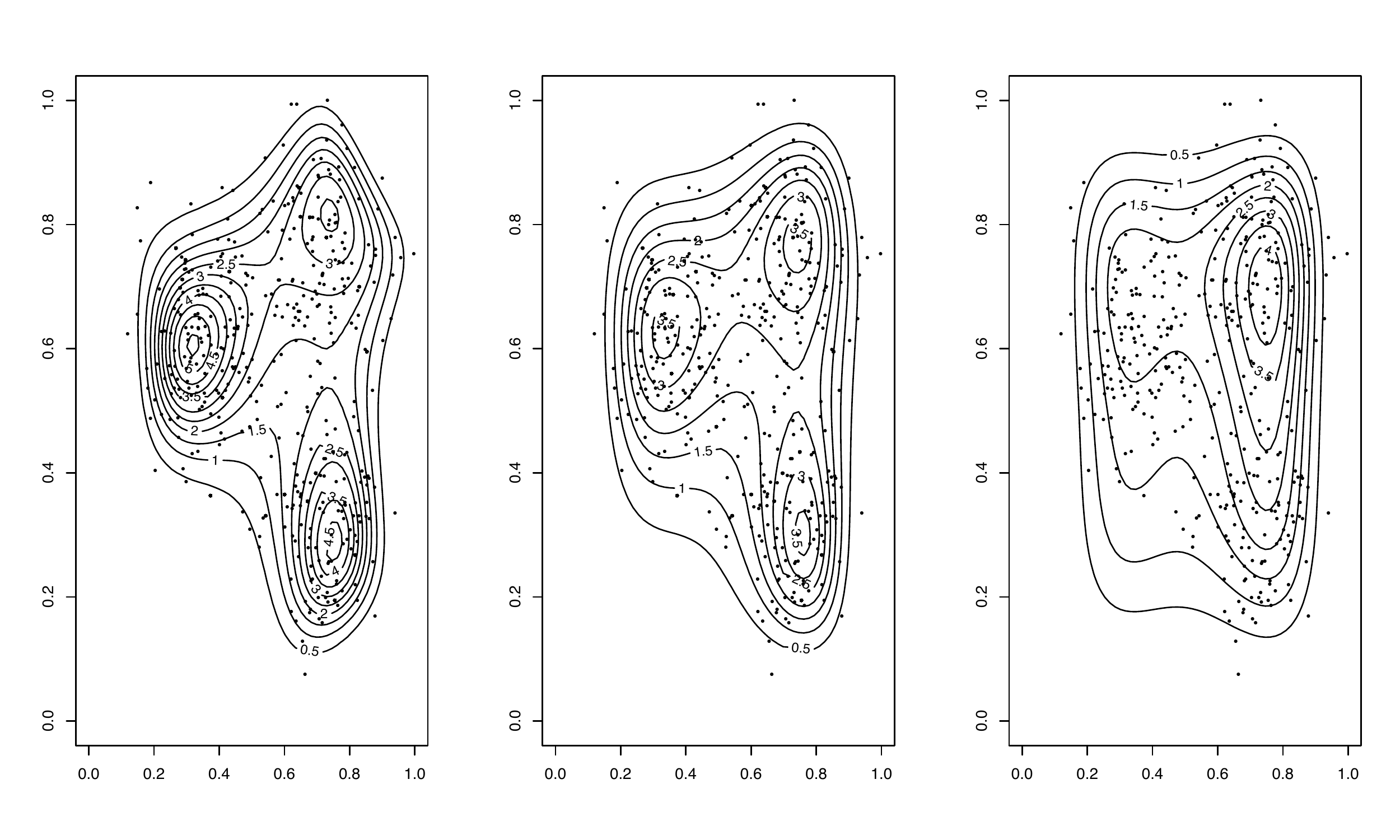}
\caption{Bivariate pseudo densities under increasing regularization. Left plot is unregularized; right plot is fully regularized (dimensions are independent). Simulated data is a mixture of three spherical Gaussians.}
\end{figure}

These are \emph{pseudodensities}: they are valid densities which obey the marginalization constraints, but they may not together correspond to the marginal distributions of any higher-dimensional joint distribution. By enforcing the marginalization constraints for $\{q_i\}$, $\{q_{ij}\}$, we find that the messages follow the fixed-point conditions
\begin{align}
M_{ij}\left(x_{j}\right) & \propto  \intop_{\mathcal{X}_{i}} \exp\{\langle \theta_{ij},\phi_{ij}\rangle/\alpha_{ij}+\langle\theta_i,\phi_i\rangle\}
\frac{\prod_{r\in N\left(i\right)\backslash j}M_{ri}\left(x_{i}\right)^{\alpha_{ri}}}{M_{ji}\left(x_{i}\right)^{1-\alpha_{ji}}}dx_{i} \notag\\
&= \intop_{\mathcal{X}_{i}} \exp\{\langle \theta_{ij},\phi_{ij}\rangle/\alpha_{ij}\rangle\} \left\{\frac{q_i(x_i)}{M_{ji}(x_i)}\right\} dx_i.
\end{align}
To find the fixed point corresponding to the pseudomarginal densities
we run the algorithm in figure \ref{funcbp} to convergence:

\begin{figure}
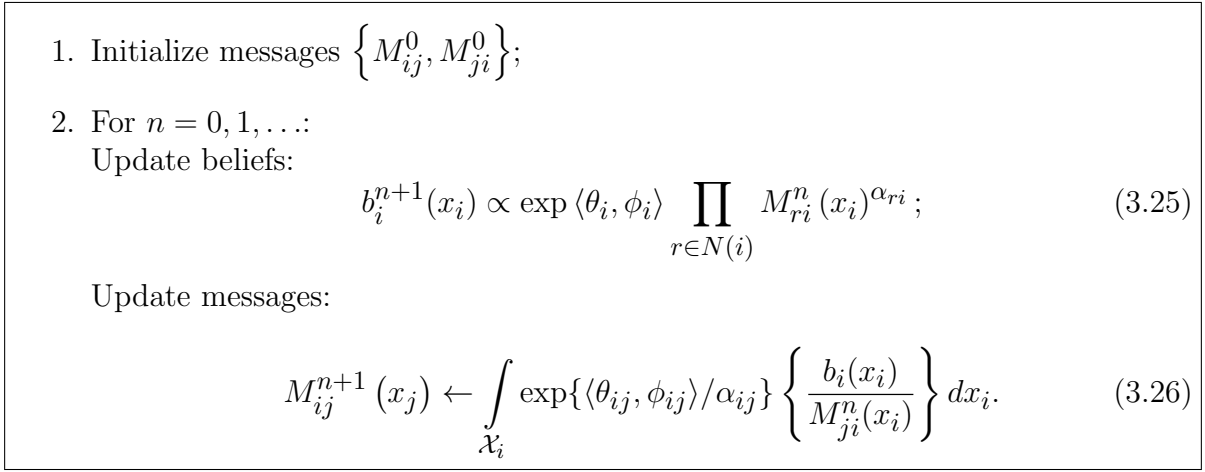
 \label{funcbp}
\begin{center}
\framebox{\begin{minipage}[t]{.95\columnwidth}

\begin{enumerate}
\item Initialize messages $\left\{ M_{ij}^{0},M_{ji}^0\right\}$;

\item For $n=0,1,\ldots$: \\
Update beliefs:
\begin{equation}
	b_i^{n+1}(x_i) \propto \exp{\langle \theta_i , \phi_i \rangle }\prod_{r\in N\left(i\right)}M_{ri}^{n}\left(x_{i}\right)^{\alpha_{ri}};
\end{equation}
Update messages:
\begin{equation}
M_{ij}^{n+1}\left(x_{j}\right)  \leftarrow  \intop_{\mathcal{X}_{i}} \exp\{\langle \theta_{ij},\phi_{ij}\rangle/\alpha_{ij}\}
\left\{\frac{b_i(x_i)}{M_{ji}^{n}(x_i)}\right\} dx_i.
\end{equation}
\end{enumerate}
\end{minipage}}
\caption{Functional Belief Propagation}
\end{center}
\end{figure}

The beliefs are normalized to integrate to 1, so they correspond to a proper density. At convergence, the beliefs $\{b_i\}$ correspond to the univariate pseudodensities $\{q_i^*\}$. The message updates can be performed in parallel. Furthermore, more elaborate schedules exist which may speed up convergence. For example, updates can be formed dynamically. Also, message updates corresponding to beliefs which have reached convergence can be skipped. See \citep{gonzalez2011parallel} for a treatment on different parallel scheduling methods. If the fixed point updates do converge, they will converge to the unique fixed point corresponding to the global minimum of \eqref{dual}. The fixed-point updates are not guaranteed to converge. In our experiments we only encountered stability issues after taking too large of a step in the ISTA algorithm for estimation, resulting in an unstable candidate step; if we encounter a convergence problem we simply take a smaller step. As such we don't find the need for damping or other techniques to encourage convergence.

To perform message passing in practice we discretized messages and approximated integrals using a Riemann sum approximation. We found this to give very accurate results in experiments. Other approximations for continuous message passing exist \citep{JMLR:noorshams13a,sudderth2010nonparametric} which could be more memory and computation efficient for large problems.

\section{Optimizing Edge Weights} \label{frankwolfe}
\begin{figure}
\centering
\includegraphics[scale=.8]{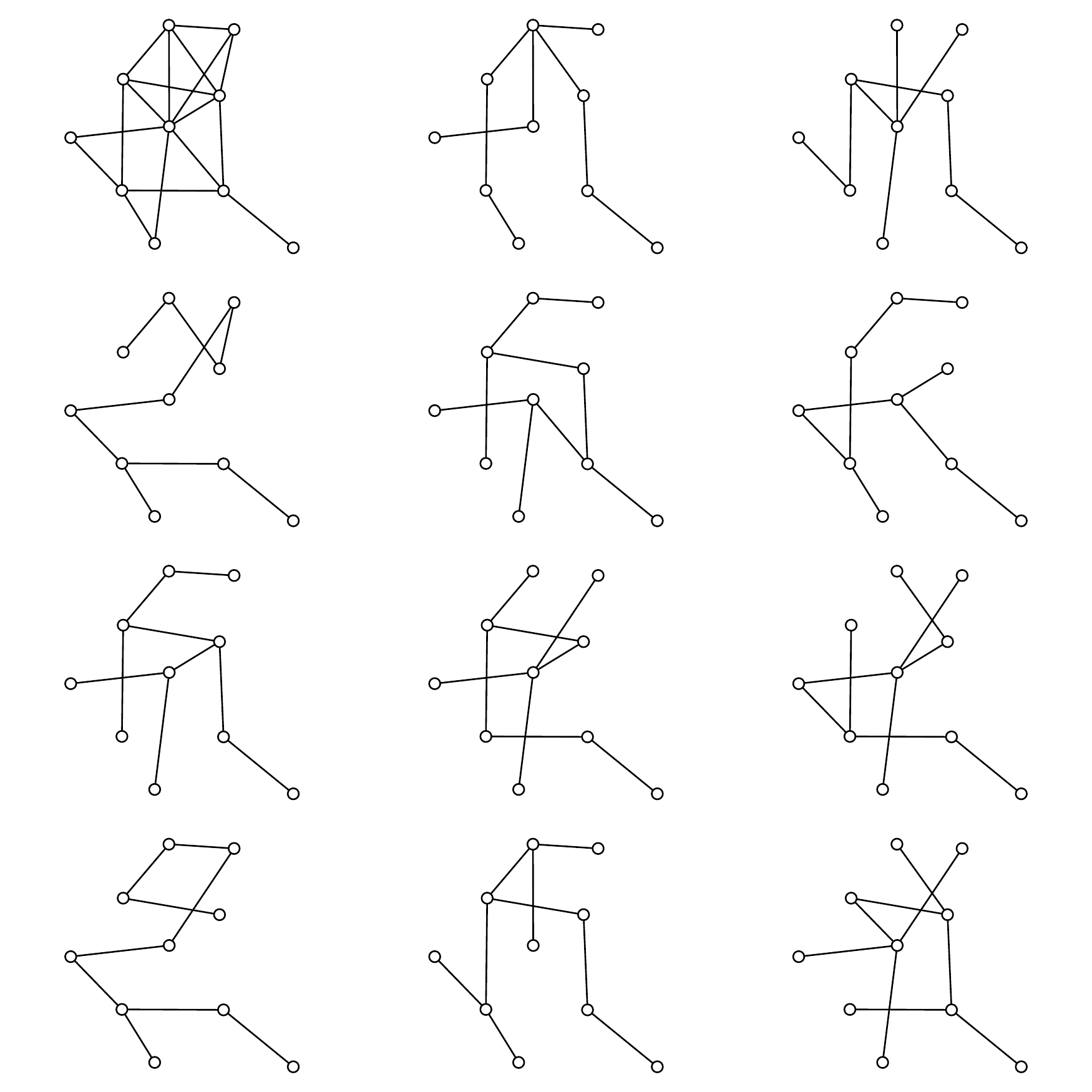}
\caption{Top left: A graph on 10 nodes. Rest: Corresponding spanning trees (not all figured).}
\end{figure}

The previous analysis outlines how to compute $Q(\theta,\alpha)$ for a set of fixed edge weights $\{\alpha_{ij}\}$. In this section we show how the edge weights can be optimized to produce tighter bounds on the likelihood by solving 
\begin{equation}
	Q(\theta):= \min_{\alpha\in\mathcal{S}_\mathcal{T}} Q(\theta,\alpha). \label{edgeweight}
\end{equation}

Our analysis follows that of \citep{wainwright2005new}. From Danskin's theorem \citep{bertsekas1999nonlinear}, observing the form of \eqref{dual} it follows that the function $Q(\theta,\alpha)$ is convex as a function of $\alpha$, with gradient
\begin{equation}
	\nabla_{\alpha_{ij}} Q(\theta,\alpha) = -I_{ij}(\tau^*(\theta)),
\end{equation}

where $\tau^*(\theta)$ are the pseudomoments from solving the dual problem \eqref{pseudomoments}. \eqref{edgeweight} is the minimization of a convex objective over a convex polytope, so it is a convex problem. However, the number of constraints characterizing $\mathcal{S}_\mathcal{T}$ are typically prohibitively large. To avoid dealing with them directly, we employ the following strategy. Suppose we have the current iterate $\alpha^t$. We linearize $Q(\theta,\alpha)$ about $\alpha^t$ and solve
\begin{eqnarray}
	& \min_s & \nabla_{\alpha} Q(\theta,\alpha)^\top (s-\alpha^t), \\
	& s.t.& s\in\mathcal{S}_\mathcal{T}. \notag
\end{eqnarray}

This is a linear program, so the solution must always fall on at least one vertex of $\mathcal{S}_\mathcal{T}$. From observing the structure of $\mathcal{S}_\mathcal{T}$ as being supported by spanning tree indicator vectors, a solution $s$ is equal to the indicator vector of a maximum weight spanning tree with weights $\{ I_{ij}(\tau^*(\theta^t))\}$. Finding the maximum weight spanning tree can be done efficiently in $O(\big| E\big| \log \big| E \big|)$ time using Kruskal's algorithm (Figure \ref{kruskal}). Lastly we update the edge weights $\alpha^{t+1}\leftarrow c\alpha^t+(1-c)s$, where $c$ is a step size $c\in (0,1)$. To ensure convergence guarantees, $c$ can either be set to $c=\frac{2}{2+t}$, or it can be chosen using line search. This technique is known as the \emph{Frank-Wolfe} algorithm \citep{bertsekas1999nonlinear}, and is known to converge at the rate of $O(1/t)$.

\begin{figure}
\begin{center}
\framebox{\begin{minipage}[t]{.95\columnwidth}%
\begin{enumerate}
	\item Input edge weights $\{w_{ij}\}$;
	\item Initialize edge set $T^0 = \emptyset$;
	\item For $k=1,\ldots,d-1:$ \\
		Find largest $w_{i^*j^*}$ such that $T^{k-1} \cup (i^*,j^*)$ doesn't form a cycle; \\
		Set $T^{k} \leftarrow T^{k-1} \cup (i^*,j^*).$
	\item Output edge set $T^{d-1}$.
\end{enumerate}
\end{minipage}}
\caption{Kruskal's Algorithm} \label{kruskal}
\end{center}
\end{figure}

\section{Variational Maximum Likelihood}

Variational regularized maximum likelihood replaces the regularized maximum likelihood equation \eqref{mlsol} with
\begin{equation} 
\underset{\theta\in\mathcal{P}}{\text{max}} \bigg\{ \langle\theta,\widehat{\mu}\rangle-Q\left(\theta\right) - \lambda \mathcal{R}(\theta_e) \bigg\} \label{approxreg},
\end{equation}
$Q\left(\theta\right)$ is the variational approximation
to the true log-partion function $Z\left(\theta\right)$.

This shares many features in common with the regularized MLE. For example, by Danskin's theorem $\nabla Q(\theta) = \tau^*(\theta)$, the optimal pseudomoments from message passing. Furthermore, $\nabla^2_{a,b}Q(\theta) = \text{cov}_{q^*}(\phi_a,\phi_b)$, and the minimality of $\phi$ implies $Q$ is strictly convex, since for any $a\not = 0$, $a^\top \nabla^2 Q(\theta) a = \text{var}_{q^*}(a^\top \phi) \not = 0$ $\nu-a.e.$.

\subsection{Optimization Algorithms}

Both the exact optimization problem in \eqref{mlsol} and the approximate problem \eqref{approxreg} can be written as minimization of a smooth (strictly convex) function plus a non-smooth convex function,
\begin{equation}
	\min_{\theta\in\mathcal{P}} \bigg\{-\mathcal{L}(\theta) + \lambda\mathcal{R}(\theta)\bigg\}.
\end{equation}
Several algorithms have been designed to solve problems of this form; for a review see \citep{bach2011convex}. We will focus on what are known as \emph{proximal gradient} methods \citep{nesterov2013gradient,beck2009fast}, a class of first-order methods which have proven effective for large scale, non-smooth optimization.
\subsubsection{ISTA}
The simplest such algorithm is called the \emph{iterative-shrinkage thresholding algorithm}, or ISTA, which works as follows. Fix the current estimate $\theta^t$. Linearize $\mathcal{L}$ about the current point and solve:
\begin{equation}
	p_L(\theta^t) = \underset{\theta\in\mathcal{P}}{\text{argmin}} \bigg\{ \langle \nabla\mathcal{L}(\theta^t) , \theta^t - \theta \rangle + \lambda\mathcal{R}(\theta) + \frac{L}{2}\Vert \theta -\theta^t\Vert^2 \bigg\},
\end{equation}
where $L$ is a step size. The squared norm term is called the proximal term, which encourages the solution not to be too far from the current step $\theta^t$. After some manipulation, is can be made equivalent to
\begin{equation}
	p_L(\theta^t)=\underset{\theta\in\mathcal{P}}{\text{argmin}}\bigg\{ \frac{1}{2}\bigg\| \theta - \bigg(\theta^t - \frac{1}{L}(-\nabla\mathcal{L}(\theta^t))\bigg)\bigg\| ^2 + \frac{\lambda}{L}\mathcal{R}(\theta)\bigg\} \label{proxupdate}.
\end{equation}

For the group regularizer $\mathcal{R}(\theta) = \sum_{(i,j)\in E} \Vert \theta_{ij} \Vert$, the solution has a closed form and is given by
\begin{eqnarray}
	(p_L(\theta^t))_i &= & \theta_i^t - \frac{1}{L}\bigg(-\nabla_{i}\mathcal{L}(\theta^t)\bigg); \\
	(p_L(\theta^t))_{ij} &= & \bigg(1-\frac{\lambda/L}{\Vert \theta_{ij}^t-\frac{1}{L}(-\nabla_{ij}\mathcal{L}(\theta^t))\Vert}\bigg)_{+}\bigg(\theta_{ij}^t-\frac{1}{L}(-\nabla_{ij}\mathcal{L}(\theta^t))\bigg),
\end{eqnarray}
and we set the update steps to $\theta^t_{i}=(p_L(\theta^t))_i$ and $\theta^t_{ij}=(p_L(\theta^t))_{ij}$ for each $i\in V$ and $(i',j)\in E$. $(\cdot)_+:=\max(\cdot,0)$. 
In the absence of regularization, the proximal gradient method simplifies to gradient descent. When $\lambda>0$, due to the soft thresholding, it may produce exactly sparse solutions, where all edge parameters $\theta_{ij}$ for a particular edge $(i,j)$ are zero. This is an advantage over other methods which only produce a sparse estimate up to numerical error, and necessitate truncation.

Since exact bounds on the Hessian of $\mathcal{L}$ aren't known, we must choose $L$ using line search. We employ the backtracking line search from \citep{beck2009fast}, in Figure \ref{istals}.

\begin{figure}
\begin{center}
\framebox{\begin{minipage}[t]{0.95\columnwidth}
\begin{enumerate}
	\item Input current iterate $\theta^t$;
	\item Fix a $L_0>0$, $\delta>1$;
	\item Find the smallest nonnegative integer $l$ such that for $L'=\delta^l L_0$,
	\begin{equation}
		\mathcal{L}(\theta^t)-\mathcal{L}(p_{L'}(\theta^t)) \leq \langle p_{L'}(\theta^t)-\theta^t , \nabla\mathcal{L}(\theta^t) \rangle + \frac{L'}{2}\Vert \theta^t - p_{L'}(\theta^t) \Vert^2,
	\end{equation}
	\item Update $\theta^{t+1}$ with \eqref{proxupdate}, using $L=\delta^{l}L_0$.
\end{enumerate}
\end{minipage}}
\caption{ISTA Line Search} \label{istals}
\end{center}
\end{figure}

\subsubsection{FISTA}

The accelerated counterpart to ISTA is the \emph{fast iterative thresholding-scaling algorithm}, or FISTA \citep{beck2009fast}. It is analogous to the accelerated gradient method in smooth optimization, which has shown to be an optimal first-order method for smooth optimization \citep{nemirovsky1983problem}. Instead of the new parameter iterates being a projection of the previous, it is a projection of a linear combination of the two previous iterates. The updates with line search are given in \ref{fistals}.

\begin{figure}
\begin{center}
\framebox{\begin{minipage}[t]{0.95\columnwidth}
\begin{enumerate}
	\item Input iterates $\theta^t,\theta^{t-1},y^t$;
	\item Fix a $L_0>0$, $\delta>1$;
	\item Find the smallest nonnegative integer $l$ such that for $L'=\delta^l L_0$,
	\begin{equation}
		\mathcal{L}(y^t)-\mathcal{L}(p_{L'}(y^t)) \leq \langle p_{L'}(y^t)-y^t , \nabla\mathcal{L}(y^t) \rangle + \frac{L'}{2}\Vert y^t - p_{L'}(y^t) \Vert^2,
	\end{equation}
	\item Set $a_{t+1} = \frac{1+\sqrt{1+4a_t^2}}{2}$,
	\item Update $\theta^{t}=p_L(y^t)$, using $L=\delta^{l}L_0$,
	\item Update $y^{t+1}=\theta^{t}+\frac{a_t-1}{a_{t+1}}(\theta^t-\theta^{t-1})$.
\end{enumerate}

\end{minipage}}
\caption{FISTA Line Search} \label{fistals}
\end{center}
\end{figure}

FISTA requires essentially the same computation at each iteration, in particular the same number of gradient evaluations. In addition to the standard proximal gradient algorithms, we found success initializing $L_0$ using a secant rule:
\begin{equation}
	L_0 = \frac{\langle \theta^t-\theta^{t-1},-\nabla\mathcal{L}(\theta^t)+\nabla\mathcal{L}(\theta^{t-1})\rangle}{\Vert \theta^{t} - \theta^{t-1}\Vert^2}.
\end{equation}
Typically we find initializing with the secant rule finds a direction of sufficient descent with little backtracking.
\subsubsection{Discussion}

ISTA and its accelerated counterpart FISTA \citep{beck2009fast} have linear convergence, that is at the rate $O(C^t)$ for some $C<1$, when $-\mathcal{L}$ is strongly convex. In contrast, if $-\mathcal{L}$ is only Lipschitz, ISTA converges at the sub-linear rate $O(1/k)$, and FISTA at the rate $O(1/k^2)$. Recall that with strong convexity, gradient descent also has linear convergence, so proximal gradient descent behaves like gradient descent despite the objective not being smooth. Thus proximal gradient methods have superior theoretical guarantees to competitors such as subgradient descent. The negative log-likelihood $-\mathcal{L}$ is only strictly convex, but it is strongly convex in a neighborhood of the solution \citep{kakade2010learning}. Thus we may think of these proximal gradient methods converging linearly after a sufficient "burn-in" phase.

FISTA, like the accelerated gradient algorithm has been shown to outperform ISTA in some real problems \citep{beck2009fast}. However, it does have some disadvantages. It is not guaranteed to decrease the objective after each iterate. Further, ISTA may converge rapidly when well-initialized. In our application, this will commonly happen, because parameters are estimated over a range of $\lambda$, each solution used as a warm start for the next.

\subsection{Choosing Tuning Parameters}\label{warmstart}
As of yet we have not discussed how to practically choose the truncation parameters $m_1,m_2$ and the regularization parameter $\lambda$. We suppose the existence of a held-out tuning set; in the absence, one may use cross-validation. We choose $m_1,m_2,\lambda$ to minimize the negative log-likelihood risk in the held out set. To save on computation, we use the idea of \emph{warm starts} which we detail in the sequel.
First, observe that the first-order necessary conditions for the regularized MLE are:
\begin{align}
	\widehat{\mu}_{ij} - \mu_{ij}(\widehat{\theta}_\lambda) - \lambda \widehat{Z}_{ij} &= 0,  &i\in V; \\
	\widehat{\mu}_{i}-\mu_{i}(\widehat{\theta}_\lambda) &= 0,  &(i,j)\in E. \label{foncvert}
\end{align}
where $\widehat{Z}$ denotes the sub gradient of the regularizer $\mathcal{R}$, at $\widehat{\theta}$, which is
\begin{equation}
	\widehat{Z}_{ij} = \begin{cases}
		\{x:\Vert x\Vert\leq 1\}, & \Vert\theta_{ij}\Vert = 0, \\
		\frac{\theta_{ij}}{\Vert\theta_{ij}\Vert}, & o/w. 	
	\end{cases}
\end{equation}

$\widehat{\theta}_{ij}=0$ when
\begin{equation}
	\lambda \geq \Vert\widehat{\mu}_{ij}-\mu_{ij}(\widehat{\theta}_\lambda)\Vert.
\end{equation}

When $\widehat{\theta}_{ij}=0$ for each $(i,j)\in E$, it's clear that $\mu_{ij}=\mu_i\mu_j=\widehat{\mu}_i\widehat{\mu}_j$ by independence and the moment-matching condition \eqref{foncvert}. This allows us to choose an upper bound $\lambda_{max}$ such that the solution will have no nonzero edge parameters:
\begin{equation}
	\lambda_{start} \geq \max_{(i,j)\in E}\Vert\widehat{\mu}_{ij}-\widehat{\mu}_i\widehat{\mu}_j\Vert.
\end{equation}

The idea behind warm starting is the following: we begin by estimating $\widehat{\theta}_{\lambda_{start}}$, which amounts to $d$ univariate density estimation problems which can be performed in parallel. Then we fit our model on a path of $\lambda$ decreasing from $\lambda_{start}$, initializing each new problem with the previous solution $\widehat{\theta}_{\lambda}$. The solution path for the regularized MLE is smooth as a function of $\lambda$, suggesting nearby choices of $\lambda$ will provide values of $\widehat{\theta}$ which are close to one another.
 
We can also incorporate warm-starting in choosing $m_1,m_2$. For a given $\lambda$, we first estimate the model for first-order polynomials, corresponding to $m_1=m_2=1$. We then increment the truncation parameters by increasing the degree of the polynomial of he sufficient statistics. We augment the previous parameter estimate vector with zeros in the place of the added parameters, and warm start \emph{ISTA} from this vector.

\section{Experiments}\label{expersec}

Our simulations were conducted on a workstation operating 23 Intel(R) Xeon(R) CPU E5-2420 1.90GHz processors using $\tt{R}$ with backend computations written in $\tt{C++}$ and compiled using the $\tt{Rcpp}$ package \citep{Eddelbuettel:2011aa}. Calculations, including message passing were parallelized using the Threading Building Blocks $\tt{C++}$ library.
We discretized messages uniformly over $[0,1]$ on a grid of 128 points, and bivariate pseudodensities approximated on a $128 \times 128$ grid. All integrals were approximated using Riemann sums over these discretizations. We fit our model using the ISTA algorithm, stopping after an objective improvement of less than $10^{-4}$ or 1000 iterations. We selected the tuning parameters as described in Section \ref{warmstart} using warm starts. To choose the edge weights, we generate a series of random spanning trees and take the average edge appearance probability as the edge weight. This produces a valid vector in the spanning tree polytope. We conducted experiments optimizing edge weights by the algorithm in Section \ref{frankwolfe}, but found the risk improvement in our simulations to be small relative to the additional computational requirement.
\begin{figure}
\begin{center}
\includegraphics[scale=.65]{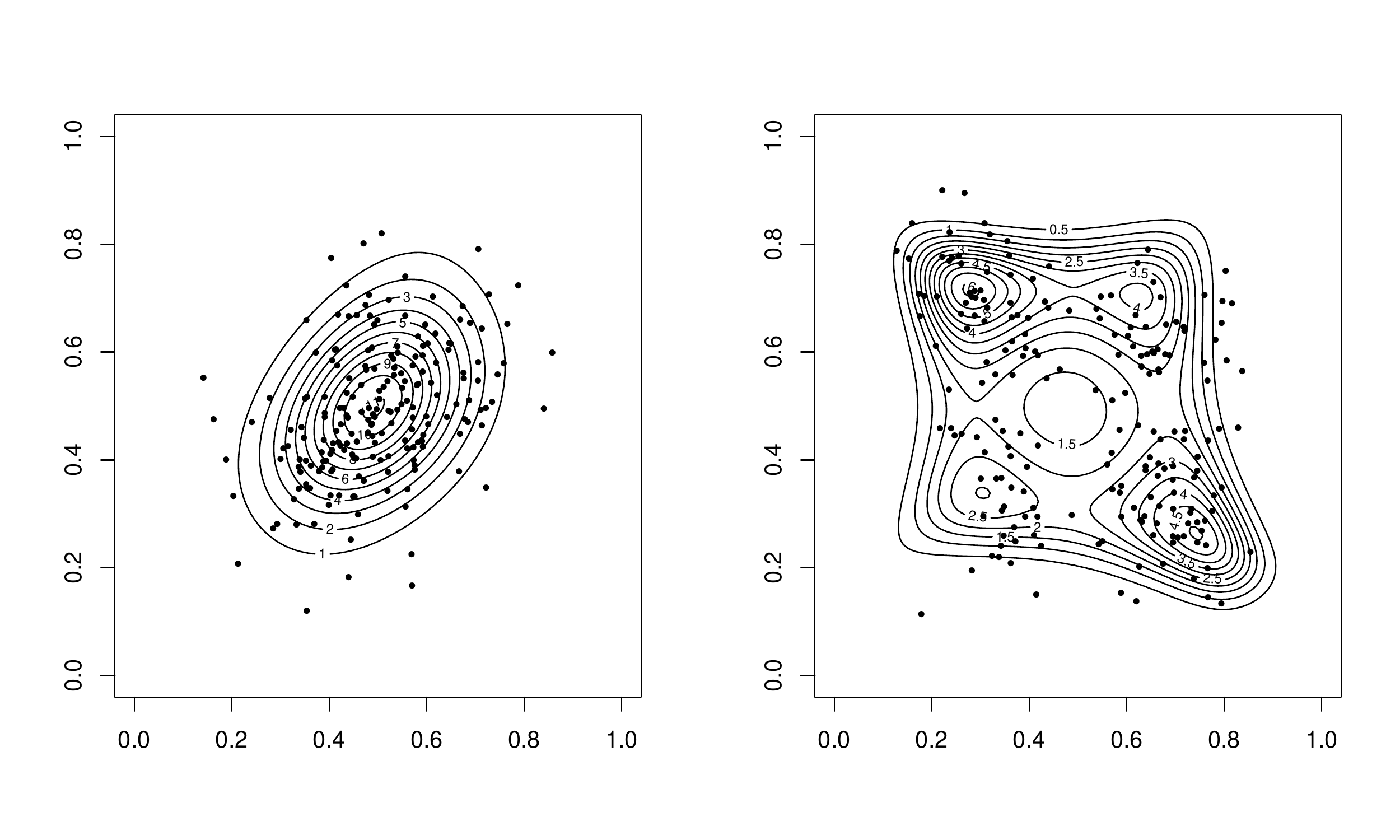}
\caption{Bivariate pseudo density contours estimated from high-dimensional synthetic data. (a) Gaussian data; (b) non-Gaussian (copula) data. \label{figmarg}}
\end{center}
\end{figure}
We generated three types of data. First we generate independent, identically distributed random Gaussian vectors each with mean $\{0.5,\ldots,0.5\}$ and covariance $\Sigma$. $\Sigma$ is scaled to have diagonal $1/8^2$ and sparse off-diagonals. The sparsity pattern was generated by randomly including edges with probability $2/d$, so the expected number of edges is $d-1$. Furthermore, we generate non-Gaussian data by generating Gaussian data by the aformentioned procedure and marginally applying the transformation $y=\text{sign}(x-0.5) \mid x - 0.5 \mid ^{0.6}/5 + 0.5$. That is, the transformed data is distributed as a Gaussian copula. These two models follow a pairwise factorization. Finally, we generated data as a mixture of three non-Gaussian distributions, each with edge structure of a randomly generated spanning tree. The distributions are given equal mixing weights. The resulting mixture of trees is not a pairwise distribution.
\subsection{Risk Paths}
We first examine the risk paths, varying $\lambda$ and hence the number of included edges for typical runs of our simulation. We compare the TRW and it's relaxed version (refitting the model under the selected sparsity constraint and setting $\lambda=0$) to the graphical lasso (using the glasso $\tt{R}$ package) and its relaxed counterpart. Examples of estimated pseudodensities are shown in \ref{figmarg}. For Gaussian data, we see that TRW performs slightly worse than glasso for sparse graphs, while the gap widens as the graph becomes denser. There are two clear reasons for this. Since the Gaussian model is correct, glasso automatically provides a correctly specified model. TRW must select the number of basis elements and so may include too many (or few) parameters. Furthermore, the variational bound from TRW worsens as the graph becomes denser.
\begin{figure}
\centering
\includegraphics[scale=1]{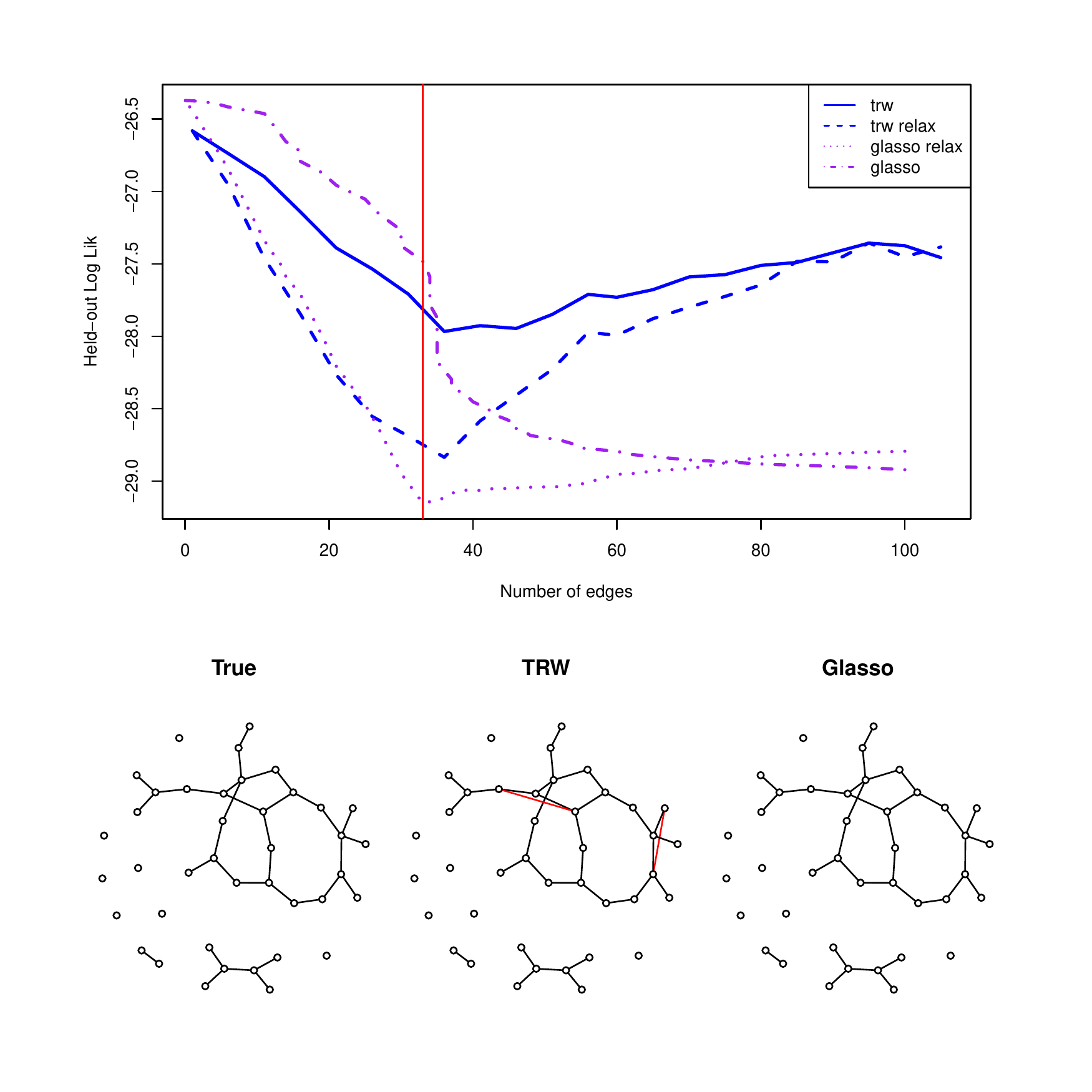}
\caption{Risk path on simulated Gaussian data. Top row is negative log-likelihood risk on held-out data; bottom row are selected graphs. Red edges are false inclusions; gray edges are false omissions.}
\end{figure}
For the non-Gaussian data, TRW clearly outperforms the glasso. The risk paths between TRW and glasso look markedly different. TRW demands many more parameters to get the optimal fit, so the performance from the relaxed TRW suffers, while the regularized version benefits by reducing overfitting, thus the relaxed and regular versions have very similar risk.
As is well-known, model selection using held-out risk is precarious as the risk path is often flat. In our simulations TRW did quite well. Both glasso and TRW included two false edges for the Gaussian simulation, while TRW omitted several edges. For the non-Gaussian data, glasso included several more false egdes than did TRW.
\subsection{Density Estimation}
We continue by comparing held-out risk estimates to several other high-dimensional density estimation algorithms in Table \ref{simtable}: glasso, spherical Gaussian mixture model (with number of components chosen by BIC, \citep{fraley2002model}), and the Kernel maximum spanning tree estimator of \citep{forest2011}. We consider sparse Gaussian, sparse Gaussian copula and tree mixture data for dimensions between 30 to 120. For all simulations we set $n=100$ and hold out 300 observations for testing. We repeat the simulations (but keeping the generating distribution fixed) 5 times and report the standard errors in parentheses.
 For Gaussian data, TRW performs competitively. It dominates the forest and spherical mixture density estimators but is slightly worse than glasso. For the Gaussian copula data TRW outperforms other methods; TRW can capture both non-Gaussianity and the cyclical dependence structure, while glasso can only capture the latter, and the forest density estimator the former. For mixtures of trees the results are similar to the copula simulation, with TRW dominating the other methods, despite the data not following a pairwise factorization.

\begin{landscape}
\begin{table}
\centering

\begin{tabular}{lrrrrrrr}
  \hline
   & d & Glasso & Glasso Refit & TRW & TRW Refit & Forest & Gaussian Mixture \\
  \hline
 Gaussian & 30 & -21.003 (0.175) & -21.22 (0.318) & -20.544 (0.107) & -20.783 (0.244) & -18.969 (0.247) & -19.716 (0.16) \\
   & 50 & -34.9 (0.185) & -35.136 (0.292) & -34.191 (0.254) & -34.675 (0.258) & -31.242 (0.349) & -32.697 (0.438) \\
   & 100 & -67.944 (0.19) & -67.698 (0.516) & -67.237 (0.263) & -67.721 (0.334) & -59.91 (0.701) & -65.418 (0.446) \\
   & 120 & -81.573 (0.531) & -81.441 (0.916) & -80.911 (0.498) & -81.35 (0.649) & -72.763 (1.301) & -78.907 (0.393) \\ \\
   Copula & 30 & -0.453 (0.191) & -0.508 (0.373) & -3.343 (0.275) & -2.85 (0.236) & -2.128 (0.243) & 1.04 (0.16) \\
    & 50 & -0.1 (0.204) & -0.193 (0.425) & -5.559 (0.279) & -4.817 (0.253) & -3.286 (0.401) & 1.802 (0.163) \\
   & 100 & 0.431 (0.295) & 0.54 (0.253) & -10.959 (0.393) & -9.91 (0.459) & -5.359 (0.261) & 3.396 (0.557) \\
    & 120 & 2.398 (0.243) & 2.75 (0.445) & -12.111 (0.434) & -11.283 (0.388) & -4.738 (0.427) & 4.313 (0.182) \\ \\
   Mixture of Trees & 30 & 0.406 (0.108) & 0.429 (0.142) & -3.215 (0.293) & -2.899 (0.29) & -1.589 (0.27) & 1.024 (0.098) \\
    & 50 & 0.67 (0.164) & 0.889 (0.306) & -5.217 (0.287) & -4.777 (0.242) & -2.415 (0.444) & 1.676 (0.222) \\
    & 100 & 2.5 (0.179) & 2.675 (0.251) & -10.461 (0.141) & -9.805 (0.228) & -3.313 (0.365) & 3.342 (0.244) \\
    & 120 & 2.837 (0.541) & 2.855 (0.499) & -12.365 (0.382) & -11.534 (0.459) & -4.382 (0.408) & 4.153 (0.534) \\
 \hline
\end{tabular}
\caption{Held-out risk estimates for synthetic data. Standard deviations for 5 replicates in parenthesis. \label{simtable}}
\end{table}

\end{landscape}
\begin{figure}
\centering
\includegraphics[scale=1]{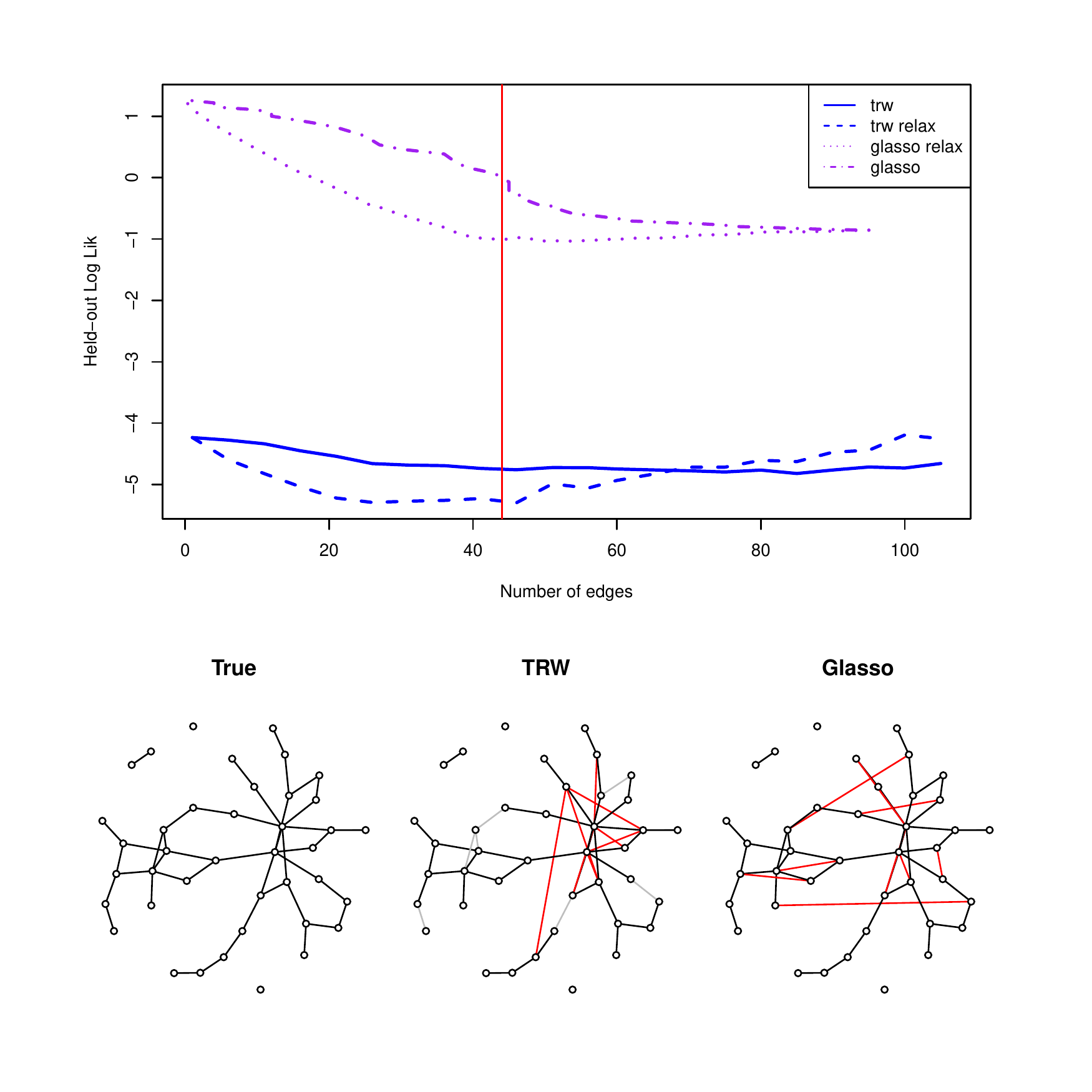}
\caption{Performance on simulated Gaussian copula data. Top row is negative log-likelihood risk on held-out data; bottom row are selected graphs. Red edges are false inclusions; gray edges are false omissions.}
\end{figure}

\subsection{ROC Curves}

Figures \ref{roctrw1} and \ref{roctrw2} display ROC curves for two types of experiments. We generate non-Gaussian data as before, with $n=d=30$. The curves trace the true negative and true positive percentages of the algorithms, varying the regularization parameter $\lambda$. We choose $m_1,m_2$ so that the resulting curve has the largest value of $\max_{\lambda}TP(\lambda)+TN(\lambda)$. The plotted curves are an average of 20 repetitions. In Figure \ref{roctrw1} the edges are chosen to be a randomly-generated spanning tree; in Figure \ref{roctrw2} the edges are included with equal probability $2/d$, which we call the ER graph. We compare the TRW estimator to the SKEPTIC estimator from \citep{liu2012high}. The SKEPTIC estimator was particularly devised for estimating Gaussian copula graphical models, while our estimator is designed for a superset of those models. 

Overall the SKEPTIC performs better in terms of area under the curve (AUC). It also generally has a better TP\% for moderate or small values of TN\%. However the TRW estimator still performs quite well in these two metrics. For large TN\% the TRW estimator manages a better TP\%. We can only speculate on why the TRW estimate performs better here, but it may be because when the selected graph is very sparse it has few or no cycles, and the tree-based approximation of TRW is more powerful. This is consistent with the experiments, as the phenomenon is more pronounced for the tree simulation than the loopy graph simulation.
\begin{figure}
\centering
\includegraphics[scale=.5]{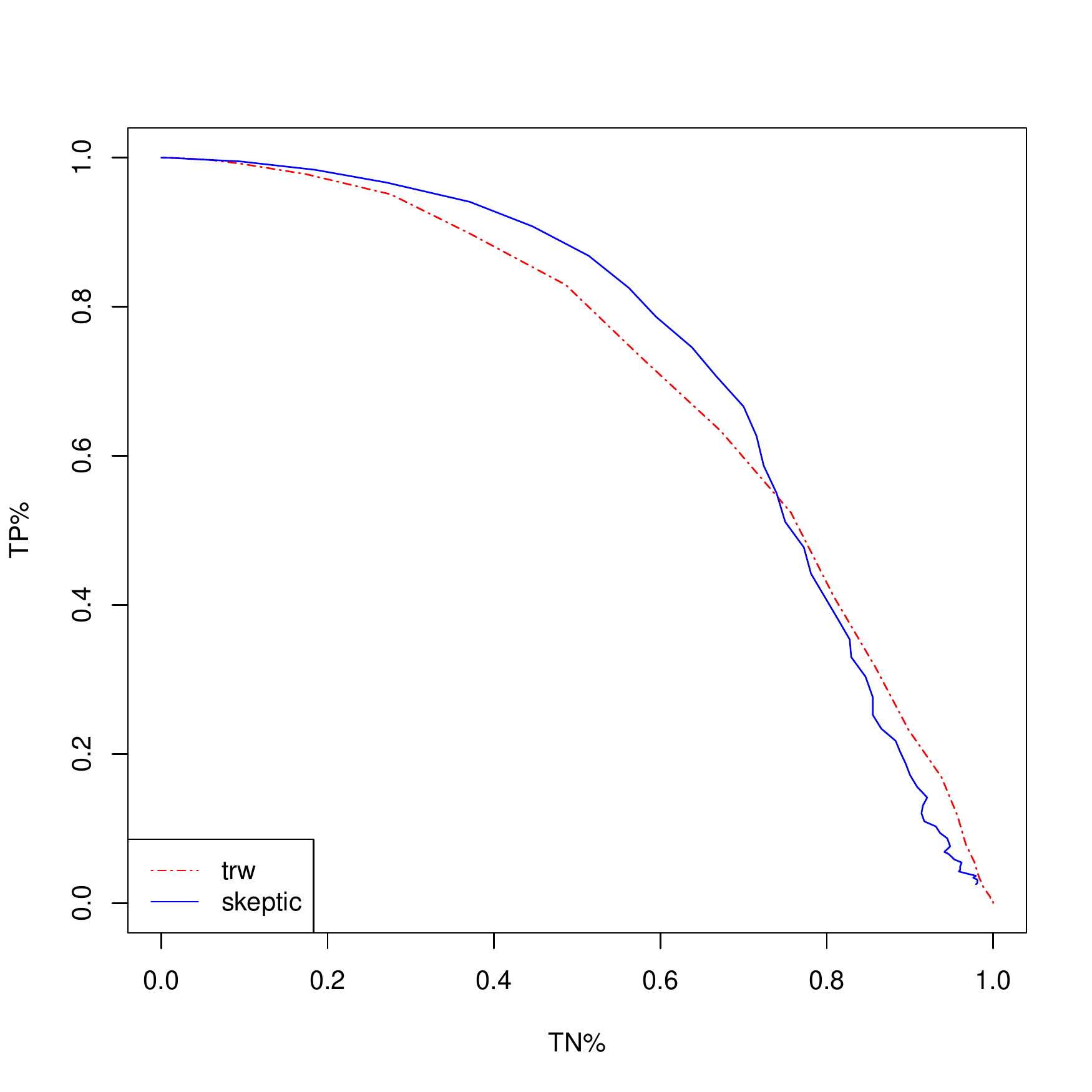}
\caption{ROC Curve, tree graph, n=d=30. Solid line: SKEPTIC estimator. Dotted line: TRW.} \label{roctrw1}
\end{figure}

\begin{figure} 
\centering
\includegraphics[scale=.5]{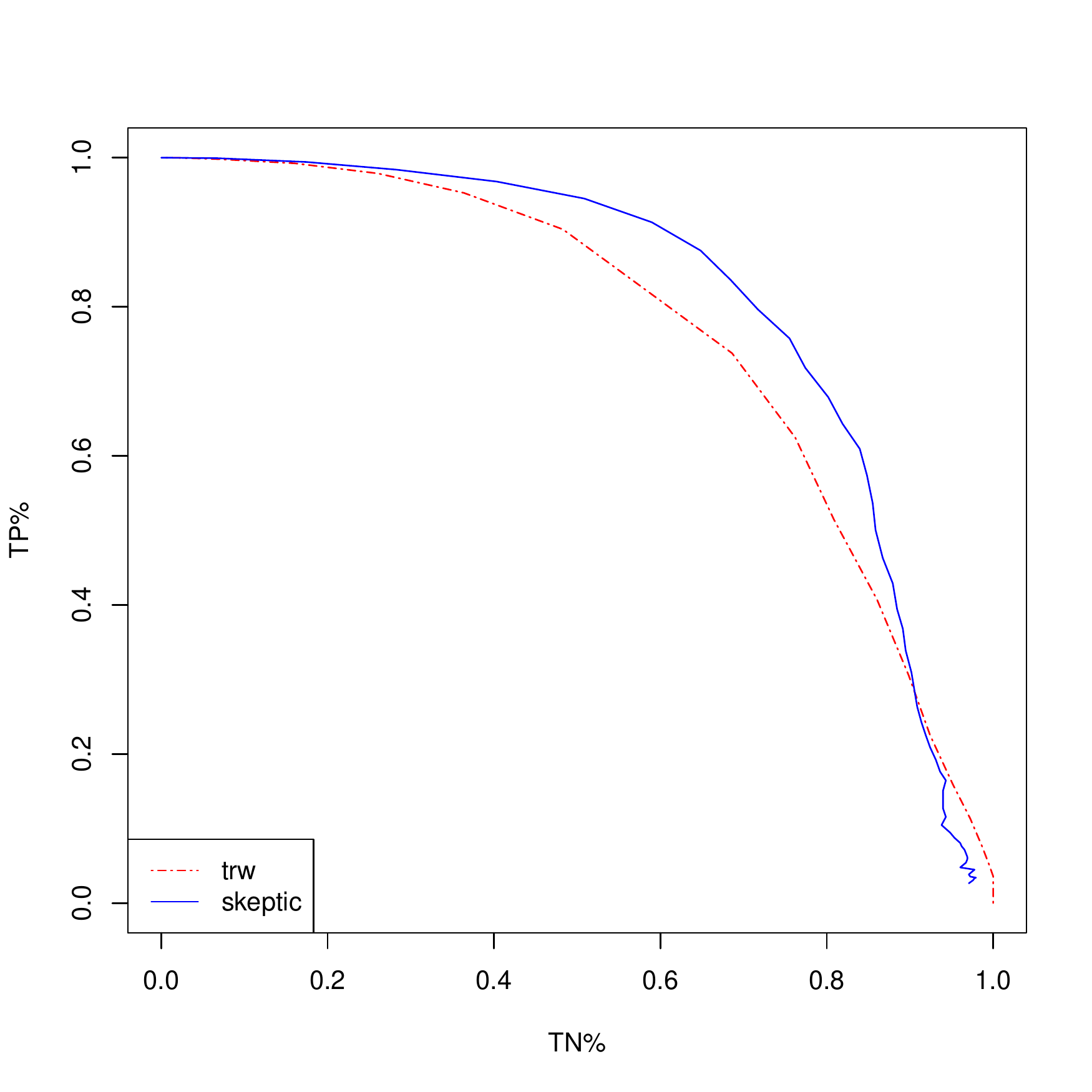}
\caption{ROC Curve, ER graph, n=d=30. Solid line: SKEPTIC estimator. Dotted line: TRW.} \label{roctrw2}
\end{figure}

\subsection{MEG Data}
Magnetoencephalography or MEG is a neuroimaging technique for mapping brain activity using electrical currents in the brain. The resulting signals are high-frequency and have a complex non-linear relation to one another. There has been interest in using various neuroimaging techniques for mapping regional brain networks \citep{kramer2011emergence}, and particularly in understanding differences in connectivity related to neurodegenerative diseases \citep{stam2010use}. We explore this on the MEG data from \citep{vigario1998independent}, which contains measurements from 122 sensors. We scale the data to be contained in the unit cube and remove large outliers (marginally larger than 6 standard deviations). Two features of this data are the temporal dependence of the signals and the presence of artifacts. Since our main motivation is graph estimation we will not address these issues, besides restricting our attention to a small timespan of the dataset. We use the first 400 observations in the series, randomly assigning 100  as training data, and the other 300 as test data.
\begin{figure}
\centering
\includegraphics[scale=.85]{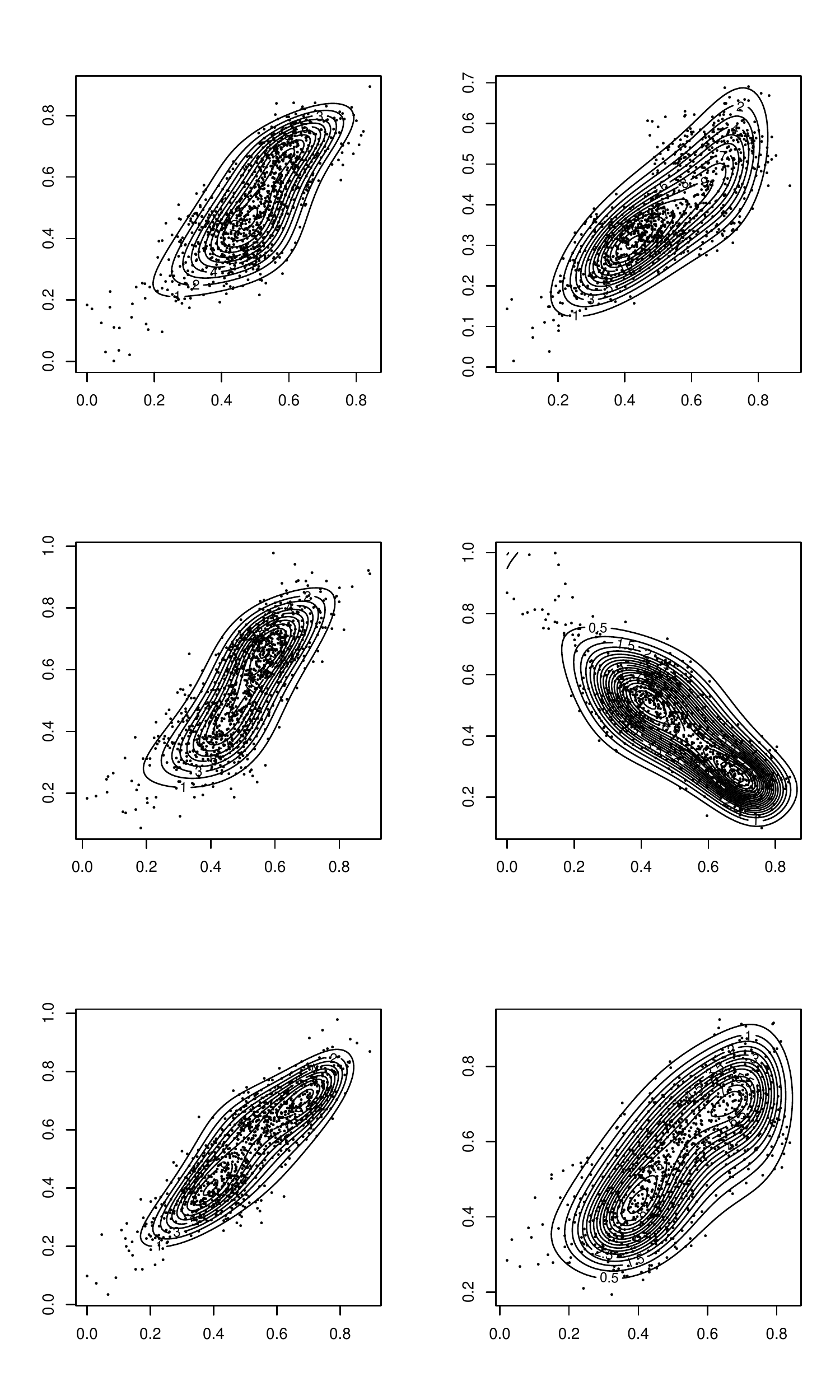}
\caption{Six estimated pseudodensities plotted with 1000 observations of the MEG data. \label{megplots}}
\end{figure}
\begin{figure}
\centering
\includegraphics[scale=1.05]{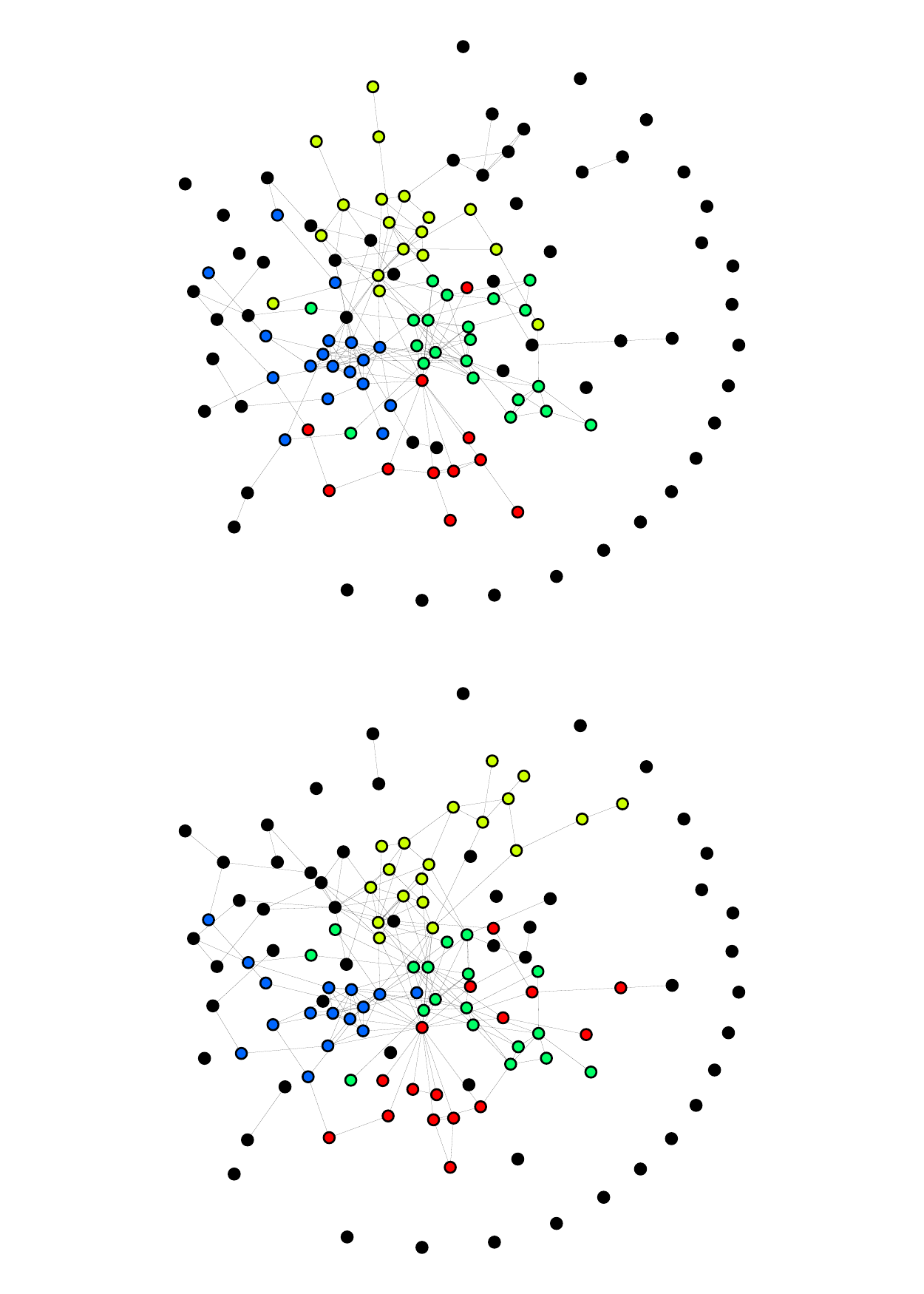}
\caption{Estimated graphs from MEG data. Top: Graph estimated from TRW ; Bottom: Graph learned from glasso} \label{megplot}
\end{figure}

Inspecting two-dimensional projections of the data and their corresponding pseudodensities in Figure \ref{megplots}, the data displays clearly non-linear and multi-modal behavior which TRW can capture, but glasso cannot.
We compare the glasso and TRW methods for graph estimation. We select a graph to minimize the held-out risk for the relaxed TRW, which has 199 edges. We then estimate the graph using the graphical lasso, setting the regularizaition parameter to include the same number of edges. The estimated graphs are shown in figure \ref{megplot}. For clarity, we color code vertices from the top four clusters produced from running the learned graphs through the community detection algorithm of \citep{newman2006finding}. Note that the position of vertices here does not correspond to the actual location of the sensors on the scalp. Comparing the graphs from the minimum risk estimators, the two graphs share many features in common, each having one large connected component containing several smaller densely connected communities. However, they have clear differences, disagreeing on 66 edges, or one-third of the edges in each respective graph. We believe due to the complex nature of the observed signals in imaging data such as MEG, our method may be able to bring new and better insights to understanding functional connectivity of brain networks.

\section{Discussion}

In this section we detail a tree-reweighted variational approximation for continuous-valued exponential families, which we apply to the exponential series estimator of chapter 1. Evaluating the variational likelihood involves message passing which can be effectively parallelized for high-dimensional problems, and provides a lower-bound to the likelihood (and upper bound to risk). We describe a proximal gradient algorithm for estimating the regularized MLE. Our experiments show this approach has very attractive performance in both risk and model selection performance, compared to other methods in the literature. We also demonstrate our method on a data set of MEG signals.

\chapter{Regularized Score Matching}

\section{Introduction}

Undirected graphical models are an invaluable class of statistical models. They have been used successfully in fields as diverse as biology, natural language processing, statistical physics and spatial statistics. The key advantage of undirected graphical models is that its joint density may be factored according to the cliques of a graph corresponding to the conditional dependencies of the underlying variables. The go-to approach for statistical estimation is the method of maximum likelihood (MLE). Unfortunately, with few exceptions, MLE is intractable for high-dimensional graphical models, as it requires computation of the normalizing constant of the joint density, which is a $d$-fold convolution. Even exponential family graphical models \citep{wainwright2008graphical}, which are the most popular class of parametric models, are generally non-normalizable, with a notable exception being the Gaussian graphical model. Thus the MLE must be approximated.  State of-the-art methods for graphical structure learning avoid this problem by performing neighborhood selection \citep{yang2012graphical,meinshausen2006high,ravikumar2010high}. However, this approach only works for special types of pairwise graphical models whose conditional distributions form GLMs. Furthermore, these procedures do not by themselves produce parameter estimates.

In this chapter we demonstrate a powerful new method for graph structure learning and parameter estimation based on minimizing the regularized Hyv\"arinen score of the data. It works for any continuous pairwise exponential family, as long as it follows some weak smoothness and tail conditions. Our method allows for multiple parameters per vertex/edge. We prove high-dimensional model selection and parameter consistency results, which adapt to the underlying sparsity of the natural parameters. As a special case, we derive a new method for estimating sparse precision matrices with very competitive estimation and graph learning performance. We also consider how our method can be used to do model selection for the general nonparametric pairwise model by choosing the sufficient statistics to be basis elements with degree growing with the sample size. We show our method can be expressed as a second-order cone program, and which we provide highly scalable algorithms based on ADMM and coordinate-wise descent.
\section{Background}

\subsection{Graphical Models}
Suppose $X=(X_1,\ldots,X_d)$ is a random vector with each entry having support $\mathcal{X}_i$, $i=1,\ldots,d$. Let $G=(V,E)$ be an undirected graph on $d$ vertices corresponding to the elements of $X$. An undirected graphical model or Markov random field is the set of distributions which satisfy the Markov property or condition independence with respect to $G$. From the Hammersley-Clifford theorem, if $X$ is Markov with respect to $G$, the density of $X$, $p$ can be decomposed as
\begin{align}
	p(x) \propto \exp\left\{\sum_{c\in \text{cl}(G)} \psi_c(x_c)\right\},
\end{align}
where $\text{cl}(G)$ is the collection of cliques of $G$. The pairwise graphical model supposes the density can be further factored according to the edges of $G$,
\begin{align}
	p(x) \propto \exp\left\{\sum_{i,j\in V,i\leq j}\psi_{ij}(x_i,x_j)\right\}.
\end{align}


For a pairwise exponential family, we parametrize $\psi_i,\psi_{ij}$ by
\begin{align}
	\psi_{ii}(x_i) &:= \sum_{u\leq m} \theta_i^u \phi_i^u(x_i), & i\in V, \\
	\psi_{ij}(x_i,x_j) &:= \sum_{u\leq m} \theta_{ij}^u \phi_{ij}^u(x_i,x_j), & (i,j)\in E.
\end{align}
Here $m$ denotes the maximum number of statistics per edge or vertex. We denote $\theta$ to be the vectorization of the parameters, $\theta:=(\theta_{11}^\top,\ldots,\theta_{d1}^\top,\theta_{22}^\top,\ldots,\theta_{dd}^\top)^\top$.
%

\subsection{Scoring Rules}
A \emph{scoring rule} \citep{dawid2005geometry} $S(x,Q)$ is a function which measures the predictive accuracy of a distribution $Q$ on an observation $x$. A scoring rule is \emph{proper} if $\mathbb{E}_p[S(X,Q)]$ is uniquely minimized at $Q=P$. When $Q$ has a density $q$, we equivalently denote the scoring rule $S(X,q)$. A \emph{local} scoring rule only depends on $q$ through its evaluation at the observation $x$. A proper scoring rule induces an entropy
\begin{align}
	H(p) = \mathbb{E}_{p}\left[S(X,p)\right],
\end{align}
as well as a divergence
\begin{align}
	D(p,q) = \mathbb{E}_p \left[ S(X,q) - S(X,p)\right].
\end{align}
An optimal score estimator is an estimator which minimizes the empirical score 
\begin{align}
\frac{1}{n}\sum_{r=1}^n S(X^r,q),
\end{align}
 over some class of densities.
%
%
%
%

\begin{exmp}
The \emph{log score} takes the form $l(x,q):=-\log q(x)$. The corresponding entropy is the Shannon entropy $H(p):=-\mathbb{E}_p\left[\log p\right]$, its corresponding divergence is the \emph{Kullback-Leibler Divergence} $\text{KL}(p \mid q) =\mathbb{E}_{X\sim p}\left[\log\frac{p(X)}{q(X)}\right]$ and the optimal score estimator is the maximum likelihood estimator. It is a proper and local scoring rule. This scoring rule was implemented in Chapter 2.
\end{exmp}

\begin{exmp}
Consider the \emph{Bregman score},
\begin{align}
b(x,q):=-g'(q(x))-\intop \rho(dy)\left(g(q(y))-q(y)g'(q(y))\right), \label{bregman}
\end{align}
where $g:\mathbb{R}^+\rightarrow\mathbb{R}$ is a convex, differentiable function and $\rho$ is some baseline measure. The corresponding entropy is $H(p)=-\intop \rho(dy)g(p)$ and divergence 
\begin{align}
D(p,q)=\intop \rho(dy)\left(g(p)-\left(g(q)+g'(q)(p-q)\right)\right).
\end{align}
When $g(x)=\log(x)$, after removing the constant term, $b$ has the form
\begin{align}
	b(x,q)& = -\frac{1}{q(x)} -\intop \rho(dy) \log(q(y)) \\
	&= -e^{-f(x)} - \intop \rho(dy) f(y),
\end{align}
where $f=\log q$. This is a proper scoring rule \citep{dawid2014theory}, but it is not local because it depends on values of $q$ besides the observation $x$. An estimation procedure for nonparametric graphical models using smoothing splines was based on this scoring rule in \citep{jeon2006effective}.
\end{exmp}
\subsection{Hyv{\"a}rinen Score}
Consider densities $q$ which are twice continuously differentiable over $\mathcal{X}=\mathbb{R}^d$ and satisfy
\begin{align}
	& \Vert p(x) \nabla \log q(x)\Vert\rightarrow 0, \text{ for all } \Vert x \Vert \rightarrow \infty. \label{boundary}
\end{align}	
where $X\sim p$. Consider the scoring rule
\begin{align}
	h(x,q) &=\frac{1}{2}\Vert\nabla \log q(x)\Vert^2_2 + \Delta \log q(x),
\end{align}
where $\nabla$ denotes the gradient operator and  $\Delta$ is the operator 
\begin{align}
	\Delta \phi(x) = \sum_{i\in V} \frac{\partial^2\phi(x)}{\partial x_i^2}.
\end{align}
This is a proper and local scoring rule \citep{parry2012proper}. Using integration by parts, it can be shown it induces the \emph{Fisher divergence}:
\begin{align}
	\text{F}(p \mid q) = \mathbb{E}_{X\sim p}\left[ \bigg\Vert \nabla \log \frac{p(X)}{q(X)} \bigg\Vert_2^2 \right].
\end{align}

The optimal score estimator is called the \emph{score matching} estimator \citep{hyvarinen2005estimation,hyvarinen2007some}. The Hyv{\"a}rinen score is homogeneous in $q$ \citep{parry2012proper}, so that it does not depend on the normalizing constant of $q$, which for multivariate exponential families is typically intractable. Second, for natural exponential families the objective of the optimal score estimator is quadratic, so the estimating equations corresponding to score matching are linear in the natural parameters \citep{forbes2014linear}. Maximum likelihood for exponential families generally involves a complex mapping from the sufficient statistics of the data to the natural parameters \citep{wainwright2008graphical,brown1986fundamentals}, necessitating specialized solvers.

\subsection{Score Matching for Exponential Families} \label{smef}
%
%
%
%
%
%

For a pairwise density, define $\phi_{\cdot,i}=\left(\phi_{1i}^\top,\ldots,\phi_{di}^\top\right)^\top$. For $i\in V$, denote $a_i(x) := \frac{\partial}{\partial x_i} \phi_{\cdot,i}$ and
\begin{align}
	(K(x))_{\cdot,i}:= \frac{\partial^2\phi_{\cdot,i}}{\partial x_i^2}.
\end{align}
Taking derivatives,
\begin{align}
	\frac{\partial}{\partial x_i} \langle \phi(x), \theta\rangle & =  \left\langle \frac{\partial\phi_{\cdot,i}}{\partial x_i},\theta_{\cdot,i}\right\rangle ,
\end{align}
thus $h$ takes the form
\begin{align}
	h(x,\theta)
	&= \sum_{i\in V} \left( \frac{1}{2}\theta_{\cdot,i}^\top a_i(x)a_i(x)^\top\theta_{\cdot,i} +K_{\cdot,i}(x)^\top\theta_{\cdot,i}\right).
\end{align}
$h$ is a sum of $d$ positive-semidefinite quadratic forms, so it is also psd quadratic. Alternatively, we may write $h(x,\theta)=\theta^\top A(x)\theta + K(x)^\top \theta$, where $A(x)$ is a psd matrix with at most $2md$ non-zero entries per row, and $K(x)$ is a vector with $K_{ij}=\frac{\partial^2 \phi_{ij}}{\partial x_i ^2}+ 1\{i\not = j\}\frac{\partial^2 \phi_{ij}}{\partial x_j ^2}$. If we write $\tilde{\theta}=\left(\theta_{\cdot,1}^\top,\theta_{\cdot,2}^\top,\ldots,\theta_{\cdot,d}^\top\right)^\top$, where $\tilde{\theta}_{ij}=\tilde{\theta}_{ji}$, we may write the scoring rule as
\begin{align}
	h(x,\tilde{\theta}) = \frac{1}{2}\tilde{\theta}^\top \tilde{A}(x)\tilde{\theta} + \tilde{K}(x)^\top\tilde{\theta},\label{pairscore}
\end{align}
where $\tilde{A}(x)$ is a block-diagonal matrix,
\begin{align}
	\tilde{A}(x)= \left(\begin{array}{cccc}
		a_1(x)a_1(x)^\top &&&\\
		& a_2(x)a_2(x)^\top & & \\
		&& \ddots & \\
		&&& a_d(x)a_d(x)^\top
	\end{array}\right)
\end{align}
and $\tilde{K}=(K_{\cdot,1}^\top,\ldots,K_{\cdot,d}^\top)^\top$. We will alternate between these two equivalent representations of $h$ based on convenience.

%
%
%
%
%
%

\begin{remark}[Bounded supports]

From the differentiability assumption we see that our derivations do not generally apply when $\mathcal{X}_i$ is a half-bounded or bounded support as the density may not be differentiable at the boundary. However, in \citep{hyvarinen2007some} a proper scoring rule was derived for half-bounded supports, which may be shown to have the same form as \eqref{pairscore}, after modifying slightly the formulas for $a_i(x),K(x)$. Here we derive a similar formula for densities on $[0,1]^d$.
\begin{prop}\label{boundedscorematch}
Consider random vectors taking values in $[0,1]^d$, with density $q$; suppose $X\sim p$. If $q$ is twice continuously differentiable and satisfies
\begin{align}
	&\Vert p(x) \nabla\log q(x) \otimes x(1-x) \Vert \rightarrow 0, \text{ for all } x\text{ approaching the boundary},
\end{align}
 where $\otimes$ denotes the tensor product $x\otimes y:= (x_1y_1,\ldots,x_dy_d)$, then 
 \begin{align}
 	h(x,q) &:= \frac{1}{2}\Vert \nabla \log q(x) \otimes x(1-x) \Vert_2^2 \notag\\
	&\qquad+ \sum_{i\in V}\left( -2(2x_i-1)x_i(1-x_i)\frac{\partial \log q(x)}{\partial x_i} + x_i(1-x_i)\frac{\partial^2 \log q(x)}{\partial x_i^2}\right),
 \end{align}
is a proper scoring rule. In  particular when $q$ is an exponential family with natural parameters $\theta$ and sufficient statistics $\phi$,
$h(x,\tilde{\theta})= \frac{1}{2} \tilde{\theta}^\top \tilde{A}(x)\tilde{\theta}+K(x)^\top\theta$ is a proper scoring rule, where
\begin{align}
	K(x)_{ij} &= -2(2x_i-1)x_i(1-x_i)\frac{\partial\phi_{ij}}{\partial x_i}+(x_i(1-x_i))^2\frac{\partial^2\phi_{ij}}{\partial x_i^2}, \\
	a_i(x) &= x_i(1-x_i)\frac{\partial \phi_{\cdot,i}}{\partial x_i}.
\end{align}
and $\tilde{A}(x)=\text{diag}(a_i(x)a_i(x)^\top)$.
\end{prop}

Thus, all of the results in this work may be effortlessly carried over to exponential families over bounded supports.

\end{remark}

\section{Previous Work}
There is a small but growing literature on applications using the Hyv\"arinen score for estimation. \citep{sriperumbudur2013density} consider using the Hyv\"arinen score for density estimation in a reproducing kernel Hilbert space (RKHS). They consider the optimization for a density $q$,
\begin{align}
	\min_q \left\{\frac{1}{n}\sum_{k=1}^n h(X^k,q) + \frac{\lambda}{2}\Vert q\Vert_{\mathcal{H}}^2\right\},
\end{align}
where $\Vert\cdot\Vert_{\mathcal{H}}^2$ is the norm of the RKHS. After an application of the representer theorem \citep{kimeldorf1971some}, they show this may be expressed as a finite-dimensional quadratic program. They derive rates for convergence to the true density with respect to the Fisher divergence.

\citep{vincent2011connection} shows that the denoising autoencoder may be expressed as a type of score matching estimator, which they call \emph{denoising score matching}. Suppose that $\tilde{X}$ is a version of a sample $X$ which has been corrupted by Gaussian noise, so that its conditional distribution has the score $\partial \log q(\tilde{x}\mid x) = \frac{1}{\sigma^2} (x-\tilde{x})$. Suppose we seek to fit the corrupted data according to a density of the form
\begin{align}
	\log p(\tilde{x}\mid W,b,c) \propto -\frac{1}{\sigma^2}\left(\langle c,\tilde{x}\rangle-\frac{1}{2}\Vert \tilde{x}\Vert_2^2+\text{softplus}\left(\sum_j\langle W_j,\tilde{x}\rangle + b_j\right)\right),
\end{align}
where $\text{softplus}(x)=\max(0,x)$, then minimizing the Fisher divergence between the model density and $q(\tilde{x}\mid x)$ can be shown to be equivalent to minimizing
\begin{align}
	\mathbb{E}_{q(\tilde{x},x)}\left[\Vert W^\top\text{sigmoid}(W\tilde{X}+b)+c-X\Vert^2\right].
\end{align}
This is a simple denoising autoencoder with a single hidden layer, encoder $f(\tilde{x})=\text{sigmoid}(W\tilde{x}+b)$, and decoder $f'(y)=W^\top y+c$.

Score matching has also been used for learning natural image statistics \citep{kingma2010regularized,koster2009estimating}.

\section{Score Matching Estimator}
Define the statistics
\begin{align}
	\widehat{\Gamma} = \frac{1}{n}\sum_{r=1}^n A(X^r), \\
	\widehat{K} = \frac{1}{n}\sum_{r=1}^n K(X^r).
\end{align}

The \emph{regularized score matching estimator} is a solution to the problem
\begin{align}
	\widehat{\theta}\in \underset{\theta}{\text{argmin}} \left\{\frac{1}{2}\theta^\top \widehat{\Gamma}\theta + \widehat{K}^\top \theta+\mathcal{R}(\theta)\right\}. \label{scorematch}
\end{align}

Here $\mathcal{R}$ is the group penalty
\begin{align}
	\mathcal{R}(\theta) = \sum_{i,j\in V} \Vert\theta_{ij}\Vert_2.
\end{align}
This norm induces sparsity in groups (i.e. edges/vertices). In high dimensions, regularizing the vertex parameters is necessary, as \eqref{scorematch} need not exist otherwise. Both the scoring rule and regularizer of \eqref{scorematch} are convex in $\theta$, so it is a convex program. In particular, observe that it can be equivalently represented as
\begin{align}
	&\min_{t,t_{ij}}\left\{ t + \lambda\sum_{ij} t_{ij}\right\} \label{socpform}\\
	s.t. \qquad  &t  \geq \frac{1}{2}\theta^\top\widehat{\Gamma}\theta+\widehat{K}^\top\theta, \notag \\
	&t_{ij} \geq \Vert \theta_{ij}\Vert \notag. 
\end{align}
\eqref{socpform} is a second-order cone program (SOCP) \citep{boyd2004convex}, as the quadratic constraint can be re-written as a conic constraint. If $\widehat{\Gamma}$ is not positive definite, particularly when $n>d$, \eqref{scorematch} may not be unique. This is typical for high dimensional problems. One can impose further assumptions to guarantee uniqueness. For example various assumptions have been described for the lasso (see an overview of these assumptions in \citep{tibshirani2013lasso}) , but we won't go into those details here.

\subsection{Gaussian Score Matching}

Consider the Gaussian density:
\begin{align}
	q(x) \propto \exp\left\{ - \frac{1}{2} x^\top \Omega x \right\},
\end{align}
for $\Omega\succ 0$. We have
\begin{align}
	\nabla \log q (x) &= -\Omega x, \\
	\nabla_i (\nabla_i \log q(x)) &= -\Omega_{ii}.
\end{align}
so the Hyv{\"a}rinen score is given by
\begin{align}
	h(x,\Omega) &= -\sum_i \Omega_{ii}+\frac{1}{2} x^\top \Omega^\top\Omega x \\ 
	&= \text{trace}\left(-\Omega + \frac{1}{2}\Omega^2 x x^\top \right).
\end{align}


Let $\widehat{\Sigma}=\frac{1}{n}\sum_{r=1}^n X^r(X^r)^\top$. The  optimal regularized score estimator $\widehat{\Omega}$ is the solution to
\begin{align}
	&\min_{\Omega=\Omega^\top} \left\{ \text{trace}\left(\frac{1}{2}\Omega\widehat{\Sigma}\Omega-\Omega\right) + \lambda\Vert\Omega\Vert_1\right\}.\label{gaussreg}
\end{align}

In the notation of \eqref{scorematch}, we have $\theta=\text{vec}(\Omega)$, $\widehat{K}=\text{vec}(I_d)$ and  $\widehat{\Gamma}_i=\widehat{\Sigma}$ for each $i\in V$.
We do not impose a positive definite constraint on $\Omega$. Doing so would still result in a convex program, indeed it is a semidefinite program, but the resulting computation becomes more complicated and less scalable in practice. However, our theoretical results imply that $\widehat{\Omega}$ is positive definite with high probability. Indeed, denote $\Vert\widehat{\Omega}-\Omega^*\Vert_{sp}$ the spectral norm (maximum absolute value of eigenvalues) of the difference $\widehat{\Omega}-\Omega^*$. Since the spectral norm is dominated by the Frobenius norm (elementwise $L_2$ norm), the consistency result in the sequel implies consistency in spectral norm, and so the eigenvalues of $\widehat{\Omega}$ will be positive with probability approaching one, assuming the population precision matrix $\Omega^*$ has strictly positive eigenvalues. Furthermore, we note that our model selection guarantees still follow whether or not the estimator $\widehat{\Omega}$ is positive definite.

%

\section{Main Results}
We suppose we are given i.i.d. data $X^1,\ldots,X^n \sim p^*$. $p^*$ need not belong to the pairwise exponential family being estimated, in which case we may think of our consistency results as being relative to the population quantity
\begin{align}
	\theta^* &:= \left\{\mathbb{E}_{p^*}[A(X)]\right\}^{-1}\mathbb{E}_{p^*}[K(X)] \\
	&= (\Gamma^*)^{-1}K^*. \label{popsol}
\end{align}
Define the maximum column sum of $\theta^*$ by 
\begin{align}
\kappa_{\theta,1}:=\max_{i\in V} \sum_{j\in V,u\leq m} (\theta^*)_{ij}^u,
\end{align}
and define the maximum degree as
\begin{align}
s := \max_{i\in V}\left| \{(i,j):(i,j)\in E\}\right|.
\end{align}

\begin{assump} \label{eigenbound}
$\Gamma_i^*=\mathbb{E}_{p^*}[a_i(X)a_i(X)^\top]$ satisfies for each $i\in V$,
\begin{align}
	\infty>\bar{\epsilon}\geq\Lambda_{\max}(\Gamma_i^*)\geq\Lambda_{\min}(\Gamma_i^*) \geq \underline{\epsilon}>0.
\end{align}
Note that this also implies that the eigenvalues of $\Gamma^*$ are bounded as the rows of $\Gamma^*$ are non-trivial linear combinations of those of $\text{diag}(\Gamma_i^*)$, so the inverse in \eqref{popsol} exists and is unique.
\end{assump}
We also suppose $\theta^*$ is sparse, in the following sense:

\begin{assump}
$\theta^*$ belongs to the set
\begin{align}
	\tilde{\mathcal{P}} := \tilde{\mathcal{P}}(E)= \left\{\theta: \Vert\theta_{ij}\Vert_2=0,\text{  for } (i,j)\in E^c\right\}.
\end{align}
\end{assump}
For both parameter consistency and model selection we require the following tail conditions:
\begin{assump} \label{quasrconcentrate}
	For each $i,j,k\in V$ and $u\leq m$ and $t\leq \nu$, for some $c_1,c_2,\nu>0$,
	\begin{align}
		\mathbb{P}\left( \left|\widehat{K}_{ij}^u-(K^*)_{ij}^u\right|\geq t\right) &\leq \exp\left\{ - c_1nt^2 \right\} \\
		\mathbb{P}\left( \left|(\widehat{\Gamma}_{i})_{jk}^u-(\Gamma^*_{i})_{jk}^u\right|\geq t\right) &\leq \exp\left\{ - c_2nt^2 \right\}.
	\end{align}
\end{assump}
\subsection{Parameter Consistency}
We present results in terms of the (vector) $L_2$ norm. Note in particular that this result doesn't require any incoherence condition (though we do require for model selection consistency in the sequel).

For the parameter consistency results in particular, we require the following sub-Gaussian assumption:

\begin{assump}
For each $i\in V$ and $r=1,\ldots,n$, $a_i(X^r)$ is a sub-Gaussian random vector.
\end{assump}

\begin{thm}\label{parconsist}
Suppose the regularization parameter is chosen as
\begin{align}
	\lambda_n\asymp\sqrt{\frac{m\kappa_{1,\theta}^2\log(md)}{n}},
\end{align}
if the sample size satisfies
\begin{align}
	n=\Omega(md),
\end{align}
then any solution to regularized score matching satisfies
\begin{align}
	\Vert\widehat{\theta}-\theta^*\Vert_{2} &=O_p\left(\sqrt{\frac{(d+\left|E\right|)m\kappa_{1,\theta}^2\log(md)}{n}}\right).
\end{align}
\end{thm}

\begin{remark}
	Consider Gaussian score matching. Here $m=1$, so if $\kappa_{1,\theta}$ is bounded we have $\Vert\widehat{\theta}-\theta^*\Vert_2 =O_p\left(\sqrt{\frac{(d+\left|E\right|)\log(d)}{n}}\right)$. This rate is the same as the graphical lasso shown in \citep{rothman2008sparse}. Furthermore, here $\Gamma_i^*=\Sigma$, so our assumption \ref{eigenbound} amounts to bounds on the eigenvalues of $\Sigma$, which are the same as for sparse precision matrix MLE. The assumption that $\kappa_{1,\theta}$ is bounded here says that the  sums of the absolute value of rows of $\Omega^*$ are bounded, which is not necessary for the regularized MLE.
\end{remark}

\begin{remark}
We might reasonably expect $\kappa_{1,\theta}=O(sm)$, in which case $\Vert\widehat{\theta}-\theta^*\Vert_2 = O_p\left(\sqrt{\frac{(d+\left| E\right|)m^3s^2\log(md)}{n}}\right)$. In this setting the regularized MLE will have the rate 
\begin{align}
O_p\left(\sqrt{\frac{(d+\left| E\right|)m \log(md)}{n}}\right).
\end{align}
 (see results in Appendix A).
\end{remark}

\subsection{Model Selection} \label{smms}

For model selection we require several additional conditions. Denote $\widehat{E}$ as the edge set learned from $\widehat{\theta}$:
\begin{align}
	\widehat{E}:=\left\{(i,j): \Vert\widehat{\theta}_{ij}\Vert_2=0\right\}.
\end{align}
Furthermore, define 
\begin{align}
	\kappa_\Gamma := \Vert (\Gamma^*)^{-1}\Vert_\infty, \\
	\kappa_\theta := \Vert\theta^*\Vert_{\max},\\
	\rho^* := \underset{(i,j)\in E}{\min} \Vert\theta_{ij}\Vert_{\max}.
\end{align}

Here $\Vert A\Vert_\infty = \max_j\sum_i \left| A_{ij}\right|$ is the matrix $\infty$ norm and $\Vert\cdot\Vert_{\max}$ the elementwise max norm. We require an \emph{incoherence condition}:
\begin{assump}
\begin{align}
\max_{(i,j)\in E^c} \Vert \Gamma_{ij,E}^*(\Gamma^*_{EE})^{-1}\Vert_2 \leq \frac{1-\tau}{\sqrt{d+E}}, && \text{for some } \tau\in (0,1].
\end{align}
where $\Vert A\Vert_2$ is the matrix operator norm. 
\end{assump}

In the following theorem we suppose $\kappa_\Gamma,\kappa_\theta,s$ are are bounded, while $\rho^*$ may change with the sample size.

\begin{thm}\label{modelselectthm}
	Suppose the regularization parameter $\lambda_n$ is chosen to be
	\begin{align}
		\lambda_n \asymp \sqrt{\frac{m\kappa_{1,\theta}^2 \log(dm)}{n}},
	\end{align}
	then if
	\begin{align}
		n&=\Omega(\max\{m\kappa_{1,\theta}^2\log(dm),m^2s^2\log(dm)\}), \\
		\frac{1}{\rho^*} &= o\left(\sqrt{\frac{\kappa_{1,\theta}^2\log(dm)}{n}}\right),
	\end{align}
	there exists a solution to the regularized score matching estimator $\widehat{\theta}$ with estimated edge set $\widehat{E}$ satisfying
\begin{align}
	\mathbb{P}(\widehat{E}=E)\rightarrow 1.
\end{align}
\end{thm}

\begin{remark}
	Assuming $m,s,\kappa_{1,\theta}$ are bounded, this implies the dimension may grow nearly exponentially with the sample size:
\begin{align}
		d=o(e^n),
\end{align}
with the probability of model selection consistency still aproaching one.
\end{remark}
\begin{remark}[Gaussian score matching]

When $m=1$, the sample complexity matches that for structure learning of the precision matrix using the log-det divergence, in \citep{ravikumar2011high}. Thus Gaussian score matching in particular  benefits from identical model selection guarantees as the graphical lasso algorithm. However it should be noted that the assumptions are slightly different. In particular the graphical lasso requires an irrepresentable condition on $\Sigma\otimes\Sigma$, while our method involves an irrepresentable condition for $\Sigma\otimes I_d$.
\end{remark}
\subsection{Model Selection for the Nonparametric Pairwise Model}

In this section we consider model selection for the nonparametric pairwise model. We suppose the log of the true density $p^*$ belongs to $W_2^r$, the Sobolev space of order $r$. This implies, along with the pairwise assumption, that $\log p^*$ has the infinite expansion
\begin{align}
	\log p^* &\propto \exp\left\{ \sum_{i,j\in V,i\leq j}\sum_{k,l=1}^\infty (\theta^*)_{ij}^{kl}\phi_{kl}(x_i,x_j)+\sum_{i\in V}\sum_{k=1}^\infty (\theta^*)_{i}^k\phi_k(x_i)\right\}, \label{infexpans}
\end{align}
where here $\{\phi_k,\phi_{kl}\}$ is a basis over $[0,1]^2$. For an expansion in  $W_2^r$, we have that the coefficients decay at the following rates:
\begin{align}
	&\sum_k \theta_i^k k^{2r} <\infty, & \text{ for all } i\in V,\\
	&\sum_{k,l} \theta_{ij}^{kl} k^{2r_i}l^{2r_j}<\infty, & \text{ for all } (i,j)\in E,& \qquad r_i+r_j=r.
\end{align}

 For our results we assume $\{\phi_k\}$ is the orthonormal Legendre basis on $[0,1]$, and $\{\phi_{kl}\}$ is the tensor product basis $\phi_{kl}(x_i,x_j)=\phi_k(x_i)\cdot \phi_l(x_j)$. This is because the supporting lemmas are particular to the Legendre basis,  but in practice one is not limited to a particular basis. Now, consider forming a density by truncating \eqref{infexpans} after $m_1$ terms for the univariate expansions, and $m_2$ for bivariate:
\begin{align}
	\log p_{\theta} &\propto \exp\left\{ \sum_{i,j \in V,i\leq j}\sum_{k,l=1}^{m_2} \theta_{ij}^{kl}\phi_{kl}(x_i,x_j)+\sum_{i\in V}\sum_{k=1}^{m_1}\theta_{i}^k\phi_k(x_i)\right\}.
\end{align}

Observe that this is a finite-dimensional exponential family. Furthermore, the normalizing constant for this family will generally be intractable, requiring a $d$-fold integral. We choose our density estimate to be $p_{\widehat{\theta}}$, where $\widehat{\theta}$ is a solution to the score matching estimator \eqref{scorematch} for this family. Furthermore, we let the number of sufficient statistics $m_1,m_2$ grow with the sample size $n$ to balance the bias from truncation with the estimation error. We denote $E$ to be the support of $p^*$:
\begin{align}
	E:= \left\{ (i,j): \Vert\theta_{ij}^*\Vert=0\right\},
\end{align}

Now, decompose the vector $\theta^*$ into the included terms and truncated terms, $\theta^*=((\bar{\theta}^*)^\top,(\theta^*_T)^\top)^\top$ and corresponding sufficient statistics $\phi=((\bar{\phi})^\top,(\phi_T)^\top)$. Denote $(a_T)_i(x) := \frac{\partial}{\partial x_i} (\phi_T){\cdot,i}$, $A_T(x)= \text{diag}((a_T)_i(x)(a_T)_i(x)^\top)$, and $\Gamma_T^* =\mathbb{E}_p [A_T(X)]$. Applying the results in Section \ref{smef}, we have the linear relation
\begin{align}
	K^* = -\Gamma^*_T\theta^*_T - \Gamma^* \bar{\theta}^*.
\end{align}

In the following theorem we assume $\kappa_T := \Vert\Gamma^*_T\Vert_{\max}$ is bounded, and $\kappa_{1,\theta}=O(m_2^2)$, in addition to the assumptions for the parametric setting stated in Section \ref{smms}, with the exception of Assumption \ref{quasrconcentrate}. Since the number of statistics grows to infinity in the nonparametric case, we need more accurate accounting of the constant terms in the concentration inequality. In lieu of the concentration assumption, we have the following assumption on the boundedness of the marginals of $p$.

\begin{assump} \label{assump4}
		 For each $i,j\in V$,
	\begin{align}
		\underline{\epsilon} & \leq p_{ij}(x_i,x_j) \leq \bar{\epsilon},
		\end{align}
		for absolute constants $\underline{\epsilon}>0,\bar{\epsilon}<\infty$.
\end{assump}

This assumption is mild for density estimation as it only requires bounds on the bivariate marginals rather than the full distribution. This is the same assumption used in Chapter 2 for the TRW estimator.

\begin{thm}\label{msnonp}
	Suppose that the truncation parameters and regularization parameter are chosen to be
	\begin{align}
		m_2\asymp n^{\frac{1}{2r+13}} \\
		m_1\asymp n^{\frac{1}{2r+13}} \\
		\lambda_n\asymp\sqrt{\frac{\log nd}{n^{\frac{2r-1}{2r+13}}}} \
	\end{align}
	and the dimension $d$ and $\rho^*$ satisfy
	\begin{align}
		d= o\left(e^{n^{\frac{2r-1}{2r+13}}}\right) \\
		\frac{1}{\rho^*} = o\left(\sqrt{\frac{\log nd}{n^{\frac{2r+1}{2r+13}}}}\right),
	\end{align}
	then there exists a solution $\widehat{\theta}$ such that the edge set $\widehat{E}$ satisfies
	\begin{align}
		\mathbb{P}(\widehat{E}=E)\rightarrow 1.
	\end{align}
\end{thm}
\begin{remark}
	If $r=2$, and $s$ grows as a constant, we may have
	\begin{align}
		d=o\left(e^{n^{3/17}}\right),
	\end{align}
	and still ensure model selection consistency. In Chapter 2 it was shown that the sample complexity for model selection in the nonparametric pairwise model the regularized exponential series MLE using Legendre polynomials is $d=o\left(e^{n^{3/7}}\right)$, though this estimator can't be computed exactly. The optimal choice of regularization and truncation parameters is much different for these two methods. This is a consequence of different estimation errors. In our supporting lemmas (see Appendix B) we require convergence of the statistic $\widehat{\Gamma}$ to its expectation. In Appendix B we show that applying Hoeffding's inequality and a union bound,
\begin{align}
	\kappa_{1,\theta}\Vert\widehat{\Gamma}-\Gamma\Vert_{\max} = O_p\left(\sqrt{\frac{m_2^{12} \log nd}{n}}\right).
\end{align}

For the regularized MLE, we needed convergence of the sufficient statistics $\widehat{\mu}$, which converges at a much faster rate of $O_p\left(\sqrt{\frac{m_2^4 \log nd}{n}}\right)$. Our results agree with intuition, that the score matching statistics, derived from the derivatives of the log-density, should be harder to estimate than the sufficient statistics.

Also it should be noted that the assumptions underlying the two results are quite different. The MLE involves conditions on the covariance of the sufficient statistics $\text{cov}_{p^*}[\phi(X)]$, while the score matching estimator requires conditions on $\Gamma^*=\mathbb{E}_{p^*}[A(X)]$. An interesting stream of future work would be to better understand the relationship between these two approaches and their assumptions.
\end{remark}
\section{Algorithms}
In this section we consider algorithms for solving \eqref{scorematch}. In our experiments we denote our method QUASR, for \textbf{Qua}dratic \textbf{S}coring and \textbf{R}egularization. There are variety of generic approaches to solving problems which may be cast as the sum of a smooth convex function plus a sparsity-inducing norm  \citep{bach2011convex}, as well as generic solvers for solving second-order cone programs. Here we will propose two novel algorithms which exploit the unique structure of the problem at hand. First we will consider an ADMM algorithm; for a detailed exposition of this approach, see \citep{boyd2011distributed}. In section \ref{cwd}, we consider a coordinate-wise descent algorithm for Gaussian score matching \citep{friedman2007pathwise}.
\subsection{Consensus ADMM} \label{cadmm}

The idea behind ADMM is that the problem \eqref{scorematch} can be equivalently written as
\begin{align}
	\underset{\theta,z}{\text{min}} \left\{\frac{1}{2}\sum_{i\in V} \left(\theta_{\cdot,i}^\top \widehat{\Gamma}_i \theta_{\cdot,i} +\theta_{\cdot,i}^\top\widehat{K}_{\cdot,i}\right)+ \lambda\sum_{i,j\in V,i\leq j}\Vert z_{ij}\Vert_2\right\},
\end{align}
subject to the constraint that $\theta_{ij}=\theta_{ji}=z_{ij}$. The scaled augmented Lagrangian for this problem is given by
\begin{align}
	L\left(\theta,y,z\right) &= \frac{1}{2}\sum_{i\in V}\left( \theta_{\cdot,i}^\top \widehat{\Gamma}_i \theta_{\cdot,i} +\theta_{\cdot,i}^\top\widehat{K}_{\cdot,i}\right) \\
	&\quad+\sum_{i,j\in V: i\leq j}\bigg( \Vert z_{ij}\Vert_2+y_{ij}^\top(\theta_{ij}-z_{ij})+y_{ji}^\top(\theta_{ji}-z_{ij}) \\
	&\qquad\qquad\qquad+\frac{\rho}{2}\Vert \theta_{ij}-z_{ij}\Vert^2+\Vert \theta_{ji}-z_{ij}\Vert^2\bigg),
\end{align}
here $\{y\}:=\left\{y_{ij},y_{ji}\right\}$ are dual variables, and $\rho$ is a penalty parameter which we choose to be 1 for simplicity. The idea behind ADMM is to iteratively optimize $L$ over the $\theta,y,z$ variables in turn. In the first step, since $\theta_{ij}=\theta_{ji}$ is included as a constraint and may be considered separately, $L$ as a function of $\theta$ decouples into $d$ independent quadratic programs, one for each "column" of $\theta$, which may be solved in parallel. In the second step, $z_{ij}$ pools the estimates $\theta_{ij}$ and $\theta_{ji}$ from the previous step, and applies a group shrinkage operator. The third step is a simple update of the dual variables.

\begin{figure}
\begin{center}
\framebox{\begin{minipage}[t]{0.95\columnwidth}
\begin{enumerate}
	\item Initialize $\theta^{(0)}$, $z^{(0)}$, $y^{(0)}$, and choose $\rho=1$;
	\item For $t=1,\ldots,$ until convergence:
		\begin{enumerate}
			\item Update $\theta$ for $i\in V$: 
			\begin{align}
	\theta_{\cdot,i}^{(t)} = \left(\widehat{\Gamma}_i+\rho I_d\right)^{-1}\left(-\widehat{K}_{\cdot,i}-y_{\cdot,i}^{(t-1)}+\rho z_{\cdot,i}^{(t-1)}\right),
\end{align}
	
			\item Update $z$ for $i,j\in V$, $i\leq j$:
			\begin{align}
				z_{ij}^{(t)} = \tilde{S}\left(\frac{1}{2}\left(\theta_{ij}^{(t)}+\theta_{ji}^{(t)}+y_{ij}^{(t-1)}/\rho+y_{ji}^{(t-1)}/\rho\right),\lambda/\rho\right),
			\end{align}
			where $\tilde{S}(x,\lambda):= \left(1-\frac{\lambda}{\Vert x\Vert_2}\right)_+ x$.
			\item Update $y$ for $i,j\in V$, $i\leq j$:
			\begin{align}
				y_{ij}^{(t)} = y_{ij}^{(t-1)} + \rho\left(x_{ij}^{(t)}-z_{ij}^{(t)}\right),\\
				y_{ji}^{(t)} = y_{ji}^{(t-1)} + \rho\left(x_{ji}^{(t)}-z_{ij}^{(t)}\right).
			\end{align}
		\end{enumerate}		
\end{enumerate}
\end{minipage}}
\caption{QUASR Consensus ADMM}
\end{center}
\end{figure}
Due to parallel updating of $\theta$ in step (a) and subsequent averaging in step (b), this is known as \emph{consensus ADMM}. At convergence, the constraints $\theta_{ij}=\theta_{ji}=z_{ij}$ are binding. In practice, we stop when the average change in parameters is small:
\begin{align}
	\sum_{i,j\in V} \Vert\theta_{ij}^{(t)}-\theta_{ij}^{(t-1)}\Vert_1\big/\sum_{i,j\in V} \Vert\theta_{ij}^{(t)}\Vert_1 < 10^{-4}.
\end{align}

In addition to parallelizing the update (a), other speedups are possible. For example, we may compute the eigenvalues $\Lambda$ and eigenvectors $Q$ of $\widehat{\Gamma}_i$, which may be computed directly from the data matrix $\left[a_i(X^1),\ldots a_i(X^n)\right]^\top$ using the singular value decomposition ($Q$ being the right singular vectors, and $\sqrt{n}\Lambda$ being the squared singular values of the data matrix). We may then cache the matrix $Q(\text{diag}(\Lambda+\rho)^{-1})Q^\top$, which is equivalent to $(\widehat{\Gamma}_i+\rho I_d)^{-1}$ up to numerical error. This can be computed for each $i\in V$, also in parallel, and only needs to be computed once (even if estimating over a sequence of $\lambda s$). When optimizing over a path of truncation parameters $m_1,m_2$, one may utilize block matrix inversion formulas and the Woodbury matrix identity to avoid computing the inverse from scratch each time. In particular, let $\widehat{\Gamma}_i$ be the current matrix of statistics, and $\widehat{\Gamma}_i^{new}$ be the statistic with a higher degree of basis expansion. Then $\widehat{\Gamma}_i^{new}$ has the form for some $\widehat{b},\widehat{C}$,
\begin{align}
	\widehat{\Gamma}_i^{new} = \left(\begin{array}{cc}
		\widehat{\Gamma}_i & \widehat{b} \\
		\widehat{b}^\top & \widehat{C}
	\end{array}
	 \right).
\end{align}
The inverse takes the form
\begin{align}
	(\widehat{\Gamma}_i^{new}+\rho I)^{-1} = \left(\begin{array}{cc}
		L &  \\
		-\left(\widehat{C}+\rho I\right)^{-1}\widehat{b}^\top L & \left(\widehat{C}+\rho I -\widehat{b}^\top \left(\widehat{\Gamma}_i+\rho I\right)^{-1}\widehat{b}\right)^{-1}
	\end{array}
	 \right),
\end{align}
where
\begin{align}
	L&:=\left(\widehat{\Gamma}_i+\rho I - \widehat{b}^\top \left(\widehat{C}+\rho I\right)^{-1} \widehat{b}\right)^{-1} \\
	&= \left(\widehat{\Gamma}_i+\rho I\right)^{-1} - \left(\widehat{\Gamma}_i+\rho I\right)^{-1}\widehat{b}^\top\left(\widehat{C}+\rho I -\widehat{b}^\top \left(\widehat{\Gamma}_i+\rho I\right)^{-1}\widehat{b}\right)^{-1}\widehat{b}\left(\widehat{\Gamma}_i+\rho I\right)^{-1}.
\end{align}

If the dimension of $\widehat{C}$ is small relative to that of $\widehat{\Gamma}_i$, $\left(\widehat{\Gamma}_i^{new}+\rho I\right)^{-1}$ can be computed quickly using the cached $\left(\widehat{\Gamma}_i + \rho I\right)^{-1}$, without the need for any additional large matrix inversions.

\subsection{Coordinate-wise Descent} \label{cwd}
In this section we consider a coordinate-wise descent algorithm for the Gaussian score matching problem \eqref{gaussreg}. Coordinate-wise descent algorithms are known to be state-of-the-art for many statistical problems such as the lasso and group lasso \citep{friedman2007pathwise} and glasso for sparse Gaussian MLE \citep{friedman2008sparse}. Regularized score matching in the Gaussian case admits a particularly simple coordinate update.
Consider the stationary condition for $\Omega$ in \eqref{gaussreg}: 
\begin{align}
	\frac{1}{2}\left(\Omega\widehat{\Sigma}+\widehat{\Sigma}\Omega\right)-I_d + \widehat{Z}=0,
\end{align}
where $\widehat{Z}$ is an element of the subdifferential $\partial\Vert\Omega\Vert_1$:
\begin{align}
	\widehat{Z}_{ij} \in \begin{cases}
		\{\theta_{ij}: \Vert\theta_{ij}\Vert_2\leq 1\}, & \text{  if } \Vert\theta_{ij}\Vert=0; \\
		\frac{\theta_{ij}}{\Vert\theta_{ij}\Vert_2}, & \text{  if } \Vert\theta_{ij}\Vert\not=0.
	\end{cases}
\end{align}
in particular, the stationary condition for a particular $\Omega_{ij}$ is
\begin{align}
	\frac{1}{2}\left(\Omega_{\cdot,i}^\top\widehat{\Sigma}_{\cdot,j}+\widehat{\Sigma}_{\cdot,i}^\top\Omega_{\cdot,j}\right)-1\{i=j\}+\widehat{Z}_{ij}=0.\label{stationary}
\end{align}
%
%
Consider updating $\Omega_{ij}$ using equation \eqref{stationary}, solving for $\Omega_{ij}$ and holding the other elements of $\Omega$ fixed. After some manipulation, we get a fixed point for $\Omega_{ij}$ is given by \eqref{fixed}. We cycle through the entries of $\Omega$, applying this update, and repeat until convergence.

\begin{figure}
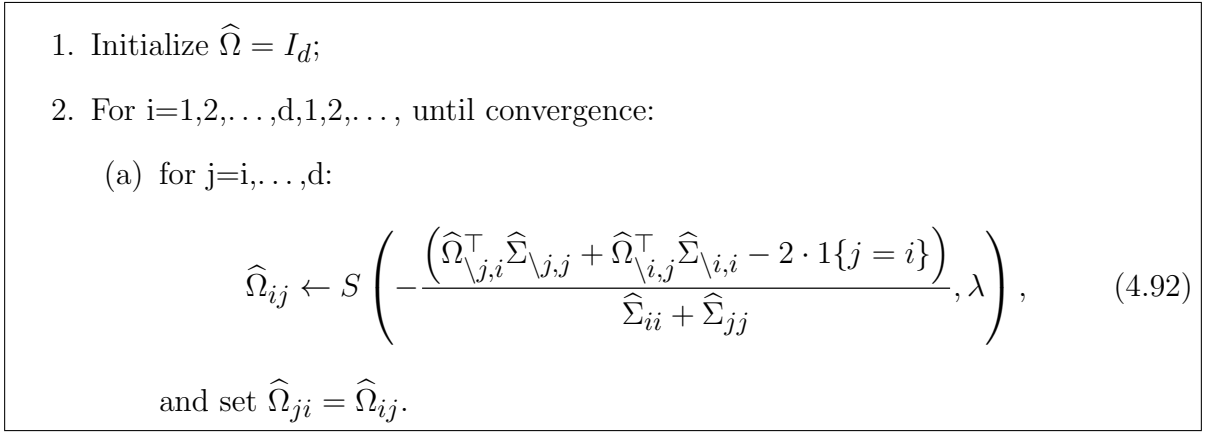

\begin{center}
\framebox{\begin{minipage}[t]{0.95\columnwidth}

\begin{enumerate}
	\item Initialize $\widehat{\Omega}=I_d$;
	\item For i=1,2,\ldots,d,1,2,\ldots, until convergence:
	\begin{enumerate}
	\item for j=i,\ldots,d:
			\begin{align}
	\widehat{\Omega}_{ij} \leftarrow S\left(-\frac{\left(\widehat{\Omega}_{\backslash j,i}^\top \widehat{\Sigma}_{\backslash j,j} + \widehat{\Omega}_{\backslash i,j}^\top \widehat{\Sigma}_{\backslash i,i} - 2\cdot 1\{j=i\} \right)}{\widehat{\Sigma}_{ii}+\widehat{\Sigma}_{jj}},\lambda\right) ,\label{fixed}
\end{align}
	and set $\widehat{\Omega}_{ji}=\widehat{\Omega}_{ij}$.
	\end{enumerate}
\end{enumerate}
\end{minipage}}
\caption{Gaussian QUASR Coordinate-wise descent}
\end{center}
\end{figure}
Here $S(x,\lambda)$ is the soft thresholding function $S(x,\lambda):=\text{max}\{\left| x\right|-\lambda,0\}\text{sign}(x)$, and $\backslash i := \{1,\ldots,i-1,i+1,\ldots,d\}$. Each update only requires two sparse inner products and a soft thresholding operation. As such, in our experiments this algorithm converges very quickly, sometimes much faster than glasso for the same set of data.

\subsection{Choosing Tuning Parameters}
As of yet we have not discussed how to practically choose the regularization parameter $\lambda$ and for nonparametric score matching, the truncation parameters $m_1,m_2$. We suppose the existence of a held-out tuning set; in the absence, one may use cross-validation. If the likelihood is available, for example if fitting Gaussian score matching, or for a fixed graph which is a tree, we minimize the negative log-likelihood risk in the held out set. In the absence of the likelihood, we choose the tuning parameters to minimize the Hyv\"arinen score of the held out set. For a discussion on using scoring rules as a replacement for the likelihood in model selection and using \emph{score differences} as surrogates for Bayes factors, see \citep{dawid2014theory}.

To save on computation, we use the idea of \emph{warm starts} which we detail in the sequel.
First, observe that the first-order necessary conditions for regularized score matching are:
\begin{align}
	 \widehat{\Gamma}\widehat{\theta}+\widehat{K}+\widehat{Z}=0,
\end{align}
where $\widehat{Z}$ denotes the sub gradient of the regularizer $\mathcal{R}$, at $\widehat{\theta}$, which is

\begin{equation}
	\widehat{Z}_{ij} = \begin{cases}
		\{x:\Vert x\Vert\leq 1\}, & \Vert\theta_{ij}\Vert = 0, \\
		\frac{\theta_{ij}}{\Vert\theta_{ij}\Vert}, & o/w. 	
	\end{cases}
\end{equation}
so $\widehat{\theta}=0$ when
\begin{equation}
	\lambda \geq \max_{ij}\Vert\widehat{K}_{ij}\Vert.
\end{equation}
This allows us to choose an upper bound $\lambda_{start}$ such that the solution will be the zero vector.

The idea behind warm starting is the following: we begin with estimating $\widehat{\theta}_{\lambda_{start}}=0$. Then we fit our model on a path of $\lambda$ decreasing from $\lambda_{start}$, initializing each new problem with the previous solution $\widehat{\theta}_{\lambda}$. The solution path for the regularized MLE is smooth as a function of $\lambda$, suggesting nearby choices of $\lambda$ will provide values of $\widehat{\theta}$ which are close to one another.
 
We can also incorporate warm-starting in choosing $m_1,m_2$. For a given $\lambda$, we first estimate the model for first-order polynomials, corresponding to $m_1=m_2=1$. We then increment the truncation parameters by increasing the degree of the polynomial of he sufficient statistics. We augment the previous parameter estimate vector with zeros in the place of the added parameters, and warm start \emph{ISTA} from this vector. See Section \ref{cadmm} for other computation savings when augmenting the sufficient statistics when choosing $m_1,m_2$.

\section{Experiments}

\subsection{Gaussian Score Matching}

We begin by studying Gaussian score matching, and comparing to the regularized Gaussian MLE, using the glasso package in $\tt{R}$ \citep{friedman2008sparse}. We consider experiments with two graph structures: in the first, a tree is generated randomly; this has $d-1$ edges. In the second, a graph is generated where an edge occurs between node $i$ and $j$ with probability $0.1$, denoted the Erd\"os-Renyi graph. This graph has expected number of edges $0.05\cdot d(d-1)$. The data is scaled to have unit variance and mean zero.

\subsubsection{Regularization Paths}
Figures \ref{regpath1} and \ref{regpath2} display regularization paths for one run of these simulations, where $d=100$ and $n$ is either 100 or 500. Relevant variables are plotted in black. For $n=500$, it appears that the score matching estimator does a better job screening out irrelevant variables for both graph types. For $n=100$ they perform similarly. The score matching estimator tends to produce nonzero parameter estimates which are larger in magnitude than the regularized MLE, which is more pronounced for $n<d$.

\begin{figure}
\centering
\includegraphics[scale=.5]{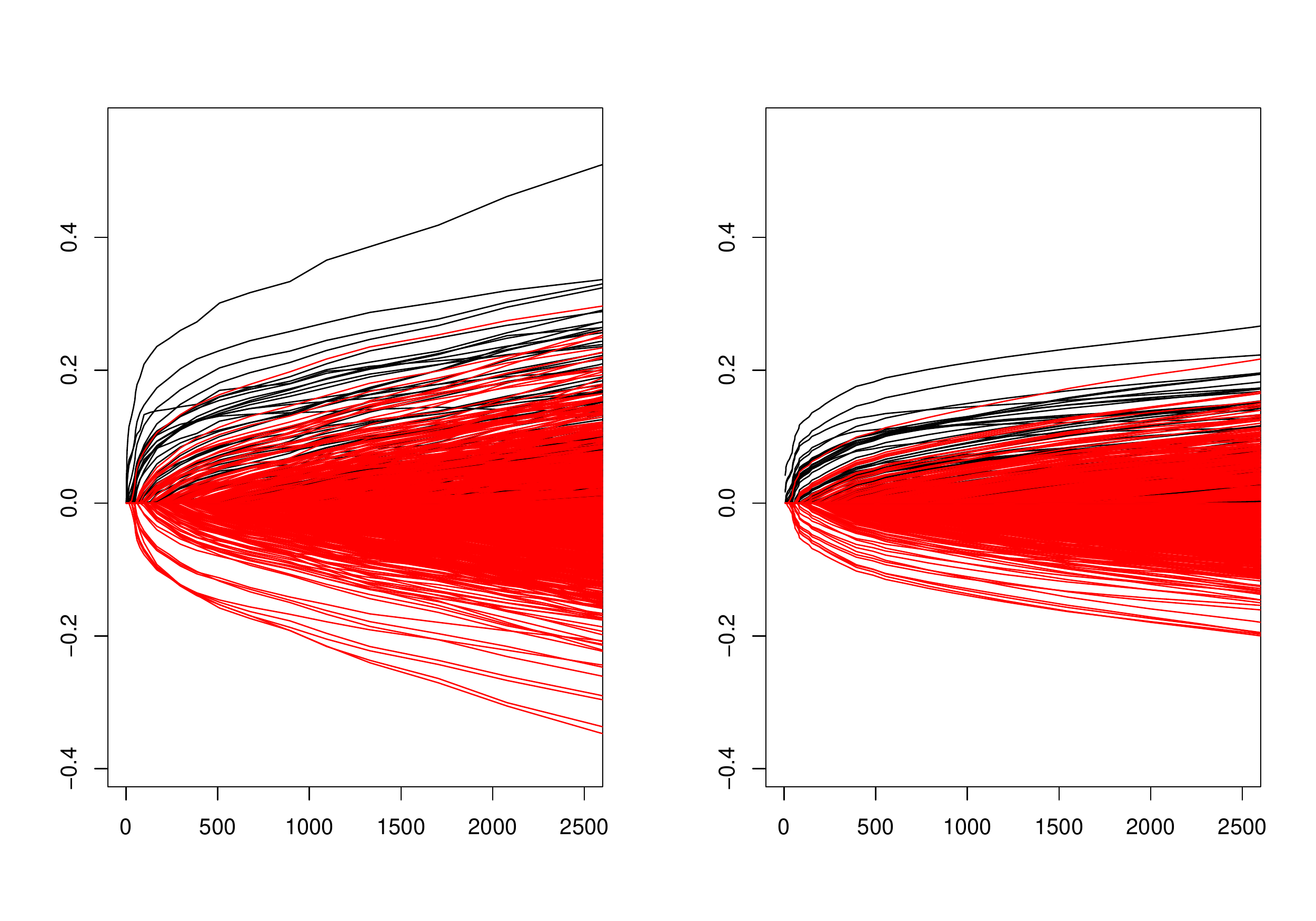}
\includegraphics[scale=.5]{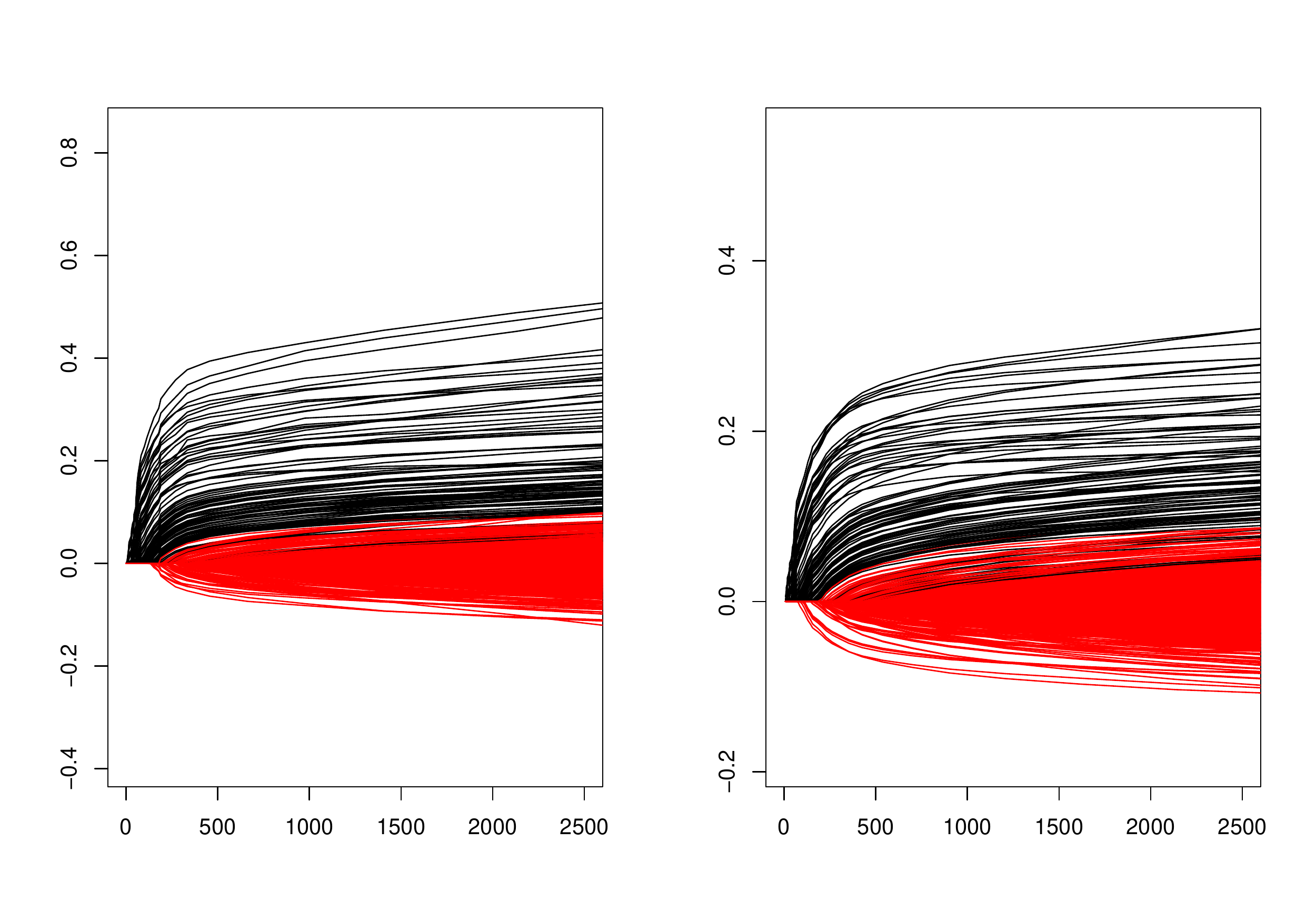}
\caption{Regularization path, tree graph. d=100. Top: n=100. Bottom: n=500.} \label{regpath1}
\end{figure}

\begin{figure}
\centering
\includegraphics[scale=.5]{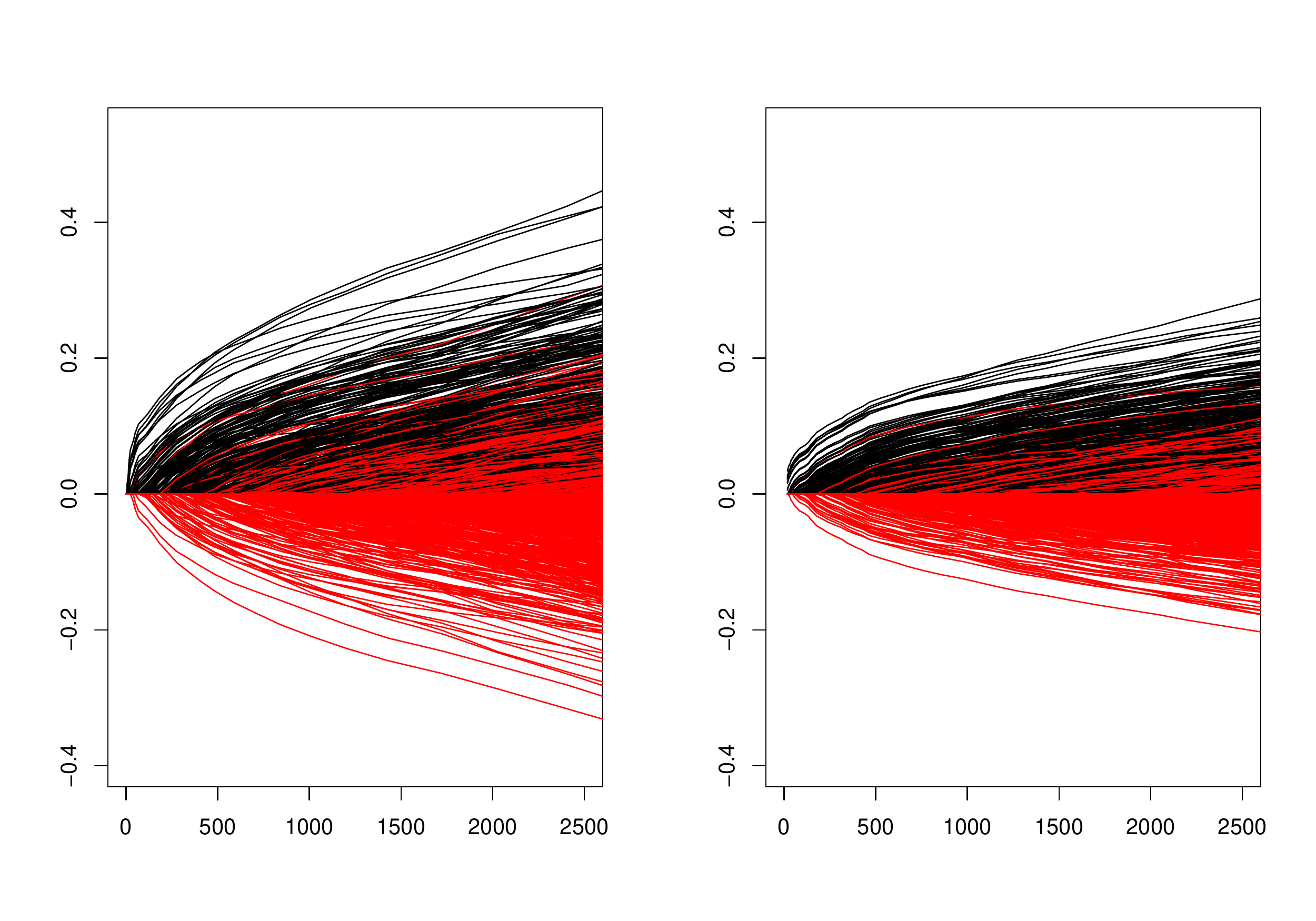}
\includegraphics[scale=.5]{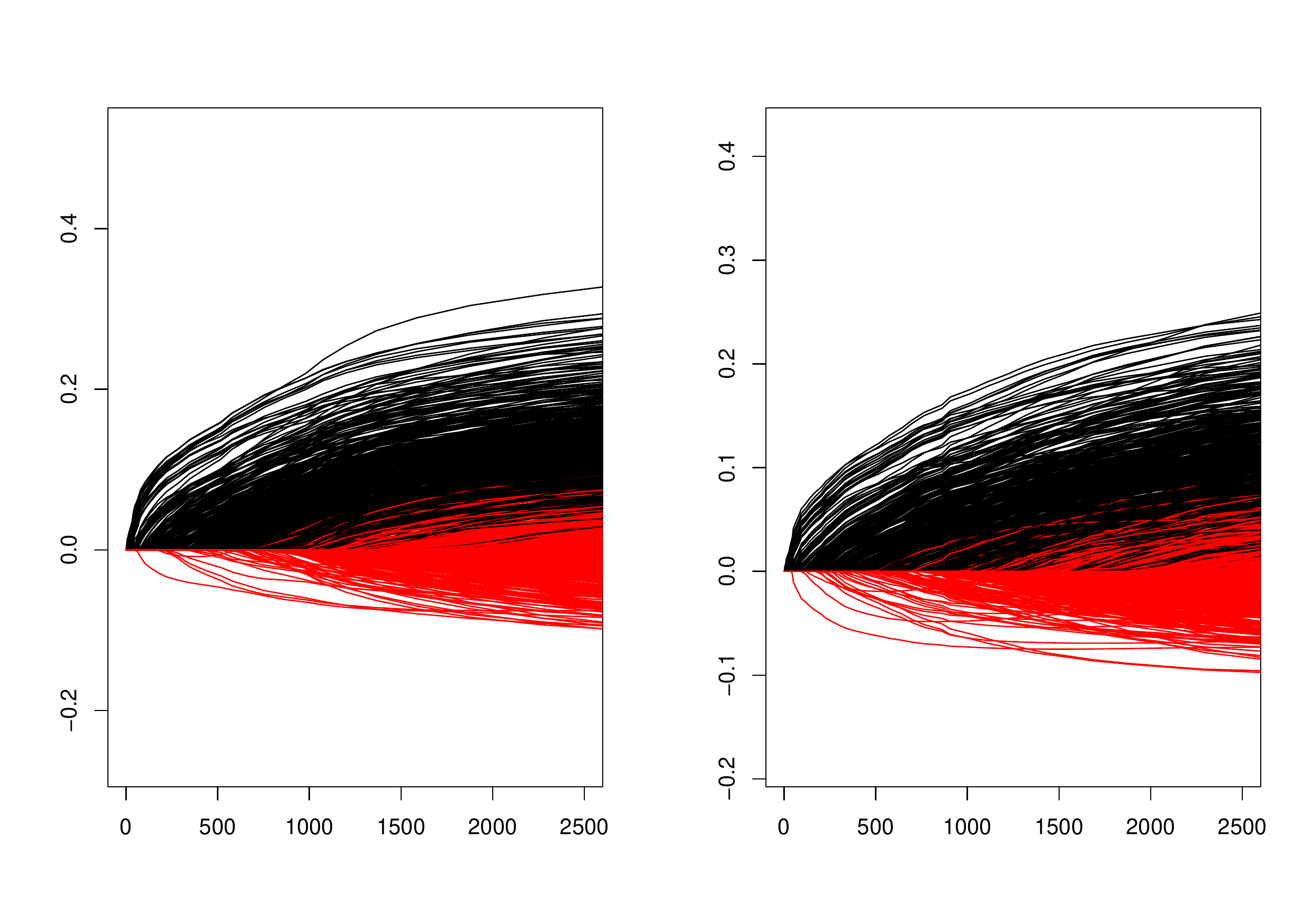}
\caption{Regularization path, Erdos-Renyi graph graph. d=100. Top: n=100. Bottom: n=500.} \label{regpath2}
\end{figure}

\subsubsection{Risk Paths}
Figures \ref{riskpath1} and \ref{riskpath2} show risk paths under the two graph structures; figure two has $d=150$, with 149 included edges; figure four has $d=100$, with 499 edges included. We choose $n=100$, and calculate the negative log likelihood risk using a held-out dataset of size $n$. The plotted curves are an average of 25 simulations from the same distribution.  In figure \ref{riskpath1}, we see the score matching estimator selects a sparser graph than the regularized MLE; furthermore, the score matching produces an estimator with smaller held-out risk. For the Erd\"os-Renyi simulation, the score matching estimator also selects a sparser graph, though it has risk slightly worse than the MLE.
\begin{figure}
\centering
\includegraphics[scale=.5]{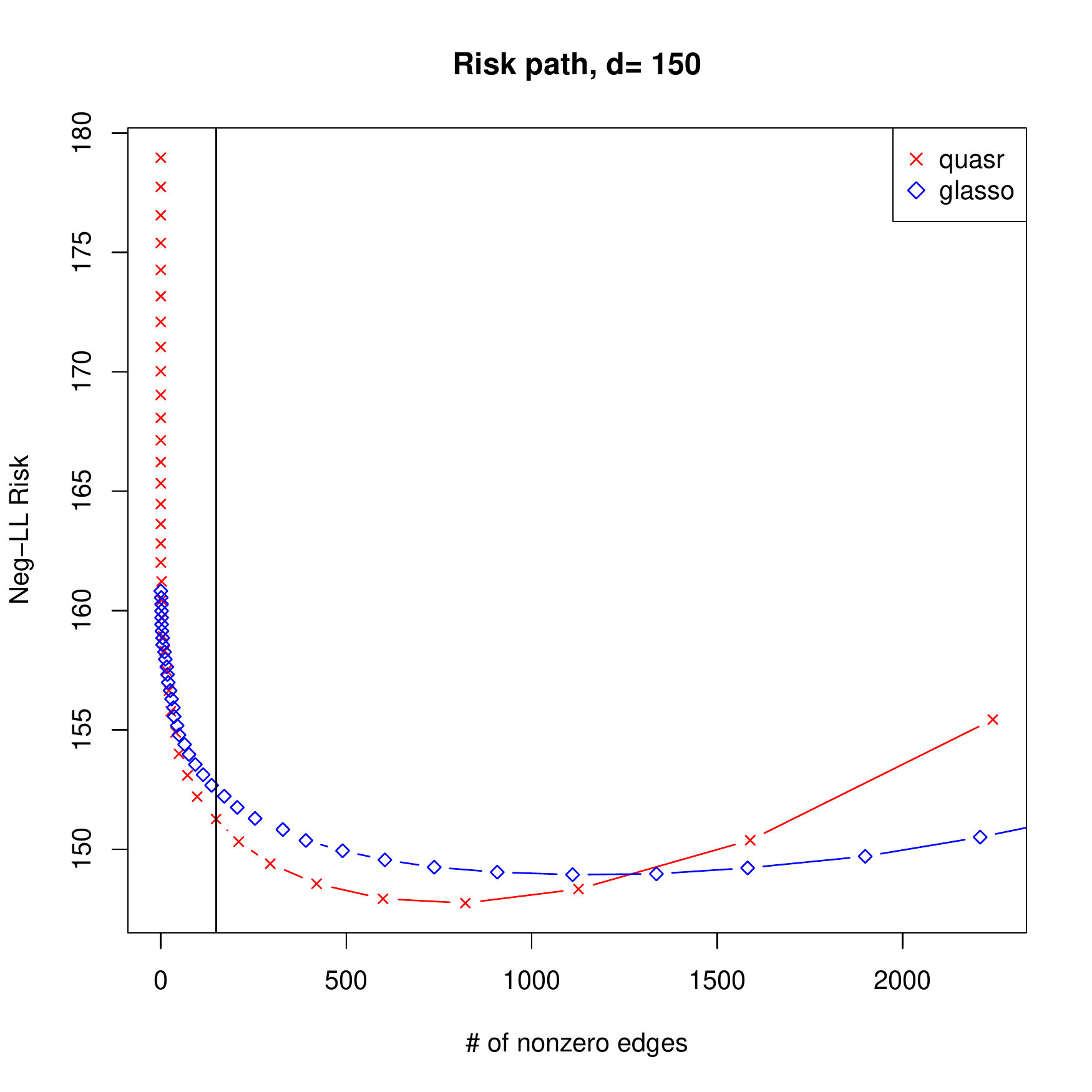}
\caption{Risk path, tree graph} \label{riskpath1}
\end{figure}

\begin{figure}
\centering
\includegraphics[scale=.5]{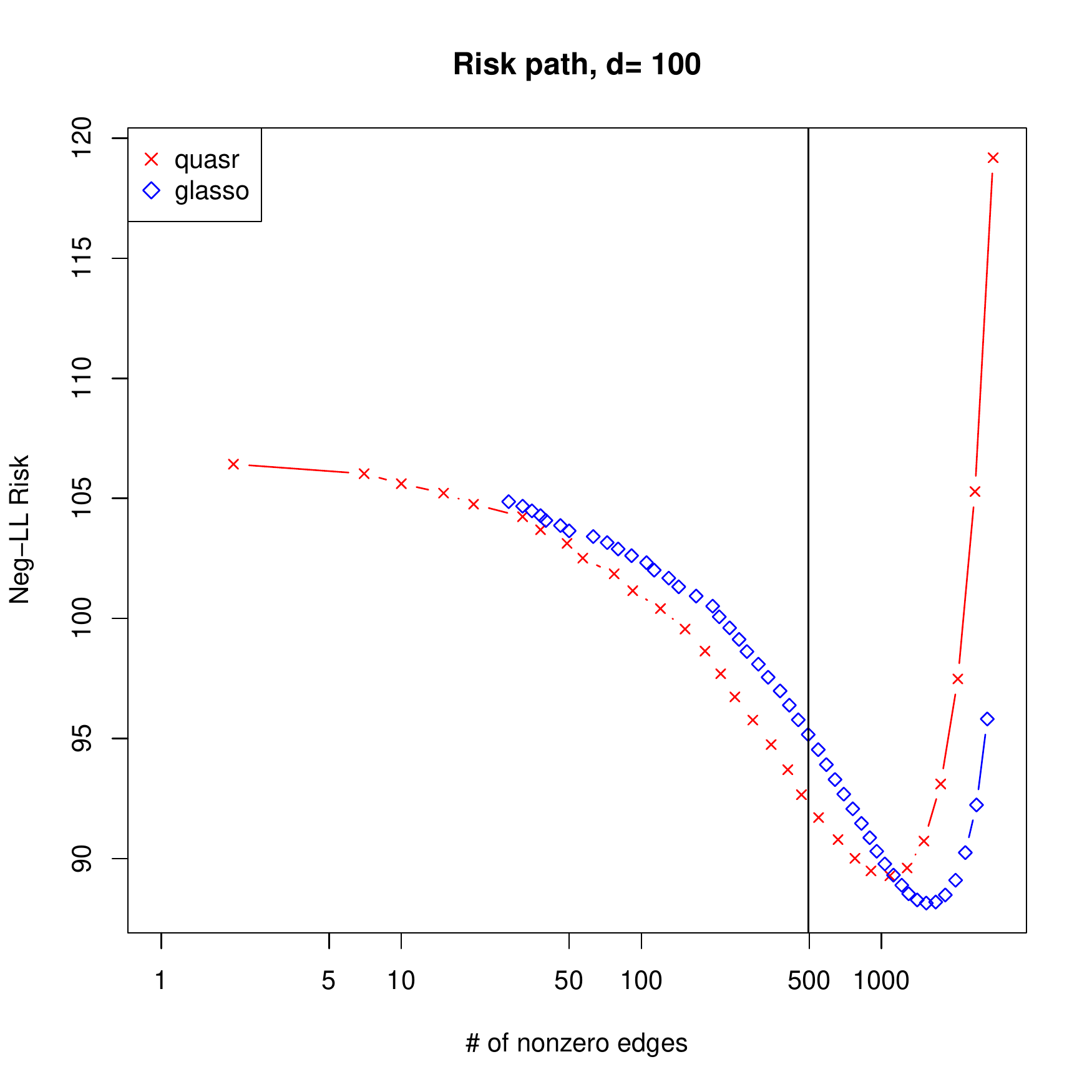}
\caption{Risk path, ER graph} \label{riskpath2}
\end{figure}
These findings are also validated in Tables \ref{table1} and \ref{table2}, varying $d$. For the tree graph, score matching dominates in risk for all values of $d$. Even for the Erd\"os-Renyi graph, score matching outperforms the regularized MLE when $d=30$.  Standard errors of 25 repetitions are in parentheses.
\begin{table}[ht]
\centering
\begin{tabular}{lll}
  \hline
 & quasr & glasso  \\
  \hline
d= 30 & 28.082 (0.341) & 28.189 (0.319)  \\
  d= 75 & 72.785 (0.472) & 73.202 (0.432)  \\
  d= 120 & 116.771 (0.578) & 117.631 (0.542) \\
  d= 150 & 147.26 (0.619) & 148.381 (0.583)  \\
   \hline
\end{tabular}
\caption{Held-out NLL error, tree graph} \label{table1}
\end{table}
\begin{table}[ht]
\centering
\begin{tabular}{lll}
  \hline
 & quasr & glasso  \\
  \hline
d= 30 & 26.157 (0.428) & 26.244 (0.386)  \\
  d= 50 & 44.351 (0.688) & 44.262 (0.718)  \\
  d= 100 & 91.383 (1.009) & 90.351 (1.077)  \\
  d= 150 & 139.03 (1.15) & 136.793 (1.336)  \\
   \hline
\end{tabular}
\caption{Held-out NLL error, ER graph} \label{table2}
\end{table}

\subsubsection{ROC Curves and Edge Selection}

Figures \ref{roc1} and \ref{roc2} show ROC curves under the same simulation setup in the previous section. The plotted points represent the graph selected from the minimal held-out risk in each of the 25 repetitions.
The two estimators display very similar ROC curves, and the score matching estimator tends to prefer higher sensitivity for lower specificity, when selecting using held-out data.

\begin{figure}
\centering
\includegraphics[scale=.5]{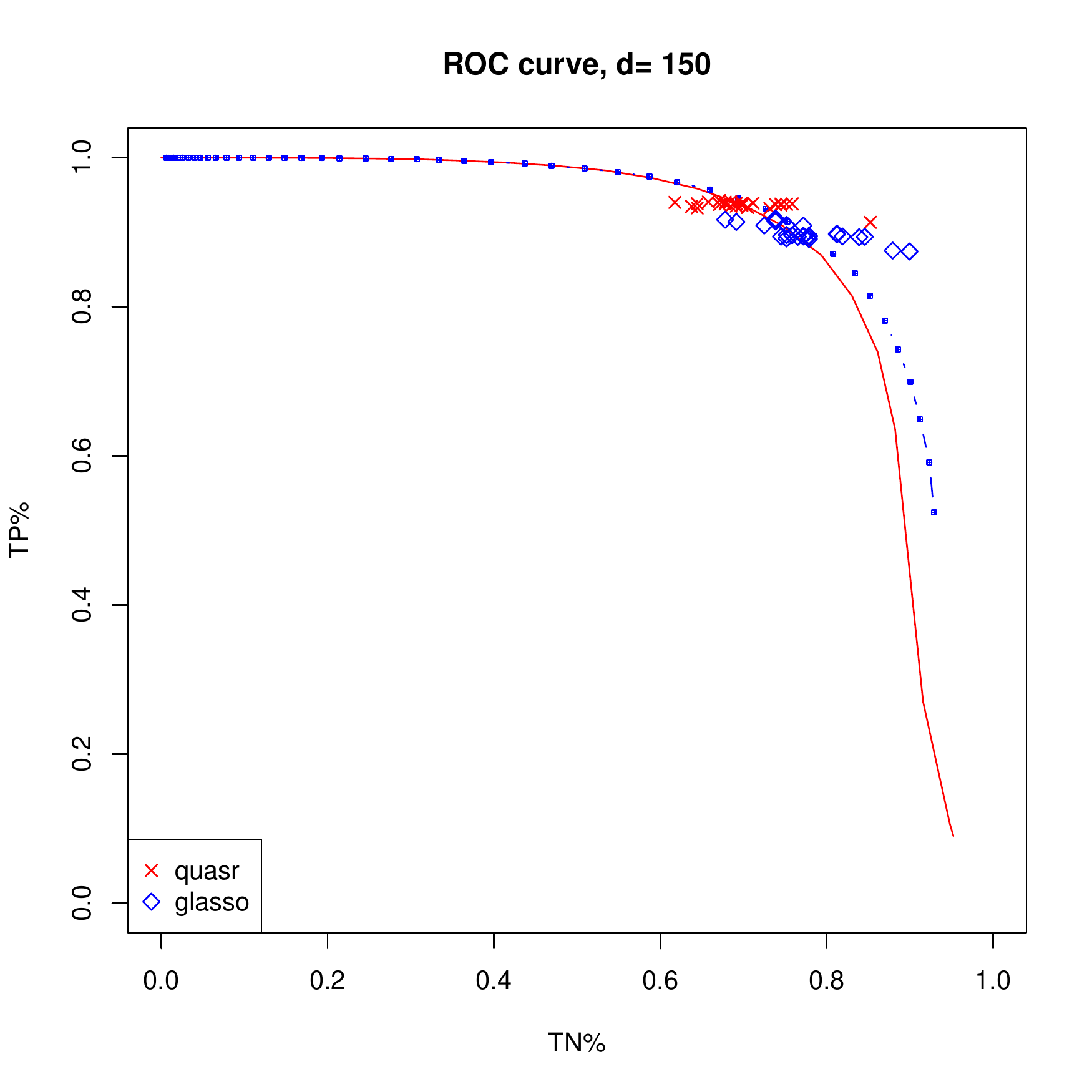}
\caption{ROC curve, tree graph. Solid line: Gaussian QUASR; dotted line: glasso.} \label{roc1}
\end{figure}

\begin{figure}
\centering
\includegraphics[scale=.5]{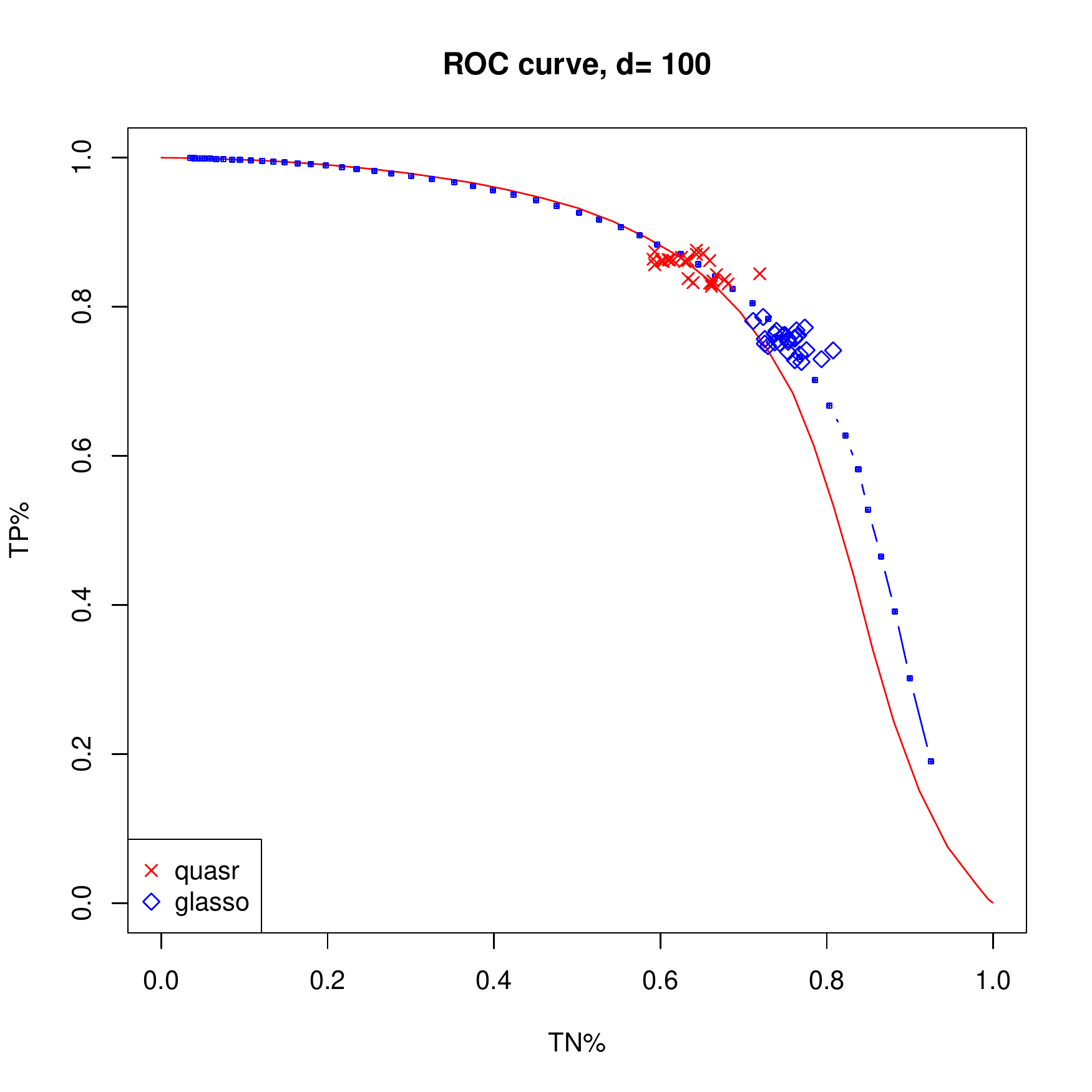}
\caption{ROC curve, Erdos-Renyi graph. Solid line: Gaussian QUASR; dotted line: glasso.} \label{roc2}
\end{figure}
Tables \ref{edgesel1} and \ref{edgesel2} displays true positive and true negative rates for varing choices of $d$, fixing $n=100$. The parentheses are the standard deviation for 25 repetitions of the experiment. As are suggested by the ROC curves, score matching prefers a higher sensitivity and lower specificity to the regularized MLE, and for simulations when $d$ is relatively small compared to $n$, score matching has a significantly higher true positive rate with only negligible reduction in true negative rate.

\begin{table}[ht]
\centering
\begin{tabular}{lll}
  \hline
 & quasr & glasso  \\
  \hline
d=30 \%TN & 0.941 (0.043) & 0.975 (0.025)  \\
  d=30 \%TP & 0.778 (0.051) & 0.674 (0.042)  \\
  d=75 \%TN & 0.821 (0.037) & 0.876 (0.042)  \\
  d=75 \%TP & 0.892 (0.017) & 0.832 (0.017)  \\
  d=120 \%TN & 0.751 (0.052) & 0.816 (0.044)  \\
  d=120 \%TP & 0.922 (0.01) & 0.874 (0.009)  \\
  d=150 \%TN & 0.699 (0.049) & 0.776 (0.052)  \\
  d=150 \%TP & 0.936 (0.006) & 0.899 (0.012)  \\
   \hline
\end{tabular}
\caption{Edge selection accuracy, tree graph.} \label{edgesel1}
\end{table}

\begin{table}[ht]
\centering
\begin{tabular}{lll}
  \hline
 & quasr & glasso  \\
  \hline
d=30 \%TN & 0.969 (0.028) & 0.986 (0.02)  \\
  d=30 \%TP & 0.743 (0.036) & 0.612 (0.039) \\
  d=50 \%TN & 0.878 (0.028) & 0.927 (0.025)  \\
  d=50 \%TP & 0.785 (0.027) & 0.673 (0.032)  \\
  d=100 \%TN & 0.633 (0.029) & 0.754 (0.02)  \\
  d=100 \%TP & 0.858 (0.013) & 0.753 (0.013)  \\
  d=150 \%TN & 0.486 (0.019) & 0.634 (0.018)  \\
  d=150 \%TP & 0.892 (0.01) & 0.791 (0.011)  \\
   \hline
\end{tabular}
\caption{Edge selection accuracy, ER graph} \label{edgesel2}
\end{table}

\subsubsection{Computation}

In Figure \ref{cputree} we compare the runtime of our algorithm with glasso. We simulate random Gaussian tree data with $n=100$, $d=50$. We fit over a path of $\lambda$s and plot runtime against number of selected edges. More regularization results in sparser graphs, and so convergence is faster. In this experiment our method is much faster than glasso, sometimes by a factor of 4 or more. The gap narrows for sparse estimated graphs. This is because while our algorithm is written efficiently in $\tt{C++}$, it doesn't (yet) utilize sparse matrix libraries, while glasso does. Since our coordinate-wise descent algorithm involves sparse inner products, we believe our runtimes can be improved in the sparse regime.

\begin{figure}
\centering
\includegraphics[scale=.9]{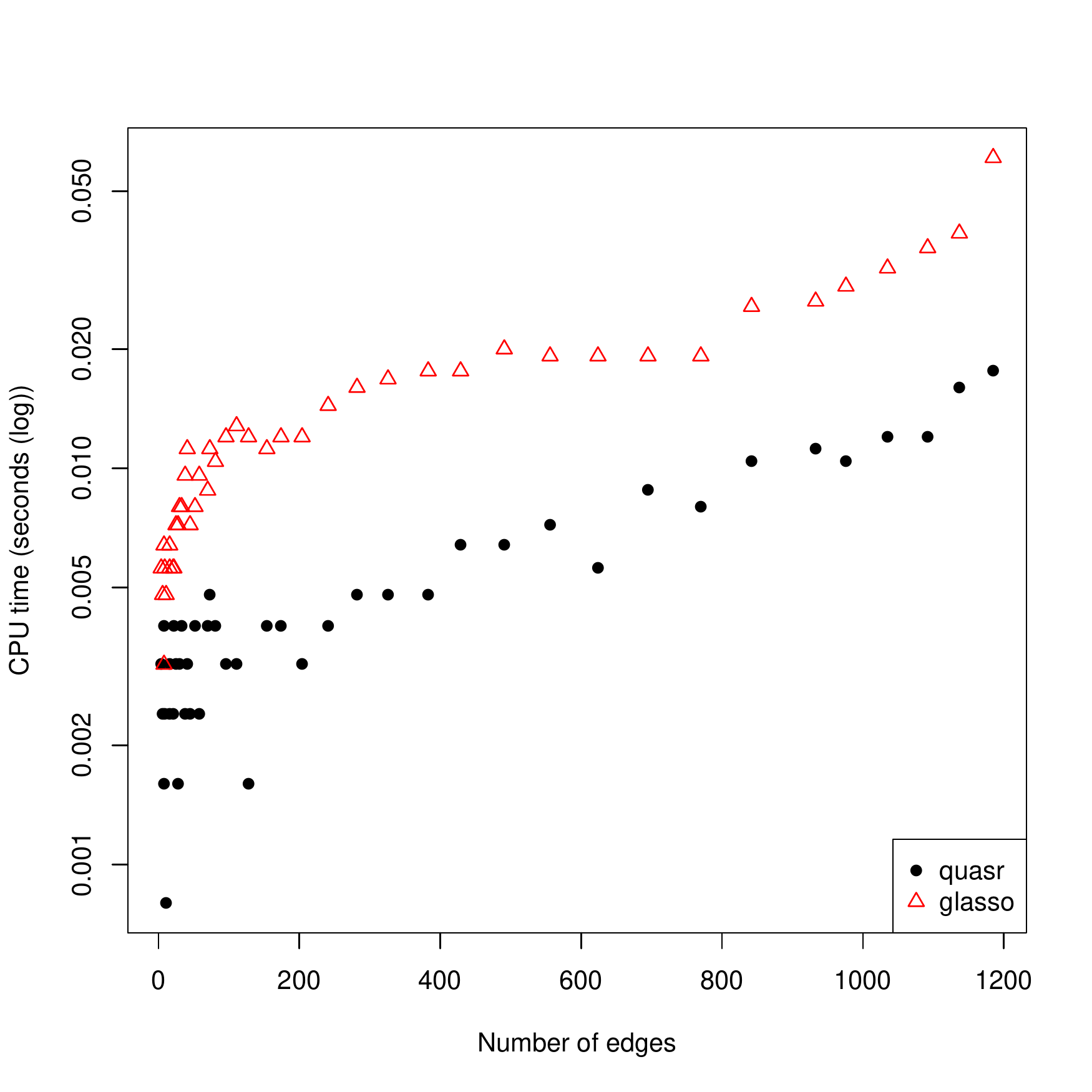}
\caption{Runtime, Gaussian QUASR and glasso. Gaussian data, n=100, d=50}. \label{cputree}
\end{figure}
\subsection{Nonparametric Score Matching}

\subsubsection{Risk and Density Estimation}
In this section we compare score matching and MLE when the sufficient statistics are chosen to be Legendre polynomials. 
For each choice of $n$, we simulate non-Gaussian data whose density factorizes as a tree. We do this by marginally applying the transformation $y=\text{sign}(x-0.5) \mid x - 0.5 \mid ^{0.6}/5 + 0.5$ to Gaussian data which has a tree factorization, which has been scaled to have means 0.5 and covariance $1/8^2$, so it fits in the unit cube; the resulting data follows a Gaussian copula distribution, which is also a pairwise distribution. We train the model using both regularized MLE and score matching under constraint that it factorizes according to the given tree, and we  estimate along a path of $\lambda$s and choose the regularization parameter $\lambda$ to minimize the held-out risk. Since the density has a (known) tree factorization, it is possible to compute the likelihood (and hence the MLE) exactly using functional message passing (Chapter 3). It is also possible to compute the marginals using the same message passing algorithm.

Figure \ref{nprisk} displays the held-out risk for both methods, with sample size varying from 24 to 5000 and $d=20$. We average over 5 replications of the experiment. The MLE outperforms the score matching estimator, but the score matching estimators performance greatly improves relative to the MLE as $n$ increases. 

Figure \ref{npcontour} displays contours from one bivariate marginal from the aforementioned simulation. The top row shows the MLE and score matching estimator for $n=24$, and the bottom for $n=182$. For $n=24$, the score matching marginal can make out  much of the distinguishing features of the density such as the multiple modes, but isn't as informative as the MLE. At $n=182$ the marginals appear almost identical.
\begin{figure}
\centering
\includegraphics[scale=.9]{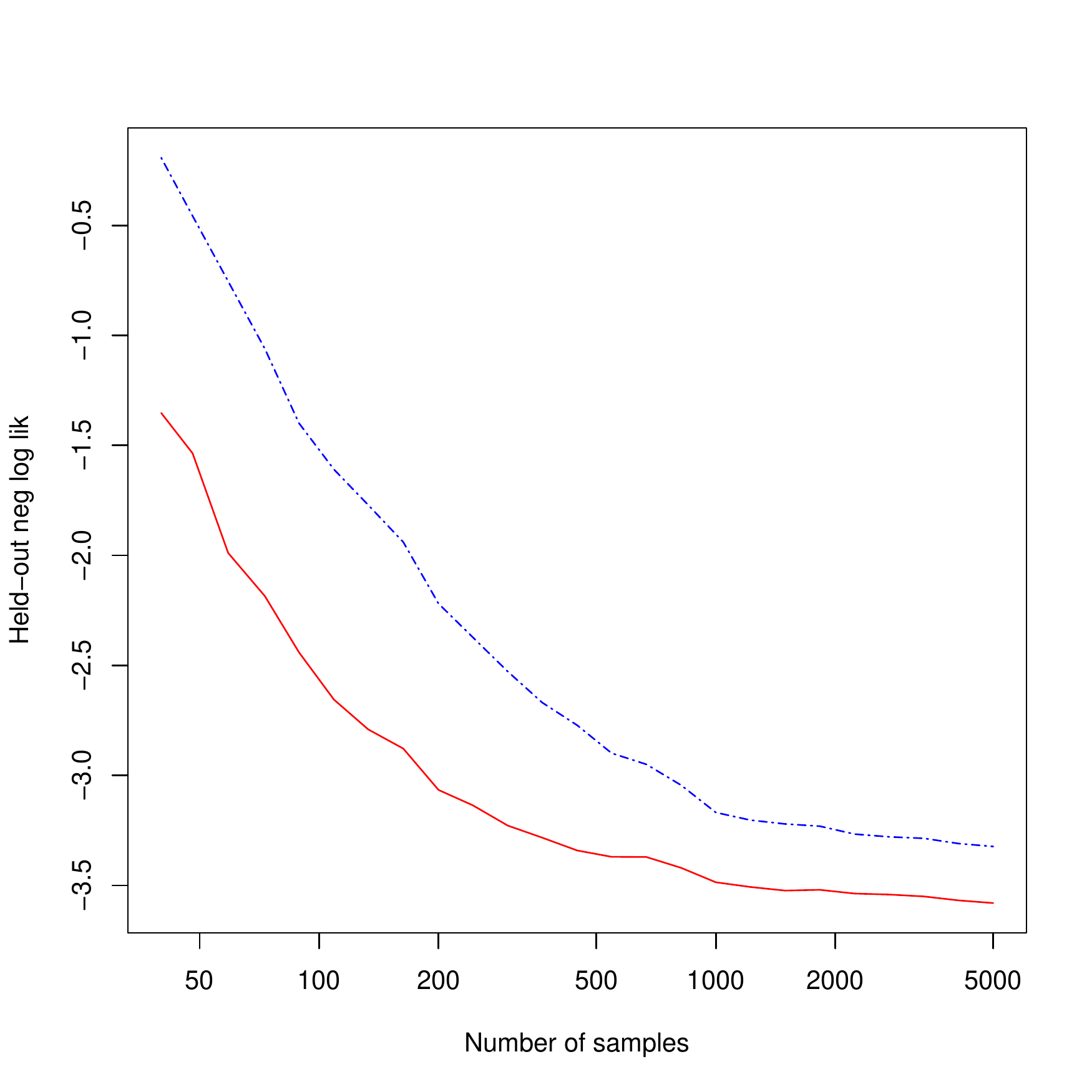}
\caption{Held-out negative log likelihood for different training sizes. Red: regularized MLE; Blue dashed: score matching. Data generated from a non-Gaussian tree distribution.} \label{nprisk}
\end{figure}
We conclude that while MLE is more efficient as may be expected, score matching performs quite well, especially with larger sample sizes. 
\begin{figure}
\centering
\includegraphics[page=1,scale=.4]{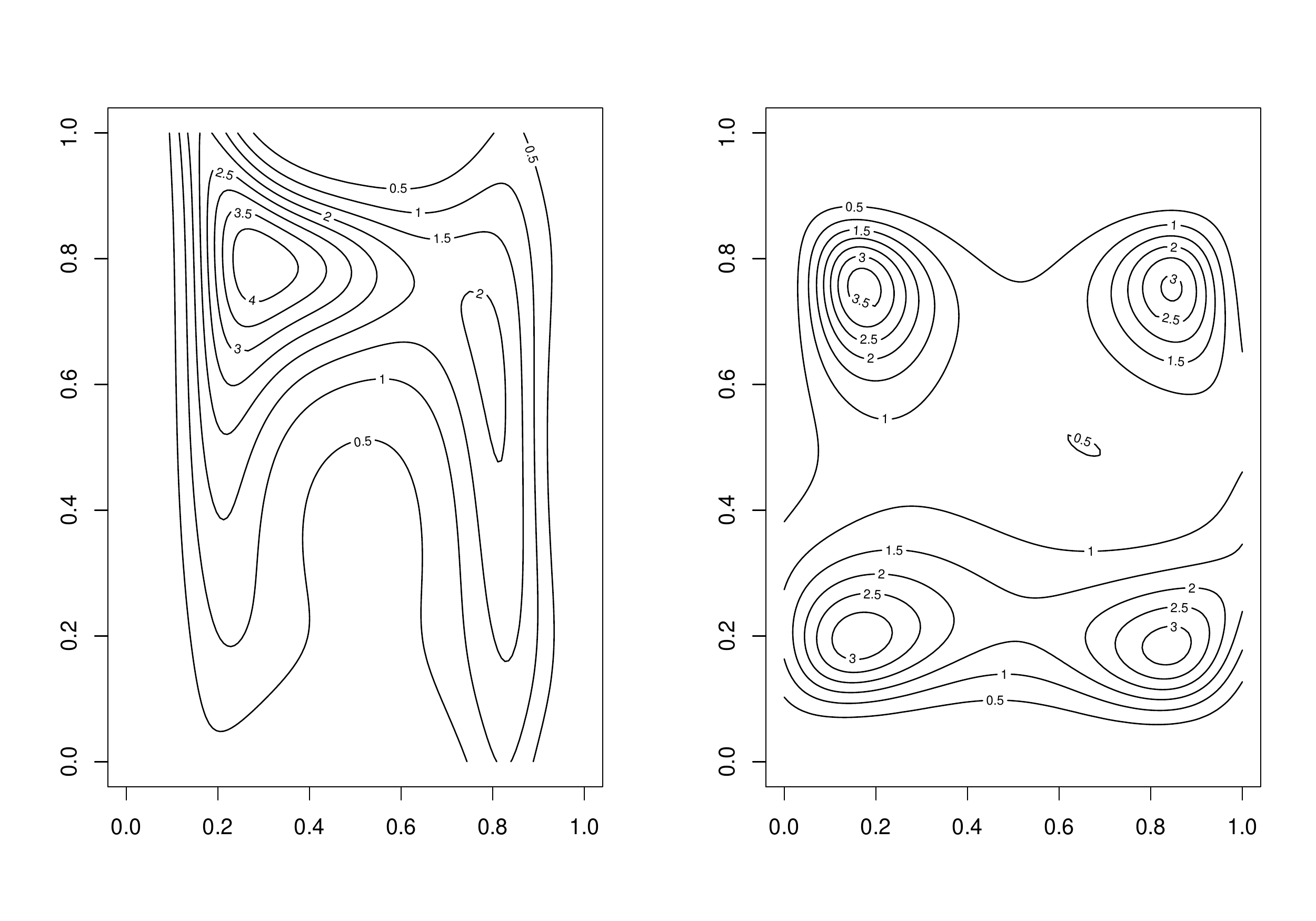}
\includegraphics[page=2,scale=.4]{bivdenrisk.pdf}
\includegraphics[page=3,scale=.4]{bivdenrisk.pdf}

\caption{Estimated bivariate marginal from non-Gaussian tree data, d=20. Left: regularized MLE; Right: Score matching. Top: n=40. Middle: n=244 Bottom: n=1828.} \label{npcontour}
\end{figure}
Furthermore, we emphasize that at this cost of statistical efficiency, score matching can be computed easily under any graph structure, even when the likelihood is not tractable, while MLE is typically not tractable and must be approximated.

\subsubsection{ROC Curves}

Figures \ref{rocquasr1} and \ref{rocquasr3} display ROC curves from four experiments. We simulate data with $d=20$ and $n$ either 30 or 100. In the first two experiments, the data is Gaussian; in the last two, it is non-Gaussian (copula). The graph is either a random spanning tree with $d-1$ edges, or a graph with each possible edge having inclusion probability $0.2$, or expected number of edges $d(d-1)*.1$. The ROC curves trace the true positive and true negative rates, varying the value of $\lambda$. The curves are averaged over 10 repetitions of the experiment (with the data i.i.d. from the same distribution). We choose $m_1,m_2$ to maximize $\max_\lambda TP(\lambda)+TN(\lambda)$. We compare our method to the SKEPTIC estimator from \citep{liu2012high},  which was designed in particular for model selection for copula graphical models, and the TRW estimator from Chapter 3. Our experiments show our method to do either just as well, or only slightly worse than competing methods.
\begin{figure}
\centering
\includegraphics[page=1,scale=.5]{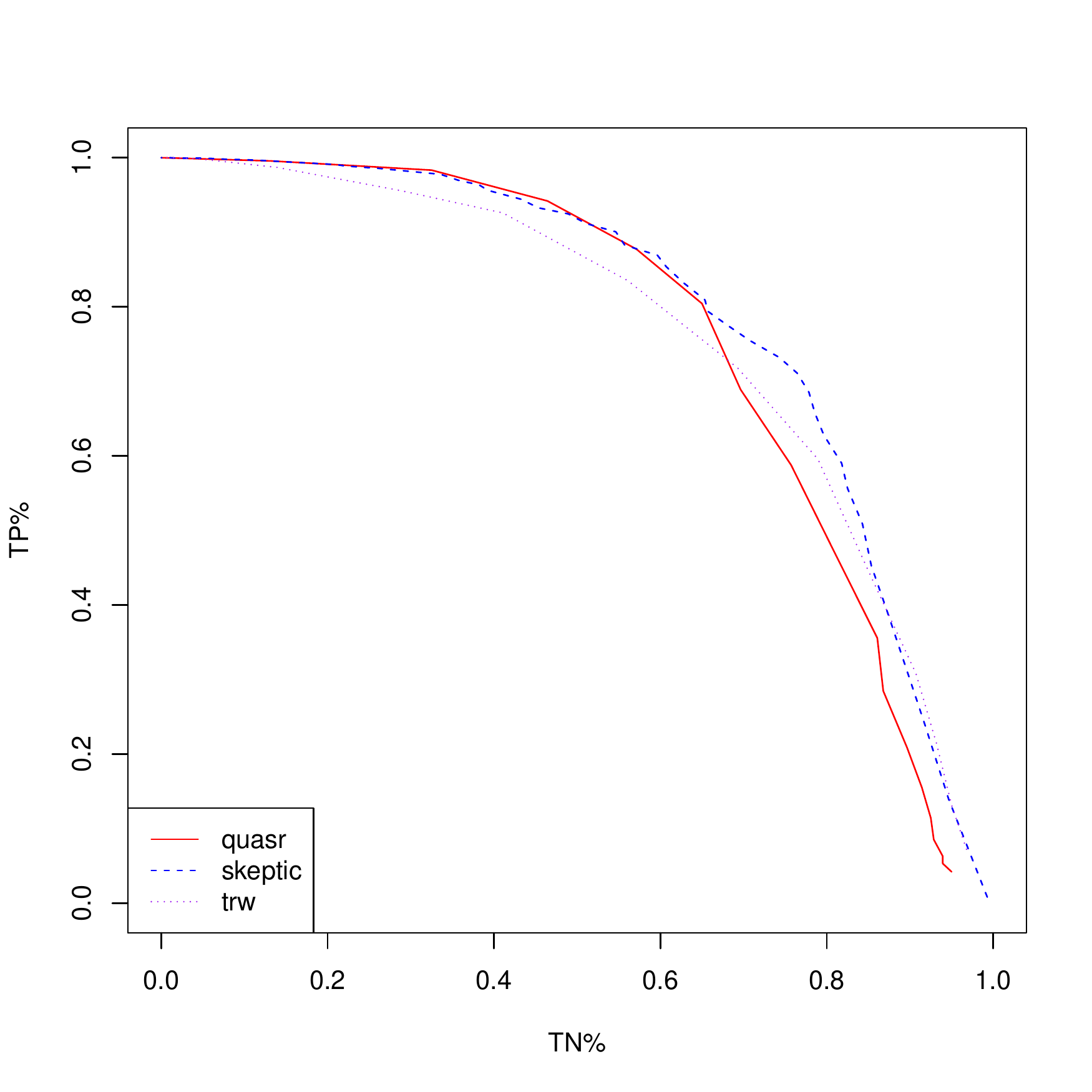}
\includegraphics[page=1,scale=.5]{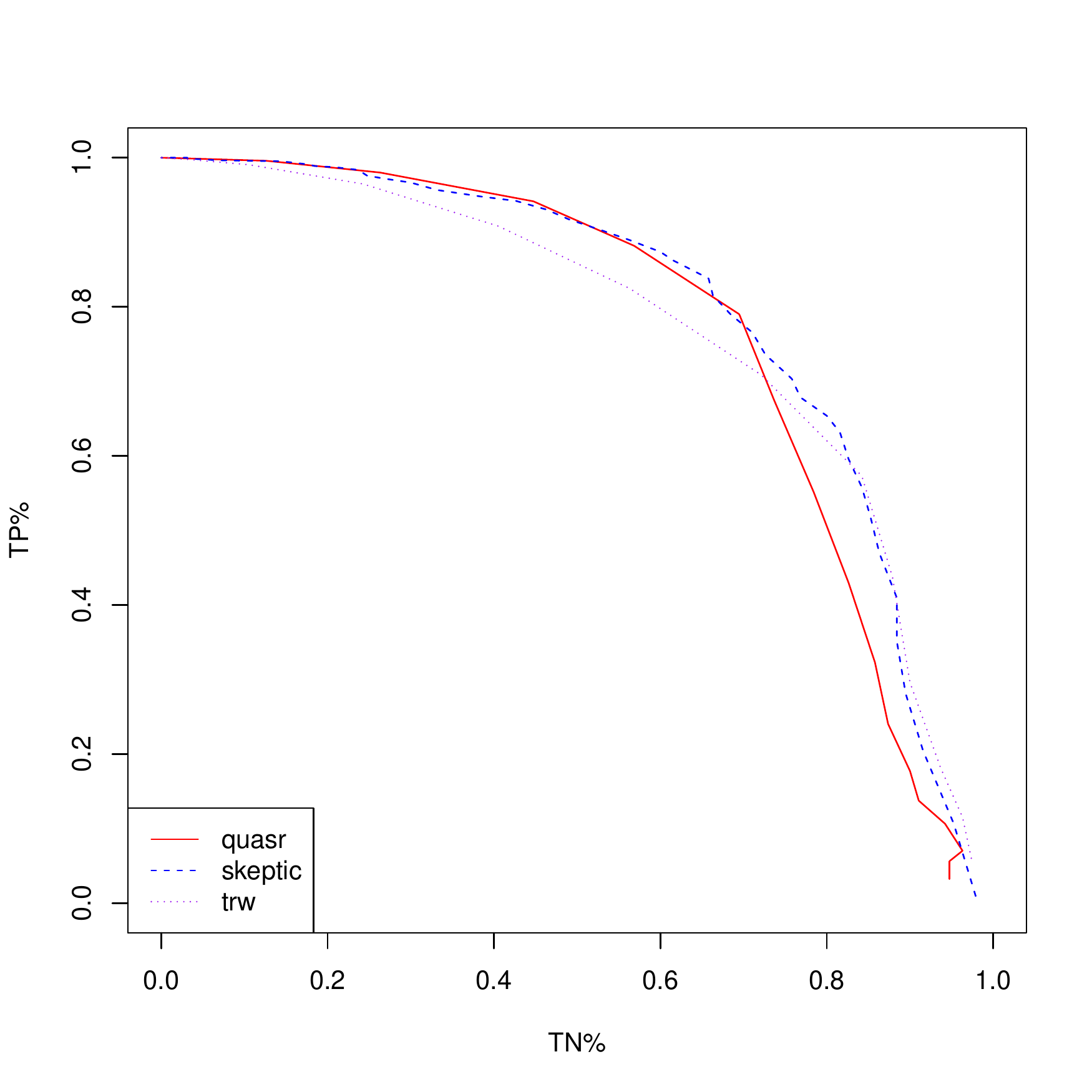}

\caption{ROC Curve, Gaussian data, n=30, d=20. Top: ER graph. Bottom: Tree graph.}\label{rocquasr1}
\end{figure}


\begin{figure}
\centering
\includegraphics[page=1,scale=.5]{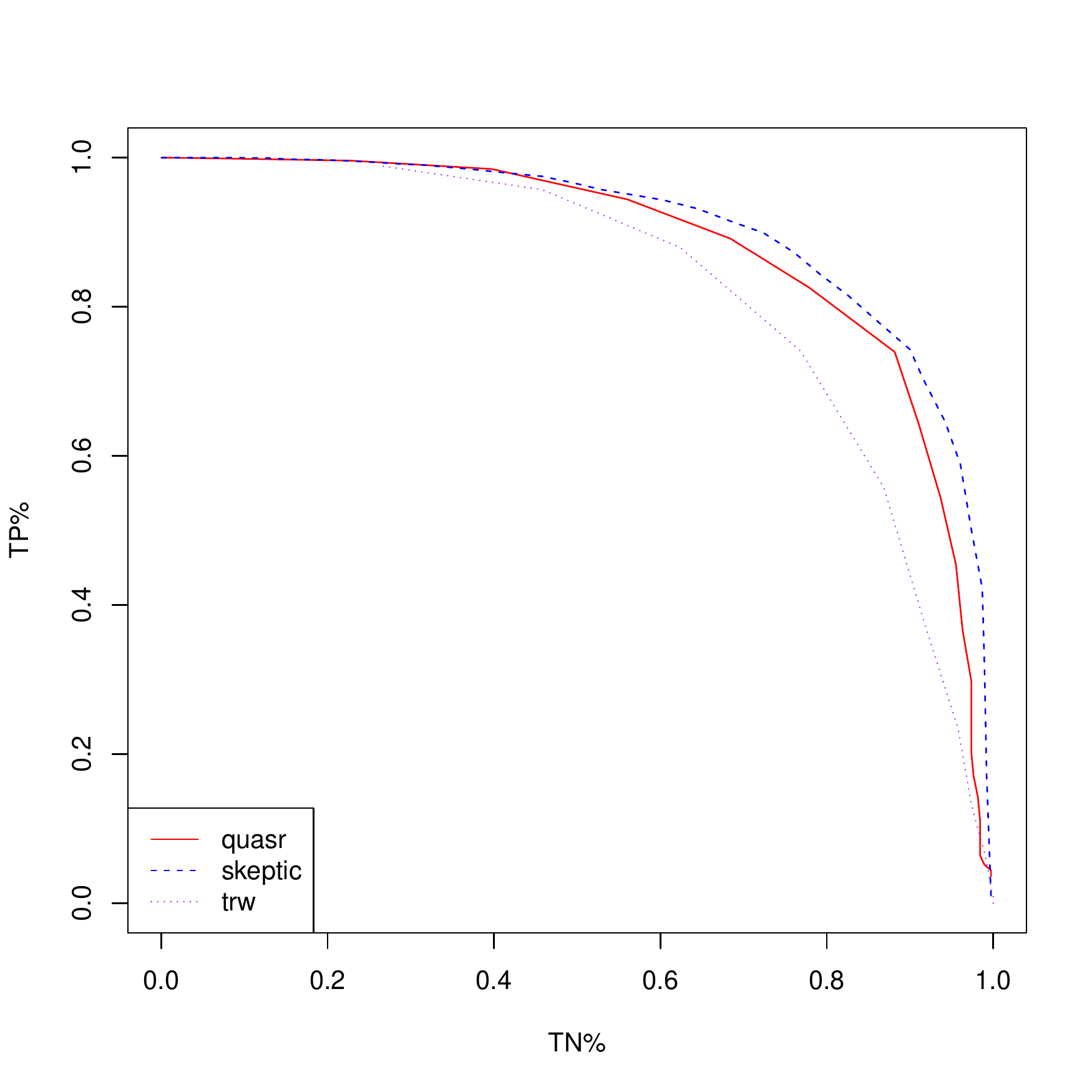}
\includegraphics[page=1,scale=.5]{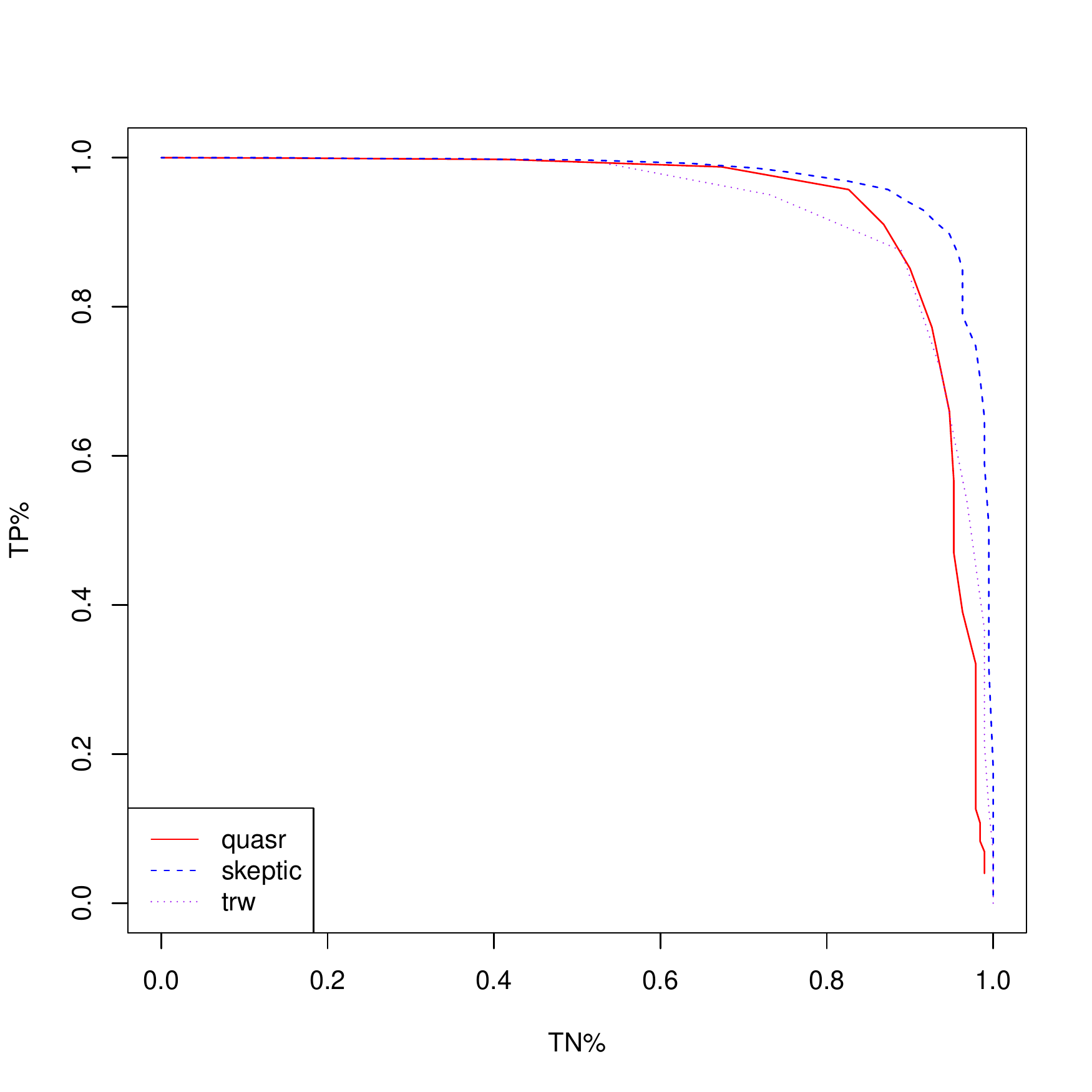}
\caption{ROC Curve, non-Gaussian data, n=100, d=20. Top: ER graph. Bottom: Tree graph.}\label{rocquasr3}
\end{figure}


\section{Discussion}

This chapter introduces a new approach to estimating pairwise, continuous, exponential family graphical models. Since the normalizing constant for these models is usually intractable, we propose a new scoring rule which obviates the need for it. Our resulting estimator may be expressed as a second-order cone program. We show consistency and edge selection results for this estimator, including as special cases a new method for precision matrix estimation, and for nonparametric edge selection with exponential series. We propose algorithms for solving the convex problem which are highly scalable and amenable to parallelization. This method has good experimental performance compared to other works in the literature.

\chapter{Conclusions}

This thesis proposes a new framework for density estimation, inference and structure learning for the nonparametric pairwise undirected graphical model. We consider approximating the log density using a truncated basis expansion, which results in a finite-dimensional exponential family. We consider two estimation approaches. The first is regularized maximum likelihood, for which we provide a variational approximation method. The second is a new method for estimation and graph learning of exponential families based on the scoring rule of \citep{hyvarinen2005estimation}. We show that score matching allows for provably consistent parameter estimation and structure learning for exponential families, despite exact inference for this class of densities being intractable. As a special case we derive a new method for sparse precision matrix estimation, which performs competitively in experiments. We also derive results for the exponential series approximation, and show its performance in experiments. 

This thesis contributes to two strains of literature. The first is that of pairwise density estimation \citep{gu2002smoothing,jeon2006effective,forest2011}. Our approach is novel in that we use exponential series estimators, rather than Mercer kernels or kernel density estimators. We also introduce a regularization approach to edge selection for learning sparse pairwise models. We show the regularized MLE estimator adapts to the unknown sparsity of the oracle pairwise density. In Chapter 3 we propose a variational approach to approximating the log-likelihood for nonparametric pairwise models; this gives an upper bound to the negative log-likelihood risk for an estimator and also gives a tractable algorithm for estimating pairwise models.

The second contribution is to the literature on high-dimensional graph selection. The literature focuses overwhelmingly on parametric models such as the Gaussian graphical model, the Ising model, or other parametric models with special structure \citep{yang2012graphical}. We contribute to this literature in two ways. Firstly, our QUASR estimator works for exponential families broadly. It allows to both have consistent parameter estimation as well as edge selection, even for exponential families which are non-normalizable. Second, applying the QUASR method to exponential series gives a method for model selection for the nonparametric pairwise graphical model. There has been recent effort to expand from parametric models to nonparametric or semiparametric graphical models, such as Gaussian copulas \citep{liu2012high} and forests \citep{forest2011}. Our approach is the first to address learning the fully nonparametric pairwise model and to demonstrate a tractable method with statistical guarantees for model selection. In addition to our theoretical and methodological contributions, we show the TRW and QUASR estimators perform well in practice, particularly compared to previous methods, and our algorithms are scalable and can take advantage of parallelization.

There still remains more work to better understand the score matching estimator and its connection to the widely-used maximum likelihood estimator. While our experiments are a start, a more detailed study comparing the assumptions of the theoretical results is warranted. It would be valuable to conduct a parallel theoretical analysis like what was done for the Danzig selector and lasso in \citep{bickel2009simultaneous}. Also, the theoretical guarantees for the score matching approach are generally weaker than for MLE; it would be interesting to find the optimal rates for estimation and sample complexity for edge selection in the computation-limited setting (i.e., restricting to algorithms, which can be computed up so some level of accuracy in polynomial time, which excludes MLE), and given that, find an algorithm which achieves those optimal rates, if indeed the score matching approach is sub-optimal.

Our work is not the final word on the subject. There are other fruitful paths for deriving good approximations or alternative estimators for exponential families with continuous potentials and for the nonparametric pairwise model, such as using semidefinite relaxations \citep{lasserre2007semidefinite}, or in the case of polynomial sufficient statistics, approximating the log density as a sum-of-squares polynomial, which can be normalized efficiently. It would also be of interest whether proper scoring rules \citep{dawid2014theory} beyond the Hyv\"arinen score would be useful for estimating these models. In Chapter 4 we show that the nonparametric pairwise estimator of \citep{jeon2006effective} optimizes the \emph{Bregman score}, which is a proper scoring rule. Exploring these connections more may give rise to more provably consistent estimators or other sound approximations or variational approaches to these problems. Also, there is more work that needs to be done to better understand the behavior of variational techniques. We hope this thesis has proven the value of the exponential series approach to nonparametric estimation, and we hope our contributions are useful tools for analysis of complex, high-dimensional data sets.

\appendix

\chapter{Proofs for Chapter 2}\label{proofs}


\begin{proof}[Proof of Lemma \ref{uniqueness}]
Observe that if $\lambda_n>0$, \eqref{mlsol} may alternatively be written as
\begin{align}
\underset{\theta\in\mathcal{P}:\mathcal{R}^*(\theta_e)\leq C}{\text{argmin}} \bigg\{ -\mathcal{L}(\theta) + \lambda\mathcal{R}(\theta)\bigg\},
\end{align}

for some $C<\infty$. Thus the edge parameters of the solution are bounded. The only question is whether the objective approaches $-\infty$ as $\Vert\theta_v\Vert$ diverges to infinity. Examine the first order condition for $\theta_v$:
\begin{align}
	\widehat{\mu}_v = \mu(\widehat{\theta})_v,
\end{align}

For any fixed $\widehat{\theta}_e$, the inverse of this equation produces a unique $\widehat{\theta}$ so long as $\widehat{\mu}_v$ belongs to the interior of the mean space, due to \eqref{onto}. For bases following the Haar condition, it has been shown that this is satisfied with probability one when $n>m_1$ \citep{crain1976exponential}. Furthermore, due to the strong convexity of the objective \eqref{strongconv}, this solution will be unique when it exists.
\end{proof}

Let $\widehat{p}=p_{\widehat{\theta}}$, where $\widehat{\theta}$ is the solution to \eqref{mlsol}. Let $\tilde{\theta}$ be any solution to the population minimizer
	\begin{equation}
		\tilde{\theta} = \underset{\theta\in\tilde{\mathcal{P}}(E)}{\mbox{argmin}} \text{KL}(p \mid p_\theta). \label{mlsol2}
	\end{equation}
$p_{\tilde{\theta}}$ is known as the \emph{information projection} \citep{csiszar1975divergence} and is the density satisfying $\mathbb{E}_{p_{\tilde{\theta}}}[\bar{\phi}_E]=\mathbb{E}_p[\bar{\phi}_E]$. Note that in particular, $\tilde{\theta}_{E^c}=0$, so $p_{\tilde{\theta}}$ and $p_{\theta^*}$ have the same sparsity pattern. Indeed, since $\{\bar{\phi}\}$ is a minimal family, $\tilde{\theta}$ is unique. For an information projection we have a Pythagorean theorem, which we present below.

\begin{lem}[Pythagorean Theorem for KL divergence] \label{pyth}

Suppose $p_\theta$ belongs to the finite exponential family with sufficient statistics $\{\bar{\phi}\}$. Then 
\begin{align}
	\text{KL}(p \mid p_\theta) &= \text{KL}(p \mid p_{\tilde{\theta}}) + \text{KL}(p_{\tilde{\theta}} \mid p_\theta).
\end{align}

\end{lem}
\begin{proof}
Notice that $\log(p/p_\theta) = \log(p/p_{\tilde{\theta}})+\log(p_{\tilde{\theta}}/p_\theta)$. Taking expectations of both sides with respect to $p$ gives
\begin{equation}
	\text{KL}(p\mid p_\theta) = \text{KL}(p \mid p_{\tilde{\theta}}) + \mathbb{E}_p(\log(p_{\tilde{\theta}}/p_\theta))
\end{equation}

From the moment-matching properties of the information projection, we have  
\begin{align}
\mathbb{E}_p(\log(p_{\tilde{\theta}}/p_\theta))=\mathbb{E}_{p_{\tilde{\theta}}}(\log(p_{\tilde{\theta}}/p_\theta))=\text{KL}(p_{\tilde{\theta}} \mid p_{\theta}).
\end{align}
The result follows.
\end{proof}
This lemma implies that $\text{KL}(p \mid \widehat{p})$ can be split into two terms: the approximation error and estimation error. We start by analyzing the approximation error.

\section{Approximation Error}

\begin{lem}[\citep{barron1991approximation}, Lemma 1] \label{klupper}
	Let $p,q$ be two densities with respect to the measure $\nu$. Then
	\begin{align}
		\text{KL}(p \mid q) \leq e^{\Vert\log p/q - c \Vert_\infty}\intop p \left(\log p/q - c \right)^2,
	\end{align}
	for any constant $c$.
\end{lem}
\begin{proof}
	Using the Taylor expansion of $e^x$, we get the bound
	\begin{align}
		e^{x} - 1 - x \leq \frac{x^2}{2}e^{x_+},
	\end{align}
	where $x_+ = \max(0,x)$. Set $h(x)= \log p(x)/q(x)-c$. Then
	\begin{align}
		\intop p \log p/q &= \intop \left( p\log p/q + qe^c - p -c \right) -e^c +1 +c \\
			&\leq \intop p \left( h + e^{-h} - 1 \right) \\
			& \leq \intop p \left( \frac{h^2 e^{h_+}}{2} \right) \\
			& \leq \frac{e^{\Vert h \Vert_\infty}}{2}\intop p h^2.
	\end{align}
\end{proof}

%
%
%

%

Consider the linear space $S_m := S_{m_1,m_2}$ of functions spanned by the truncated basis elements $\{\bar{\phi} \}$. Such functions need not be the logarithm of a valid density.
\begin{thm} \label{klapprox}
For $f=\log p$ let
\begin{align}
	\Delta_m &= \Vert f - f_m \Vert_{L_2(p)}, \\
	\gamma_m &= \Vert f - f_m \Vert_\infty,
\end{align}
be the $L_2(p)$ and $L_\infty$ approximation errors for some $f_m \in S_m$. Then the information projection $p_{\tilde{\theta}}$ satisfies
\begin{equation}
	\text{KL}(p \mid p_{\tilde{\theta}}) \leq  \frac{1}{2}e^{2\gamma_m}\Delta_m^2.
\end{equation}

Thus, if $\gamma_m$ is bounded,
\begin{align}
	\text{KL}(p \mid p_{\tilde{\theta}}) = O(\Delta_m^2).
\end{align}
\end{thm}

\begin{proof}
Let $f_m = \langle\bar{\theta}^*,\bar{\phi}\rangle + \bar{\theta}^*_0$ be the approximation of $f$ presumed to satisfy the given bounds on the error $f-f_m$. Define a density $p_{\bar{\theta}^*}$ as $\log p_{\bar{\theta}^*} = \langle\bar{\theta}^*,\bar{\phi}\rangle - Z(\bar{\theta}^*)$ and denote $f_{\bar{\theta}^*} = \log p_{\bar{\theta}^*}$. Using the bound from $\ref{klupper}$ and setting $c=\theta^*_0 + Z(\bar{\theta}^*)$, we have that 

\begin{align}
	\text{KL}(p \mid p_{\bar{\theta}^*}) & \leq \frac{e^{\Vert f - f_{\bar{\theta}} - \theta^*_0- Z(\bar{\theta}) \Vert_\infty}}{2}\intop p \left( f - f_{\bar{\theta}} - \theta^*_0 - Z(\bar{\theta}^*) \right)^2 \\
	& \leq \frac{e^{2\Vert f-f_m \Vert_\infty}}{2}\Vert f-f_m\Vert_{L_2(p)}^2 \\
	& = \frac{e^{2\gamma_m}}{2}\Delta_m^2.
\end{align}

Since $\tilde{\theta}$ achieves the minimum $\text{KL}$-risk, we have
\begin{align}
	\text{KL}(p \mid p_{\bar{\theta}}) \leq \text{KL}(p \mid p_{\bar{\theta}^*}) \leq  \frac{e^{2\gamma_m}}{2}\Delta_m^2 .
\end{align}
\end{proof}

\begin{lem}\label{klapprox2}
The $L_2(p)$ and $L_\infty$ approximations to $f=\log p$ satisfy
\begin{align}
	\Delta_m^2 &= O\left(dm_1^{-2r}+\left| E\right| m_2^{-2r}\right), \\
	\gamma_m &= O\left(dm_1^{-r+\alpha+1/2}+\left| E\right|m_2^{-r+2\alpha+1/2}\right),
\end{align}

and the truncated variables $\theta^*_T$ satisfy
\begin{align}
	\Vert \theta^*_T\Vert_\infty = o(\max\{m_1^{-r-1/2},m_2^{-r-1/2}\}).
\end{align}
\end{lem}
\begin{proof}
	By the boundedness of $p_{ij}$ and Parceval's identity,
	\begin{align}
		\Delta_m^2 &= \intop p \langle\theta^*_T,\phi_V\rangle^2 \\
		&\leq \bar{\epsilon}\intop \langle\theta^*_T,\phi_V\rangle^2 \\
		&= \bar{\epsilon}\Vert\theta^*_T\Vert_2^2.
	\end{align}
	Observe that
	\begin{align}
		\Vert\theta^*_T\Vert^2 &= \sum_{i\in V}\sum_{k>m_1} (\theta_i^k)^2  \\
		&+\sum_{(i,j)\in E}\left(\sum_{k\geq 1,l>m_2}(\theta_{ij}^{kl})^2+\sum_{k> m_2,l\leq m_2}(\theta_{ij}^{kl})^2\right) \notag
	\end{align}
	Now, we have for each $i\in V$,
	
	\begin{align}
		\sum_{k> m_1} (\theta_i^k)^2  &= \sum_{k>m_1} (\theta_i^k )^2k^{2r} k^{-2r} \\
		&\leq \left(\sum_{k>m_1} (\theta_i^k)^2k^{2r}\right)m_1^{-2r} \\
		&\leq Cm_1^{-2r},
	\end{align}
	
	and also	
	\begin{align}
		\sum_{k\geq 1,l\geq m_2}(\theta_{ij}^{kl})^2 &= \sum_{k\geq 1,l\geq m_2}(\theta_{ij}^{kl})^2 k^{2r_i}l^{2r_j} k^{-2r_i}l^{-2r_j} \\
		&\leq \left(\sum_{k\geq 1,l\geq m_2}(\theta_{ij}^{kl})^2 k^{2r_i}l^{2r_j}\right)m_2^{-2r_j} \\
		&= O(m_2^{-2r}).
	\end{align}
	
	We similarly get a bound of $O(m_2^{-2r})$ when summing over $k>m_2,l\leq m_2$. Combining we have
	\begin{align}
	\Delta_m^2=O(\Vert\theta^*_T\Vert_2^2) = O(dm_1^{-2r}+\left| E\right| m_2^{-2r}).
	\end{align}
	To analyze $\gamma_m$, observe that
	\begin{align}
		\gamma_m &= \bigg|\sum_{i\in V}\sum_{k>m_1} \theta_i^k\phi_k  \\
		&\qquad+\sum_{(i,j)\in E}\left(\sum_{k\geq 1,l>m_2}(\theta_{ij}^{kl})\phi_{kl}+\sum_{k> m_2,l\leq m_2}(\theta_{ij}^{kl})\phi_{kl}\right) \bigg| \notag\\
		&\leq \sum_{i\in V}\sum_{k>m_1}\left| \theta_i^kb(k)\right| \\
		&\qquad+\sum_{(i,j)\in E}\left(\sum_{k\geq 1,l>m_2}\left|(\theta_{ij}^{kl})b(k)b(l)\right|+\sum_{k> m_2,l\leq m_2}\left|(\theta_{ij}^{kl})b(k)b(l)\right|\right). \notag
	\end{align}

	 Then we have
	\begin{align}
		\sum_{k>m_1} \left| \theta_i^k b(k)\right|
		&= O\left(\sum_{k>m_1}\left|\theta_i^k k^{\alpha}\right|\right)\\
		&\leq O\left(\left(\sum_{k>m_1} (\theta_i^k k^r)^2\right)^{1/2}\left( \sum_{k>m_1} k^{2(-r+\alpha)}\right)^{1/2}\right) \\
		&= O\left(m_1^{-r+\alpha+1/2}\right).
	\end{align}
	
	Similarly,
	
		\begin{align}
		 \sum_{k,l>m_2} \left|\theta_{ij}^{kl} b(k)b(l)\right| &=
		 O\left(\sum_{k\geq 1,l>m_2}\left|\theta_{ij}^{kl}k^{\alpha}l^\alpha\right|\right)\\
		&\leq O\left(\left(\sum_{k\geq 1,l>m_2} (\theta_{ij}^{kl} k^{r_i}l^{r_j})^2\right)^{1/2} \left(\sum_{k\geq 1,l>m_2} k^{2(-r_i+\alpha)}l^{2(-r_j+\alpha)}\right)^{1/2}\right)\\
		&= O\left(m_2^{-r_j+1/2+\alpha}\right).
	\end{align}
	
	The result is similar summing over $k\geq m_2,l\leq m_2$. Combining, we get
	
	\begin{align}
		\gamma_m = O\left(dm_1^{-r+\alpha+1/2}+\left| E\right|m_2^{-r+\alpha+1/2}\right)
	\end{align}
	
	Finally, the series $\sum_{k\geq 1} (\theta^k_i )^2 k^{2r}$ converges if and only if $\left|(\theta_i^k)^2\right| = o(k^{-2r-1})$. Similarly, $\sum_{k\geq 1} (\theta_{ij}^{kl})^2 k^{2r}$ converges if and only if $\left|(\theta_{ij}^{kl})^2\right|=o(k^{-2r-1})$, and similarly when summing over $l$. It follows that
	\begin{align}
		\Vert\theta^*_T\Vert_\infty = o(\max\{m_1^{-r-1/2},m_2^{-r-1/2}\}).
	\end{align} 
\end{proof}

\section{Estimation Error}

Consider the Taylor expansion of $Z(\theta)$ up to order 2, noting that $\nabla Z(\theta)=\mathbb{E}_\theta[\bar{\phi}]=\mu(\theta)$ and $\nabla^2Z(\theta) = \text{cov}_{p_\theta}[\bar{\phi}]$:

\begin{align}
	Z(\theta+\delta) &= Z(\theta) + \langle\delta,\mu(\theta)\rangle + \frac{1}{2}\delta^\top \left(\text{cov}_{p_{\theta+z\delta}}[\bar{\phi}]\right) \delta,
\end{align}
for some $z\in[0,1]$. See for example \citep{kakade2010learning,portnoy1988asymptotic}.
%
%

Consider the function

\begin{equation*}
	\mathcal{E}(\delta) = \mathcal{L}(\tilde{\theta} + \delta) - \mathcal{L}(\tilde{\theta}) - \lambda (R(\tilde{\theta}+\delta) - R(\tilde{\theta})).
\end{equation*}

Note that $\mathcal{E}(0)=0$, and so $\widehat{\delta} = \widehat{\theta}-\tilde{\theta}$ must satisfy $E(\widehat{\delta})\geq 0$, if $\widehat{\theta}$ exists.

%
%
%
%
\begin{lem} \label{estbounds}
	Denote $\mu^* = \mathbb{E}_p[\bar{\phi}]=\mathbb{E}_{\tilde{\theta}}[\bar{\phi}]$. Suppose $\lambda_n  \geq\mathcal{R}^*(\widehat{\mu}_e-\mu_e^*)$, $\Vert\widehat{\mu}_v-\mu_v\Vert \leq \frac{r_\mathcal{S}c_\mathcal{S}}{\sqrt{2}}$ and $\lambda\sqrt{\tilde{E}} \leq \frac{\sqrt{2}r_\mathcal{S}c_\mathcal{S}}{3}$, then the regularized MLE $\widehat{\theta}_\lambda$ is unique and satisfies
\begin{align}
	\Vert\widehat{\theta}_v-\tilde{\theta}_v\Vert &\leq \frac{1}{2c_\mathcal{S}}\Vert\widehat{\mu}_v-\bar{\mu}^*_v\Vert,\\
	\Vert\widehat{\theta}_e-\tilde{\theta}_e\Vert &\leq \frac{3}{4c_\mathcal{S}}\lambda_n\sqrt{\left| E \right|},\\
	\text{KL}(p_{\tilde{\theta}} \mid p_{\widehat{\theta}_\lambda}) &\leq \frac{1}{2c_\mathcal{S}}\left(\Vert\widehat{\mu}_v-\mu^*_v\Vert^2+\frac{9}{4}\lambda^2\left| E \right|\right).
\end{align}

\end{lem}

\begin{proof}
Using the cumulant expansion of $Z$, we have for some $z\in[0,1]$,

\begin{align}
	\mathcal{L}(\tilde{\theta}+\delta)-\mathcal{L}(\tilde{\theta}) &= \langle\delta,\widehat{\mu}\rangle + Z(\tilde{\theta})-Z(\tilde{\theta}+\delta) \\
	&= \langle\delta,\widehat{\mu}-\mu^*\rangle - \frac{1}{2}\delta^\top\left(\text{cov}_{p_{\tilde{\theta}+z\delta}}[\bar{\phi}]\right)\delta \label{eq:proof2}.
\end{align}

Consider the set $\mathcal{S}:=\{\theta:\Vert\theta-\theta^*\Vert\leq r_{\mathcal{S}}\}$ for a constant $r_\mathcal{S}$. $\mathcal{S}$ is a convex and compact set; since $\bar{\phi}$ is minimal, $\nabla^2Z(\theta)=\text{cov}_{\theta}[\bar{\phi}]>0$  \eqref{strongconv} for all $\theta\in\mathcal{S}$; thus $-\mathcal{L}$ is strongly convex over $\mathcal{S}$ with $\delta^\top (\text{cov}_{\theta}[\bar{\phi}])\delta > c_{\mathcal{S}}\Vert\delta\Vert^2$ for a constant $c_\mathcal{S}$. 

By the (generalized) Cauchy-Schwarz inequality,
\begin{eqnarray}
	\big| \langle \mu_e - \widehat{\mu}_e ,\delta_e \rangle \big| &\leq &\mathcal{R}(\delta_e)\mathcal{R}^*(\mu_e-\widehat{\mu}_e) \\
	&\leq &\frac{\lambda_n}{2}\mathcal{R}(\delta_e) \\
	& =& \frac{\lambda_n}{2}\left( \mathcal{R}(\delta_{\tilde{\mathcal{P}}})+\mathcal{R}(\delta_{\tilde{\mathcal{P}}^\perp}) \right). \label{gencs}
\end{eqnarray}

the last line following because $\mathcal{R}$ is decomposable with respect to $\tilde{\mathcal{P}}$. Further, from  \citep{negahban2012unified} Lemma 3, because $\mathcal{R}$ is decomposable, it  holds that

\begin{equation}
	\mathcal{R}(\tilde{\theta}+\delta)-\mathcal{R}(\tilde{\theta}) \geq \mathcal{R}(\delta_{\tilde{\mathcal{P}}^\perp})-\mathcal{R}(\delta_{\tilde{\mathcal{P}}}). \label{negh}
\end{equation}
Combining \eqref{negh} and \eqref{gencs},
\begin{align}\langle \delta_e,\widehat{\mu}_e-\mu_e\rangle - \frac{\lambda_n}{2}\left(\mathcal{R}(\tilde{\theta}+\delta)-\mathcal{R}(\tilde{\theta}) \right) &\leq \frac{3\lambda_n}{2}\mathcal{R}(\delta_{\tilde{\mathcal{P}}})-\frac{\lambda_n}{2}\mathcal{R}(\delta_{\tilde{\mathcal{P}}^\perp}) \label{eq:proof1}.
\end{align} 

Using the subspace compatibility constant we have that 
\begin{align}\mathcal{R}(\delta_{\tilde{\mathcal{P}}})\leq \sqrt{\left| E\right|} \Vert\delta_{\tilde{\mathcal{P}}}\Vert \leq \sqrt{\left| E \right|}\Vert\delta_e\Vert.
\end{align}
Additionally, the second term in \eqref{eq:proof1} is negative and can be ignored. From Cauchy-Schwarz we also have 
\begin{align}
\langle \delta_v, \mu_v-\widehat{\mu}_v \rangle \leq \Vert\mu_v-\widehat{\mu}_v\Vert\Vert\delta_v\Vert.
\end{align}

 Now, consider the set 
 \begin{align}
 \mathcal{C}=\left\{\delta:\Vert\delta_v\Vert \leq \frac{1}{2c_\mathcal{S}}\Vert\widehat{\mu}_v-\mu_v\Vert; \Vert\delta_e\Vert \leq \frac{3}{4c_\mathcal{S}}\lambda\sqrt{\tilde{E}}\right\}
 \end{align}
 denote the boundary of $\mathcal{C}$ by $\partial\mathcal{C}$ and its interior by $\text{int}\mathcal{C}$. Note that $0\in\mathcal{C}$. If $E(\delta)<0$ for each $\delta\in\partial\mathcal{S}$, noting that $\mathcal{E}(0)=0$, by the convexity of $\mathcal{E}$ and the fact that $\mathcal{C}$ is a compact set, it must be that $\widehat{\delta}\in\text{int}\mathcal{C}$. Suppose that $\Vert\widehat{\mu}_v-\mu_v^*\Vert \leq \frac{2r_\mathcal{S}c_\mathcal{S}}{\sqrt{2}}$ and $\frac{3}{2}\lambda\sqrt{\left| E\right|}\leq\frac{2r_\mathcal{S}c_\mathcal{S}}{\sqrt{2}}$, so that for $\delta\in\mathcal{C}$, $\Vert\delta\Vert\leq r_\mathcal{S}$, so that for any $z\in[0,1]$, $\tilde{\theta}+z\delta\in\mathcal{S}$.
 Using \eqref{eq:proof2} and \eqref{eq:proof1} we obtain for $\delta\in\partial\mathcal{C}$,
\begin{align}
	\mathcal{E}(\delta) &= \mathcal{L}(\tilde{\theta}+\delta)-\mathcal{L}(\tilde{\theta})-\lambda\left(\mathcal{R}(\tilde{\theta}+\delta)-\mathcal{R}(\tilde{\theta})\right) \\
	&\leq \langle\delta_v,\widehat{\mu}_v-\mu^*_v\rangle +\langle\delta_e,\widehat{\mu}_e-\mu^*_e\rangle - c_\mathcal{S}\Vert\delta\Vert^2 -\lambda\left(\mathcal{R}(\tilde{\theta}+\delta)-\mathcal{R}(\tilde{\theta})\right) \\
	&\leq \Vert\delta_v\Vert\Vert\widehat{\mu}_v-\mu^*_v\Vert + \frac{3\lambda\sqrt{\tilde{E}}}{2}\Vert\delta_e\Vert - c_\mathcal{S}\Vert\delta\Vert^2 \\
	&\leq \Vert\delta_v\Vert\left(\Vert\widehat{\mu}_v-\mu^*_v\Vert - c_\mathcal{S}\Vert\delta_v\Vert\right) + \Vert\delta_e\Vert\left(\frac{3\lambda\sqrt{\tilde{E}}}{2}-c_\mathcal{S}\Vert\delta_e\Vert\right)\\
	&< 0.
\end{align}

The claim is verified.

Furthermore, since 

\begin{align}
\text{KL}(p_{\tilde{\theta}}\mid p_{\widehat{\theta}})=\langle\tilde{\theta}-\widehat{\theta},\mu^*\rangle-Z(\tilde{\theta})+Z(\widehat{\theta}),
\end{align}
note that 
\begin{align}
	\mathcal{E}(\widehat{\delta}) &= \langle\widehat{\delta},\widehat{\mu}\rangle - Z(\tilde{\theta}+\widehat{\delta})+Z(\tilde{\theta}) - \lambda (\mathcal{R}(\tilde{\theta}+\delta)-\mathcal{R}(\tilde{\theta})) \\
	&= -\text{KL}(p_{\tilde{\theta}} \mid p_{\widehat{\theta}})+\langle\widehat{\delta},\widehat{\mu}-\mu^*\rangle - \lambda(\mathcal{R}(\tilde{\theta}+\delta)-\mathcal{R}(\theta^*)).
\end{align}
because $\mathcal{E}(\widehat{\delta})\geq 0$,
\begin{align}
	\text{KL}(p_{\tilde{\theta}}\mid p_{\widehat{\theta}}) &\leq \langle\widehat{\delta},\widehat{\mu}-\mu^*\rangle - \lambda\left(\mathcal{R}(\tilde{\theta}+\widehat{\delta})-\mathcal{R}(\tilde{\theta})\right) \\
	&\leq \Vert\widehat{\delta}_v\Vert\Vert\widehat{\mu}_v-\mu^*_v\Vert + \frac{3\lambda\sqrt{\left| E\right|}}{2}\Vert \widehat{\delta}_e\Vert \\
	&\leq \frac{1}{2c_\mathcal{S}}\left(\Vert\widehat{\mu}_v-\mu^*_v\Vert^2 + \frac{9\lambda^2\tilde{E}}{4}\right).
\end{align}

\end{proof}

\begin{lem} \label{concentration}
Suppose that the univariate sufficient statistics satisfy $\mid \phi_k\mid \leq b(k)$ for all $k$ and $b(k)$ is increasing in $k$. For any $t>0$,
\begin{align}
	\mathbb{P}(\Vert\widehat{\mu}_v-\mu_v\Vert^2 > t) &\leq \frac{\bar{\epsilon}dm_1}{nt}, \\
	\mathbb{P}(\mathcal{R}^*(\widehat{\mu}_e-\mu^*_e) > t) &\leq 2\exp\left\{ \log(d^2m_2^2) - \frac{nt^2}{8\bar{\epsilon}m_2^2b(m_2)^4}\right\}, \\
	\mathbb{P}(\Vert\widehat{\mu}-\mu^*\Vert_\infty >t) &\leq 2\exp\left\{\log(dm_1+d^2m_2) - \frac{nt^2}{8\bar{\epsilon}\max\{b(m_2)^4,b(m_1)^2\}}\right\}.
\end{align}
\end{lem}
\begin{proof}
Let $(\text{cov}_p[\bar{\phi}])_v$ be the covariance matrix of the sufficient statistics $\phi$ restricted to the vertex parameters. Then $((\text{cov}_p[\bar{\phi}])_v)^{-1/2}(\widehat{\mu}_v-\mu_v)$ is a vector of  $dm_1$ mean zero, uncorrelated random variables each with variance $1/n$. It follows from Markov's inequality that

\begin{align}
	\mathbb{P}\left((\widehat{\mu}_v-\mu_v)^\top((\text{cov}_p[\phi])_v)^{-1}(\widehat{\mu}_v-\mu_v) > t \right) &\leq \frac{dm_1}{nt}.
\end{align}

From the boundedness of the marginals $p_{ij}$ and the orthonormality of $\{\phi\}$, $\bar{\epsilon}^{-1}\Vert\delta\Vert^2\leq \delta^\top(\text{cov}_p[\phi])^{-1}\delta$ for any $\delta$. It follows that 
\begin{align}
	\mathbb{P}(\Vert\widehat{\mu}_v-\mu_v\Vert^2 > \bar{\epsilon}t) \leq \frac{dm_1}{nt}.
\end{align}
Now, to bound $\mathcal{R}^*(\widehat{\mu}_e-\mu^*_e)$ we will need to find an exponential concentration for each $\Vert\widehat{\mu}_{ij}-\mu_{ij}\Vert$. Suppose that the univariate sufficient statistic is bounded: $\mid \phi_k\mid\leq b(k)$ for all $k$. Applying Hoeffding's inequality, we get for each indices $k,l$,

\begin{align}
	\mathbb{P}(\bar{\epsilon}^{-1/2}\big|\widehat{\mu}_{ij}^{kl}-\mu_{ij}^{kl} \big| \geq t) &\leq 2 \exp\left\{-n\frac{t^2}{8b(k)^2b(l)^2}\right\},
\end{align}

and so

\begin{align}
	\mathbb{P}\left(\bar{\epsilon}^{-1}\Vert\widehat{\mu}_{ij}-\mu_{ij}\Vert^2 \geq m_2^2b(m_2)^4t^2\right) &\leq \mathbb{P}\left(\bigcup_{k,l}\left\{\bar{\epsilon}^{-1/2}\mid\widehat{\mu}_{ij}^{kl}-\mu_{ij}^{kl}\mid \geq tb(k)b(l)\right\}\right) \\
	&\leq \sum_{k,l=1}^{m_2} \exp\left\{-\frac{nt^2}{8} \right\} \\
	&\leq 2m_2^2 \exp\left\{ - \frac{nt^2}{8}\right\}.
\end{align}

Applying the union bound once more, 
\begin{align}
	\mathbb{P}(\mathcal{R}^*(\widehat{\mu}_e-\mu^*_e) \geq t^2\bar{\epsilon}) \leq 2\exp\left\{\log(d^2 m_2^2) - \frac{nt^2}{8m_2^2b(m_2)^4} \right\}.
\end{align}

\end{proof}

\begin{proof}[Proof of Theorem \ref{riskthm}]
	From theorem \ref{klapprox}, we have the approximation error
	\begin{align}
		\text{KL}(p\mid p_{\theta}) = O\left(e^{2\gamma_m}\Delta_m^2 \right),
	\end{align}
	From \ref{klapprox2}, we know that 
\begin{align}	
\Delta_m^2=O(dm_1^{-2r}+\left|E\right|^{-2r}),
\end{align}
 and 
\begin{align}
\gamma_m=O(dm_1^{-r+\alpha+1/2}+\left| E\right|m_2^{-r+\alpha+1/2}).
\end{align}
	To keep $e^{\gamma_m}$ bounded, we require
	
	\begin{align}
		m_1 = \Omega(d^{\frac{1}{r-\alpha-1/2}}), \\
		m_2 = \Omega(\left| E\right|^{\frac{1}{r-\alpha-1/2}}).
	\end{align}
	From lemma \ref{estbounds}, the estimation error satisfies, for $\lambda_n \geq\mathcal{R}^*(\widehat{\mu}_e-\mu^*_e)$,
	\begin{align}
		\text{KL}(p_{\tilde{\theta}}\mid p_{\widehat{\theta}}) = O\left(\lambda_n^2\left|E\right| + \Vert\widehat{\mu}_v-\mu^*_v\Vert_2^2d\right).
	\end{align}
	
	Applying lemma \ref{concentration}, $\mathcal{R}^*(\widehat{\mu}_e-\mu^*_e)>\lambda_n$ with probability no more than
	
	\begin{align}
		2\exp\left\{\log(m_2^2d^2)-\frac{\lambda_n^2 n}{8\bar{\epsilon}m_2^2b(m_2)^4}\right\},
	\end{align}
	
	choosing $\lambda_n =\frac{2\bar{\epsilon}^{1/2}m_2b(m_2)^2\log^{1/2}(md)}{n^{1/2}}$, the probability is bounded by $\frac{2}{m_2^2d^2}$. Also, 
\begin{align}	
\Vert\widehat{\mu}_v-\mu^*_v\Vert_2^2 > \frac{d\bar{\epsilon}m_1}{n}t
\end{align}
with probability no more than $\frac{1}{t}$.
	 Thus applying lemma \ref{pyth}, under the stated conditions we have
	
	\begin{align}
		\text{KL}(p\mid p_{\widehat{\theta}}) &= \text{KL}(p\mid p_{\tilde{\theta}}) + \text{KL}(p_{\tilde{\theta}}\mid p_{\widehat{\theta}}) \\
		&= O_p\left(dm_1^{-2r}+\left| E\right|m_2^{-2r} + \frac{m_2^2b(m_2)^4\log(nd)}{n}\left| E\right|+\frac{m_1}{n} d\right)
	\end{align}

\end{proof}

\section{Model Selection}

\subsubsection{Taylor expansion in Hilbert Space}
For any $f=\langle\theta,\phi\rangle\in W_r^2$, $\Vert\theta\Vert_2=\Vert f\Vert_{L_2}<\infty$. Thus the set of such $\theta$ belong to a Hilbert space. Applying the Taylor expansion with remainder for Hilbert spaces, for $\Vert\delta\Vert<\infty$ we have
\begin{align}
	Z(\theta+\delta) &= Z(\theta) +  DZ(\theta)\cdot\delta+\frac{1}{2}D^2 Z(\delta)\cdot \delta^2+\frac{1}{6}D^3(\theta+z\delta)\cdot \delta^3.
\end{align}

here $z$ is some point in $[0,1]$, and $D^k Z$ is the $k$th Fr\'echet derivative of $Z$ represented as a multilinear map, $\cdot \delta^k$ denotes evaluation at $(\delta,\ldots,\delta)$. Now, consider $\theta^*,\delta$ which can be decomposed to a truncated vector and a remainder, $\theta^*=(\bar{\theta}^*,\theta^*_T)^\top$ and $\delta=(\bar{\delta},\delta_T)^\top$ and write $\phi=(\bar{\phi},\phi_T)$. Taking the derivative of $Z$ with respect to $\bar{\theta}^*$, and inputting the particular expression for the terms of the Taylor series (see \citep{portnoy1988asymptotic}), we get the Taylor expansion for $\bar{\mu}$,

\begin{align}
	\bar{\mu}(\theta^*+\delta)-\bar{\mu}(\theta^*) &= \text{cov}_{\theta^*}[\bar{\phi},\phi]\delta + \frac{1}{2}\mathbb{E}_{\theta^*+z\delta}[\langle\delta,\phi\rangle^2 \bar{\phi}] \\
	&= \text{cov}_{\theta^*}[\bar{\phi},\bar{\phi}]\bar{\delta}+\text{cov}_{\theta^*}[\bar{\phi},\phi_T]\cdot\delta_T \\
	&\qquad+ \underset{\bar{R}}{\underbrace{\frac{1}{2}\mathbb{E}_{\theta^\dagger}\left[\left(\langle\bar{\delta},\bar{\phi}-\mathbb{E}_{\theta^\dagger}[\bar{\phi}]\rangle+\langle\delta_T,\phi_T-\mathbb{E}_{\theta^\dagger}[\phi_T]\rangle\right)^2\cdot(\bar{\phi}-\mathbb{E}_{\theta^\dagger}[\bar{\phi}])\right]}}. \notag \label{sobtaylor}
\end{align}

Here $\theta^\dagger=\theta^*+z\delta$ for some $z\in[0,1]$. We denote the final term the \emph{remainder}, $\bar{R}:=\bar{R}(\delta)$.
\subsubsection{Primal Dual Witness}
Our proof technique is the \emph{primal dual witness} method, used previously in analysis of model selection for graphical models \citep{modelselecttaylor2014,ravikumar2011high}. It proceeds as follows: construct a primal-dual pair $(\widehat{\theta},\widehat{Z})$ which satisfies $\text{supp}(\widehat{\theta})=\text{supp}(\theta^*)$, and also satisfies the stationary conditions for \ref{mlsol} with high probability. The stationary conditions for \ref{mlsol} are
\begin{align}
	\widehat{\mu}-\mu(\widehat{\theta}) + \lambda_n\widehat{Z} = 0,
\end{align}

where $\widehat{Z}$ is an element of the subdifferential $\partial\mathcal{R}(\widehat{\theta}_e)$:

\begin{align}
	(\partial\mathcal{R}(\theta_e))_{ij} = \begin{cases}
		\{\theta_{ij}: \Vert\theta_{ij}\Vert_2\leq 1\}, & \text{  if } \theta_{ij}=0; \\
		\frac{\theta_{ij}}{\Vert\theta_{ij}}\Vert_2, & \text{  if } \theta_{ij}\not=0; \\
		0, & \text{  if } i=j.
	\end{cases}
\end{align}
From this we may conclude that the solution to \ref{mlsol} is sparsistent. In particular, we have the following steps:
\begin{enumerate}
	\item Set $\widehat{\theta}_{E^c} = 0$;
	\item Set $\widehat{Z}_{ij} = \partial\mathcal{R}(\theta^*_e)_{ij}=\frac{\bar{\theta}^*_{ij}}{\Vert\bar{\theta}^*_{ij}\Vert}$ for $(i,j)\in E$;
	\item Given these choices for $\widehat{\theta}_{E^c}$ and $\widehat{Z}_E$, choose $\widehat{\theta}_E$ and $\widehat{Z}_{E^c}$ to satisfy the stationary condition \ref{mlsol}.
\end{enumerate}

For our procedure to succeed, we must show this primal-dual pair $(\widehat{\theta},\widehat{Z})$ is optimal for \ref{mlsol}, in other words

\begin{align}
	&\widehat{\theta}_{ij} \not = 0 , &\text{for}\qquad (i,j)\in E;\label{pd1} \\
	&\Vert\widehat{Z}_{ij}\Vert <1, &\text{for}\qquad (i,j)\not\in E. \label{pd2}
\end{align}

In the sequel we show these two conditions hold with probability approaching one.

We begin by proving a bound on the remainder.

\begin{lem}\label{remainderlem}
	The remainder \eqref{sobtaylor} satisfies
	
	\begin{align}
		\bar{R} = \left(\Vert\Delta_E\Vert_\infty^2+\Vert\theta^*_T\Vert_\infty^2\right)\bar{K}_{\theta^\dagger}.
	\end{align}
\end{lem}

\begin{proof}
	By Cauchy-Schwarz,
	\begin{align}
		\langle\bar{\phi}_E-\mathbb{E}_{\theta^\dagger}[\bar{\phi}_E],\Delta_E\rangle &\leq \Vert\Delta_E\Vert_\infty\Vert\bar{\phi}_E-\mathbb{E}_{\theta^\dagger}[\bar{\phi}_E]\Vert_1,
	\end{align}
	and similarly for $\langle\phi_T,\theta^*_T\rangle$,
	
	\begin{align}
		\langle \phi_T-\mathbb{E}_{\theta^\dagger}[\phi_T],\theta^*_T\rangle &\leq \Vert\theta^*_T\Vert_\infty\Vert \phi_T-\mathbb{E}_{\theta^\dagger}[\phi_T]\Vert_1.
	\end{align}
	 So
	\begin{align*}
		\bar{R} &\leq \max\{\Vert\Delta\Vert_\infty,\Vert\theta^*_T\Vert_\infty\}^2\mathbb{E}[(\Vert\bar{\phi}_E\Vert_1+\Vert\phi_T\Vert_1)^2 \cdot(\bar{\phi}-\mathbb{E}[\bar{\phi}])] \\
		&=\max\{\Vert\Delta\Vert_\infty,\Vert\theta^*_T\Vert_\infty\}^2\bar{K} \\
		&\leq (\Vert\Delta_E\Vert_\infty^2 + \Vert\theta^*_T\Vert_\infty^2) \bar{K}.
	\end{align*}
	
	Applying the $L_\infty$ norm gives
	
	\begin{align}
		\Vert\bar{R}\Vert_\infty &\leq (\Vert\Delta_E\Vert_\infty^2 + \Vert\theta^*_T\Vert_\infty^2)\kappa_R \max\{b(m_2)^2,b(m_1)\}.
	\end{align}
\end{proof}

\subsubsection{Condition \eqref{pd1}}

Denote $b:=\max\{b(m_2)^2,b(m_1)\}$ and $m:=\max\{m_2,\sqrt{m_1}\}$.

\begin{lem} \label{infnorm}
Let 
	\begin{align}
	\tilde{r}:= 2\kappa_\Gamma(\Vert W_E\Vert_\infty+ \lambda_n/m + (\kappa_T+1) \Vert\theta^*_T\Vert_\infty),
	\end{align}
	 and suppose $\tilde{r}\leq\frac{1}{2b\kappa_R\kappa_\Gamma}$ and $\Vert \theta^*_T\Vert_\infty \kappa_R b \leq 1$. Then
	\begin{align}
		\Vert\widehat{\theta}_E-\bar{\theta}^*_E\Vert_\infty \leq \tilde{r}.
	\end{align}
\end{lem}

\begin{proof}
The stationary condition for $\widehat{\theta}_E$ is given by

\begin{align}
	\widehat{\mu}_E - \mu(\widehat{\theta})_E + \lambda_n \widehat{Z}_E = 0.
\end{align}

Denote $W:= \widehat{\mu}-\bar{\mu}^*$. Then we may re-write

\begin{align}
	&W_E + \bar{\mu}^*_E - \mu(\widehat{\theta})_E + \lambda_n \widehat{Z}_E \\
	&=W_E-\Gamma_{EE}\widehat{\Delta}_E-\Gamma_{ET} \theta^*_T - R_E + \lambda_n\widehat{Z}_E, \label{edgepar}
\end{align}
where $\widehat{\Delta}_E=\widehat{\theta}_E-\bar{\theta}^*_E$. Re-arranging and applying the $L_2$ norm, we get for $(i,j)\in E$,

\begin{align}
	\Vert\widehat{\Delta}_{ij}\Vert_2 &= \Vert\Gamma_{ij,E}^{-1}(-W_E+\Gamma_{ET}\theta^*_T+R_E-\lambda_n\widehat{Z}_E)\Vert_2 \\	
	&=\Vert\Gamma_{ij,E}^{-1}\Vert_2(\Vert W_{E}\Vert_2+m\sqrt{d+E}\Vert\Gamma_{E,T}\Vert_\infty\Vert \theta^*_T\Vert_\infty+\Vert R_{E}\Vert_2 +\lambda_n\Vert\widehat{Z}_{E}\Vert_2) \\
	&\leq\frac{\kappa_\Gamma}{m\sqrt{d+E}}(\Vert W_{E}\Vert_2+m\sqrt{d+E}\kappa_T \Vert \theta^*_T\Vert_\infty+\Vert R_{E}\Vert_2+\sqrt{d+E}\lambda_n). \\
	&=\kappa_\Gamma \left(\Vert W_E\Vert_\infty + \kappa_T \Vert\theta^*_T\Vert_\infty + \Vert R_E\Vert_\infty +\lambda_n/m\right).
\end{align}
Consider the mapping

\begin{align}
	F(\Delta_E) &:= -\Gamma^{-1}_{EE}(W_E-\Gamma_{EE}\Delta_E-\Gamma_{EV}\theta^*_T-R_E(\Delta_E)+\lambda_n\widehat{Z}_E)+\Delta_E \\
	 &=-\Gamma^{-1}_{EE}(W_E -\Gamma_{ET} \theta^*_T - R_E + \lambda_n\widehat{Z}_E).
\end{align}

Due to the uniqueness of the solution to the stationary conditions, $F(\Delta_E)$ has a unique fixed point $F(\Delta_E)=\Delta_E$ at $\widehat{\Delta}=\widehat{\theta}-(\bar{\theta}^*)$. If we can show that $\Vert F(\Delta_E)\Vert_\infty\leq\tilde{r}$ for every $\Vert\Delta_E\Vert_\infty\leq\tilde{r}$, since $F$ is continuous and $\{\Delta_E:\Vert \Delta_E\Vert_\infty\leq\tilde{r}\}$ is a convex and compact set, applying Brouwer's fixed point theorem \citep{ortega2000iterative} implies that the unique fixed point of $F$ satisfies $\Vert\widehat{\Delta}_E\Vert_\infty\leq\tilde{r}$. The $L_2$ norm of the map follows, for $(i,j)\in E$,

\begin{align}
	\Vert F(\Delta_E)_{ij}\Vert_2 &\leq \Vert\Gamma_{ij,E}^{-1}\Vert_2(\Vert W_{E}\Vert_2+m\Vert\Gamma_{E,T}\Vert_{\infty}\Vert \theta^*_T\Vert_\infty+\Vert R_{E}\Vert_2+\lambda_n\Vert\widehat{Z}_E\Vert_2)\\
	&\leq\kappa_\Gamma\left(\Vert W_{E}\Vert_\infty + \kappa_T \Vert \theta^*_T\Vert_\infty+ \Vert R_{E} \Vert _\infty + \lambda_n/m \right).
\end{align}

Let $\tilde{r}:= 2\kappa_\Gamma (\Vert W_E\Vert_\infty + (\kappa_T+1)\Vert \theta^*_T\Vert_\infty +\lambda_n/m)$ and consider $\Vert\Delta_{ij}\Vert_\infty \leq\tilde{r}$ for $(i,j)\in E$. Suppose that also $\tilde{r}\leq\frac{1}{2b\kappa_R\kappa_\Gamma}$.

From lemma \ref{remainderlem} the remainder is bounded by

\begin{align}
	\Vert \bar{R}_{E}\Vert_\infty &\leq b\tilde{\kappa}_R(\Vert\Delta_E\Vert_\infty^2 + \Vert \theta^*_T\Vert_\infty^2).
\end{align}

If $\Vert\Delta_E\Vert_\infty\leq \frac{1}{2b\tilde{\kappa}_R\kappa_\Gamma}$, and $\Vert \theta^*_T\Vert_2 b\kappa_R\leq 1$, this is bounded by 

\begin{align}
	\frac{1}{2\kappa_\Gamma}\Vert\Delta_E\Vert_\infty + \Vert \theta^*_T\Vert_\infty,
\end{align}

Thus, since $\Vert F\Vert_\infty \leq \text{max}_{ij}\Vert F_{ij}\Vert_2$,
\begin{align}
	\Vert F(\Delta_E)\Vert_\infty &\leq \kappa_\Gamma\left(\Vert W_E\Vert_\infty+(\kappa_\Gamma+1)\Vert \theta^*_T\Vert_\infty+\lambda_n/m+ \frac{\Vert\Delta_E\Vert_\infty}{2\kappa_\Gamma}\right)\\
	&\leq \frac{\tilde{r}}{2}+\frac{\tilde{r}}{2}=\tilde{r}.
\end{align}

It follows that the fixed point $\widehat{\Delta}_E=\widehat{\theta}_E-\bar{\theta}^*_E$ satisfies
%

\begin{align}
	\Vert\widehat{\Delta}_E\Vert_\infty &\leq 2\kappa_\Gamma(\Vert W_E\Vert_\infty + \lambda_n/m + (\kappa_T+1)\Vert \theta^*_T \Vert_\infty).
\end{align}
\end{proof}

\subsubsection{Condition \eqref{pd2}}
\begin{lem}\label{subgradpdw}
	Suppose that 
	\begin{align}
		\max\left\{8\kappa_\Gamma m/\tau\Vert W\Vert_\infty ,m(\Vert W\Vert_\infty+(1+\kappa_T)\Vert\theta^*_T\Vert_\infty)\right\}\leq\lambda_n\leq\frac{\tau m}{128b \kappa_R\kappa_\Gamma^2},\\
		\Vert\theta^*_T\Vert_\infty \leq\min\left\{\frac{\tau\lambda_n}{4(1+\kappa_T)m},\frac{m}{b\kappa_R}\right\},	
	\end{align}
 Then $\max_{(i,j)\in E^c}\Vert\widehat{Z}_{ij}\Vert_2<1$.
\end{lem}
\begin{proof}
Recall that $\widehat{\theta}_{E^c}=\bar{\theta}^*_{E^c}=0$. For $(i,j)\in E^c$, the stationary condition gives

\begin{align}
	W_{E^c} -\Gamma_{E^cE}((\theta^*_m)_E-\widehat{\theta}_E) -\Gamma_{E^cV} \theta^*_T - R_{E^c}+\lambda_n \tilde{Z}_{E^c} &=0
\end{align}

Now, re-arranging \eqref{edgepar},
\begin{align}
	\bar{\theta}^*_E-\widehat{\theta}_E = \Gamma_{EE}^{-1}(W_E+\lambda_n\widehat{Z}_E+R_E+\Gamma_{ET}\theta^*_T).
\end{align}
It follows that

\begin{align}
	\widehat{Z}_{E^c} = \frac{1}{\lambda_n} \left\{-W_{E^c}-\Gamma_{E^cE}\Gamma_{EE}^{-1}(W_E+\Gamma_{E^c V} \theta^*_T+R_E + \lambda_n\widehat{Z}_E)-\Gamma_{E^cT} \theta^*_T -R_{E^c} \right\}.
\end{align}

Now, for $(i,j)\in E^c$,

\begin{align}
	\Vert\widehat{Z}_{ij}\Vert_2 &\leq \frac{1}{\lambda_n} \left\{ \Vert W_{ij}\Vert_2 + \Vert \Gamma_{ij,E}\Gamma^{-1}_{EE}(W_E+\Gamma_VV+R_E+\lambda_n\tilde{Z}_E)\Vert_2+\Vert\tilde{\Gamma}_{ij,E}\theta^*_T\Vert_2+\Vert R_{ij}\Vert_2 \right\} \\
	&\leq \frac{1}{\lambda_n}\bigg\{ \Vert W_{ij}\Vert_2 + \Vert\Gamma_{ij,E}\Gamma_{EE}^{-1}\Vert_2\bigg((\Vert W_E\Vert_2 + m\sqrt{d+E}\kappa_{T}\Vert \theta^*_T\Vert_\infty + \Vert R_E\Vert_2  \\ 
	&\qquad+\lambda_n\Vert\widehat{Z}_E\Vert_2\bigg)+\Vert R_{ij}\Vert_2\bigg\}\notag\\
	&\leq \frac{1}{\lambda_n}\bigg\{ \Vert W_{ij}\Vert_2 +\Vert R_{ij}\Vert_2+ \frac{(1-\tau)}{\sqrt{d+E}}\bigg(\Vert W_E\Vert_2  \\
	&\qquad+m\sqrt{d+E}\kappa_{T}\Vert \theta^*_T\Vert_\infty +\Vert R_E\Vert_2 \bigg)+(1-\tau)\lambda_n \bigg\} \notag\\
	&\leq \frac{1}{\lambda_n}\left\{ m(2-\tau)(\Vert W\Vert_\infty + \Vert R\Vert_\infty)\right\} \\
	&\qquad+\frac{1}{\lambda_n}(1-\tau)\kappa_T m\Vert\theta^*_T\Vert_\infty+1-\tau. \notag
\end{align}

From the assumptions of the lemma, and applying Lemma \ref{infnorm}, we get

\begin{align}
	\Vert\Delta_E\Vert_\infty &\leq 2\kappa_\Gamma(\Vert W_E\Vert_\infty + \lambda_n/m + (\kappa_T+1) \Vert \theta^*_T\Vert_\infty) \\
	&\leq 4\kappa_\Gamma\lambda_n/m.
\end{align}
From Lemma \ref{remainderlem}, if $16b\kappa_R\kappa_\Gamma^2\lambda_n\leq\frac{\tau m}{8}$ and $b\kappa_R\Vert\theta^*_T\Vert_\infty\leq m$,
\begin{align}
	\Vert \bar{R}\Vert_\infty &\leq b\kappa_R (\Vert \Delta_E\Vert_\infty^2 + \Vert \theta^*_T\Vert_\infty^2) \\
	&\leq  16b\kappa_R\kappa_\Gamma^2 \lambda_n^2/m^2+m\Vert\theta^*_T\Vert_\infty \\
	&\leq \frac{\tau\lambda_n}{8m}+m\Vert\theta^*_T\Vert_\infty,
\end{align}

if $(1+\kappa_T)m\Vert\theta^*_T\Vert_\infty\leq\frac{\tau}{4}\lambda_n$, we may bound $\Vert\widehat{Z}_{ij}\Vert_2$ by:

\begin{align}
	\Vert\widehat{Z}_{ij}\Vert_2 &\leq(2-\tau)\frac{\tau}{4}+1-\tau+\frac{1+\kappa_T}{\lambda_n}m\Vert\theta^*_T\Vert_\infty \\
	&\leq 1-\tau+\frac{\tau}{2}+\frac{\tau}{4} = 1-\frac{3\tau}{4}<1.
\end{align}

\end{proof}

\begin{proof}[Proof of Theorem \ref{modelselect}]

For the assumptions of Lemma \ref{subgradpdw} to hold, we need

\begin{align}
	\lambda_n =\Omega( m(\Vert W\Vert_\infty+ \Vert\theta^*_T\Vert_\infty)). \label{subgradcond}
\end{align}

From Lemma \ref{concentration}, $\Vert W\Vert_\infty>t$ with probability no more than
\begin{align}
	2\exp\left\{2\log(nd) - \frac{nt^2}{8\underline{\epsilon}\max\{b(m_2)^4,b(m_1)^2\}}\right\},
\end{align}

and the truncation parameters satisfy

\begin{align}
	\Vert\theta^*_T\Vert_\infty = O(\max\{m_1^{-r-1/2},m_2^{-r-1/2}\}).
\end{align}

Supposing that $m_2= m_1$, \eqref{subgradcond} is satisfied with the choice 

\begin{align}
	\lambda_n\asymp m_2\left(\sqrt{\frac{m_2^{4\alpha}\log(nd)}{n}} + m_2^{-r-1/2}\right),
\end{align}

with probability approaching one. Balancing the two terms, if we choose $m_2\asymp n^{\frac{1}{2r+1+4\alpha}}$, we get

\begin{align}
	\lambda_n \asymp \sqrt{\frac{\log(nd)}{n^{\frac{2r-1}{2r+1+4\alpha}}}}
\end{align}

Lastly, for condtion \ref{pd1} to hold, we need $\max_{(i,j)\in E}\Vert\widehat{\theta}_{ij}-\bar{\theta}_{ij}^*\Vert_\infty\leq\frac{\rho^*}{2}$. Using lemma \ref{infnorm}, this is satisfied if
\begin{align}
	\frac{\lambda_n}{m_2\rho^*}\rightarrow 0,
\end{align}

which is satisfied with high probability when

\begin{align}
	\frac{1}{\rho^*} = o\left(\sqrt{\frac{n^{\frac{2r+1}{2r+1+4\alpha}}}{\log(nd)}}\right).
\end{align}

\end{proof}

\chapter{Proofs for Chapter 4}

\begin{proof}[Proof of Proposition \ref{boundedscorematch}]

Consider the functional

\begin{align}
	J(q) &:= \mathbb{E}_p\left[(\Vert\nabla \log p-\nabla \log q )\otimes x(1-x)\Vert_2^2\right].
\end{align}

If $J(q)=0$, then it must be $\nabla \log p= \nabla \log q$  a.e., because their integrated squared distance is zero with respect to a weight function which is nonzero a.e. This implies $\log q = \log p + c$ a.e. for some constant $c$, but $c=0$ because $p$ and $q$ must both integrate to one. Furthermore, $J$ is non-negative so it is minimized when $q=p$. If $p$ and $q$ belong to an exponential family with respective natural parameters $\theta$ and $\theta'$, $\theta=\theta'$ when the family is minimal.

Now,

\begin{align}
	J(q) &= \mathbb{E}_p\left[\Vert\nabla \log q \otimes x(1-x)\Vert_2^2\right] \\
	&\qquad+ 2\mathbb{E}_p\left[\sum_{i\in V} (\nabla_i \log q \cdot \nabla_i \log p)\otimes x(1-x)\Vert_2^2\right]+  constant, \notag
\end{align}

the constant not depending on $q$. We have, by integration by parts,

\begin{align}
	\mathbb{E}_p\left[(\nabla_i \log q \nabla_i \log p)x_i(1-x_i)\right] &=\intop p_i (\nabla_i \log q \nabla_i \log p)x_i(1-x_i) \\
	&= \intop p_i \frac{\nabla_i p_i}{p_i}(\nabla_i \log q )x_i(1-x_i) \\
	&= p_i(x_i)(\nabla_i \log q x_i(1-x_i))\bigg]_{x_i=1} \\
	&\qquad-p_i(x_i)(\nabla_i \log q x_i(1-x_i))\bigg]_{x_i=0} \notag\\
	&\qquad- \intop p_i \nabla_i (\nabla_i \log q x_i(1-x_i)) \notag\\
	&= - \intop p_i \nabla_i (\nabla_i \log q x_i(1-x_i)), 
\end{align}

where in the last line we applied the boundary assumption. Thus, we see that $J(q)$ is equal to $\mathbb{E}_p\left[h(X,q)\right]$ plus some terms which don't depend on $q$, so from the argument above $\mathbb{E}_p\left[ h(X,q)\right]$ is minimized when $p=q$. We conclude that $h$ is a proper scoring rule.

%
%
%
%
%

\end{proof}

\section{Parameter Estimation}

\begin{lem}\label{parconsistlem}
	If $n\geq Cmd$ and $\lambda_n\geq 2\mathcal{R}^*((\widehat{\Gamma}-\Gamma)\theta^*+\widehat{K}-K)$, with probability at least $1-2d\exp\left\{-\frac{\bar{\epsilon}^2}{4\underline{\epsilon}^2}md\right\}$, the regularized score matching estimator $\widehat{\theta}$ satisfies
	
\begin{align}
	\Vert\widehat{\theta}-\theta^*\Vert_2 &\leq \frac{7\lambda_n}{\underline{\epsilon}}\sqrt{d+\left| E\right|}.
\end{align}
\end{lem}

\begin{proof}
Define the function

\begin{align}
	\mathcal{E}(\delta) &= \mathcal{L}(\theta^*+\delta)-\mathcal{L}(\theta^*)+\lambda_n(\mathcal{R}(\theta^*+\delta)-\mathcal{R}(\theta^*))  \\
	&=\frac{1}{2}(\theta^*+\delta)^\top\widehat{\Gamma}(\theta^*+\delta)+(\theta^*+\delta)^\top\widehat{K}-\frac{1}{2}(\theta^*)^\top\widehat{\Gamma}(\theta^*)^\top - (\theta^*)^\top\widehat{K} \\
	&\qquad+ \lambda_n(\mathcal{R}(\theta^*+\delta)-\mathcal{R}(\theta^*))\notag \\
	&=\frac{1}{2}\delta^\top\widehat{\Gamma}\delta +  \delta^\top(\widehat{\Gamma}\theta^*+\widehat{K})  + \lambda_n(\mathcal{R}(\theta^*+\delta)-\mathcal{R}(\theta^*)) \\
	&= \frac{1}{2}\delta^\top\widehat{\Gamma}\delta +  \delta^\top(\widehat{\Gamma}\theta^*-\Gamma\theta^*+\widehat{K}-K)  + \lambda_n(\mathcal{R}(\theta^*+\delta)-\mathcal{R}(\theta^*)).
\end{align}

Since $\mathcal{E}(0)=0$, it must be that $\mathcal{E}(\widehat{\delta})\leq 0$.

Using the sub-Gaussian assumption, we may apply \citep{vershynin2010introduction} Remark 5.51, which says for any $c\in(0,1),t\geq 1$, with probability at least $1-2\exp\{-t^2md\}$, if $n\geq C(t/c)^2 md$, then for any $i\in V$, and any vector $\delta_i$
\begin{align}
	\delta_i^\top \widehat{\Gamma}_i\delta_i \geq \delta_i^\top \Gamma_i\delta_i + \bar{\epsilon} c\Vert\delta_i\Vert^2,
\end{align}

setting $c=\frac{\underline{\epsilon}}{2\bar{\epsilon}}$, and $t=1/c$ we get that if $n\geq Cd$, with probability at least $1-2\exp\left\{\frac{\bar{\epsilon}^2}{4\underline{\epsilon}^2}d\right\}$,

\begin{align}
	\frac{\underline{\epsilon}}{2}\Vert\delta_i\Vert^2 \leq \delta_i^\top\widehat{\Gamma}_i\delta_i.
\end{align}

Applying the union bound over all $i\in V$, with probability at least $1-2d\exp\left\{\frac{\bar{\epsilon}^2}{4\underline{\epsilon}^2}d\right\}$,

\begin{align}
	\frac{\underline{\epsilon}}{2}\delta^\top\delta \leq \delta^\top\widehat{\Gamma}\delta,
\end{align}

where $\delta=(\delta_1^\top,\ldots,\delta_d^\top)^\top$.
%

By (generalized) Cauchy-Schwarz,

\begin{align}
	\left|\langle\delta,(\widehat{\Gamma}-\Gamma)\theta^*+\widehat{K}-K\rangle\right| &\leq\mathcal{R}(\delta)\mathcal{R}^*((\widehat{\Gamma}-\Gamma)\theta^*+\widehat{K}-K) \\
	&\leq \frac{\lambda_n}{2}(\mathcal{R}(\delta_\mathcal{P}+\mathcal{R}(\delta_{\mathcal{P}^\perp}).\label{gencs}
\end{align}

where $\delta_A$ denotes the projection of $\delta$ onto the set $A$. From  \citep{negahban2012unified} Lemma 3, because $\mathcal{R}$ is decomposable, it  holds that

\begin{align}
	\mathcal{R}(\theta^*+\delta)-\mathcal{R}(\theta^*) \geq \mathcal{R}(\delta_{\tilde{\mathcal{P}}^\perp})-\mathcal{R}(\delta_{\tilde{\mathcal{P}}}). \label{negh}
\end{align}

Combining \eqref{gencs} and \eqref{negh},

\begin{align}
&\langle \delta,(\widehat{\Gamma}-\Gamma)\theta^*+\widehat{K}-K\rangle + \frac{\lambda_n}{2}\left(\mathcal{R}(\theta^*+\delta)-\mathcal{R}(\theta^*) \right) \\&\geq -\frac{3\lambda_n}{2}\mathcal{R}(\delta_{\tilde{\mathcal{P}}})-\frac{\lambda_n}{2}\mathcal{R}(\delta_{\tilde{\mathcal{P}}^\perp}) \geq -\frac{3\lambda_n}{2}\mathcal{R}(\delta_{\tilde{P}})\label{eq:proof1}.
\end{align} 
Using the subspace compatibility constant we have that 
\begin{align}\mathcal{R}(\delta_{\tilde{\mathcal{P}}})\leq \sqrt{d+\left| E\right|} \Vert\delta_{\tilde{\mathcal{P}}}\Vert \leq \sqrt{d+\left| E \right|}\Vert\delta\Vert.
\end{align}
Thus conditioning on the aformentioned probability,

\begin{align}
	\mathcal{E}(\delta) &\geq \frac{\underline{\epsilon}}{4} \Vert\delta\Vert^2- \frac{3\lambda}{2}\Vert\delta\Vert\sqrt{d+E} \\
	&=\Vert\delta\Vert\left(\frac{\underline{\epsilon}}{4}\Vert\delta\Vert-\frac{3\lambda_n}{2}\sqrt{d+E}\right). \label{ineq}
\end{align}

Now, consider the set 

\begin{align}
\mathcal{C} = \left\{\delta: \Vert\delta\Vert\leq\frac{7\lambda_n}{\underline{\epsilon}}\sqrt{d+E}\right\}.
\end{align}

$\mathcal{C}$ is a compact, convex set. Furthermore, for all $\delta\in\partial\mathcal{C}$, from \eqref{ineq} we see that $\mathcal{E}(\delta)>0$. Also observe that $0\in\text{int}\mathcal{C}$. Since $\mathcal{E}(\widehat{\delta})\leq 0$, it must follow that $\widehat{\delta}\in\text{int}\mathcal{C}$, in other words

\begin{align}
	\Vert\widehat{\theta}-\theta^*\Vert_2 &\leq \frac{7\lambda_n}{\underline{\epsilon}} \sqrt{d+E}.
\end{align}

\end{proof}

\begin{proof}[Proof of Theorem \ref{parconsist}]
Applying a concentration bound to $\widehat{K}_{ij}^u-K_{ij}^u$ in addition to a union bound, we have that $\Vert\widehat{K}-K\Vert_{\max}>t$ with probability no more than $\exp\{2\log(md)-c_2nt^2\}$ for $t\leq\nu$, for constants $c_2,\nu>0$. Similarly,

\begin{align}
	\Vert(\widehat{\Gamma}-\Gamma)\theta^*\Vert_{\max}&\leq 2\kappa_{\theta^*,1}\Vert\widehat{\Gamma}-\Gamma\Vert_{\max},
\end{align}

and $\Vert\widehat{\Gamma}-\Gamma\Vert_{\max}>t$ with probability no more than $\exp\{2\log(md)-c_1nt^2\}$ for $t<\nu_2$ for $c_1,\nu_2>0$. Furthermore, observe that for any vector $\theta$,

\begin{align}
	\mathcal{R}^*(\theta) &\leq \sqrt{m} \Vert\theta\Vert_{\max},
\end{align}

thus setting $\lambda_n\asymp \sqrt{\frac{mk_{1,\theta}^2\log(md)}{n}}$, the conditions in Lemma \ref{parconsistlem} will be satisfied with probability approaching one.
\end{proof}

\section{Model Selection}

Our proof technique is the \emph{primal dual witness} method, used previously in analysis of model selection for graphical models \citep{modelselecttaylor2014,ravikumar2011high}. It proceeds as follows: construct a primal-dual pair $(\widehat{\theta},\widehat{Z})$ which satisfies $\text{supp}(\widehat{\theta})=\text{supp}(\theta^*)$, and also satisfies the stationary conditions for \ref{scorematch} with high probability. The stationary conditions for \ref{scorematch} are
\begin{align}
	\widehat{\Gamma}\widehat{\theta}+\widehat{K} + \lambda_n\widehat{Z} = 0,
\end{align}

where $\widehat{Z}$ is an element of the subdifferential $\partial\mathcal{R}(\widehat{\theta}_e)$:

\begin{align}
	(\partial\mathcal{R}(\theta_e))_{ij} = \begin{cases}
		\{\theta_{ij}: \Vert\theta_{ij}\Vert_2\leq 1\}, & \text{  if } \theta_{ij}=0; \\
		\frac{\theta_{ij}}{\Vert\theta_{ij}}\Vert_2, & \text{  if } \theta_{ij}\not=0.
	\end{cases}
\end{align}
From this we may conclude that there exists a solution to \ref{scorematch} is sparsistent. In particular, we have the following steps:
\begin{enumerate}
	\item Set $\widehat{\theta}_{E^c} = 0$;
	\item Set $\widehat{Z}_{ij} = \partial\mathcal{R}(\theta^*_e)_{ij}=\frac{\bar{\theta}^*_{ij}}{\Vert\bar{\theta}^*_{ij}\Vert}$ for $(i,j)\in E$;
	\item Given these choices for $\widehat{\theta}_{E^c}$ and $\widehat{Z}_E$, choose $\widehat{\theta}_E$ and $\widehat{Z}_{E^c}$ to satisfy the stationary condition \ref{scorematch}.
\end{enumerate}

For our procedure to succeed, we must show this primal-dual pair $(\widehat{\theta},\widehat{Z})$ is optimal for \ref{scorematch}, in other words

\begin{align}
	&\widehat{\theta}_{ij} \not = 0 , &\text{for}\qquad (i,j)\in E;\label{pdscm1} \\
	&\Vert\widehat{Z}_{ij}\Vert <1, &\text{for}\qquad (i,j)\not\in E. \label{pdscm2}
\end{align}

In the sequel we show these two conditions hold with probability approaching one.

\begin{lem}\label{primaldual1}
	Suppose that $\Vert\widehat{\Gamma}_{EE}-\Gamma_{EE}\Vert_{\max}\leq\frac{1}{2ms\kappa_\Gamma}$. Then there exists a solution to \eqref{scorematch}, $\widehat{\theta}$, satisfying
	
\begin{align}
	\Vert\widehat{\theta}_{E}-\theta^*_{E}\Vert_\infty &\leq 2\kappa_\Gamma\left(2\kappa_{1,\theta}\Vert(\widehat{\Gamma}_{EE}-\Gamma_{EE})\Vert_{\max}+\Vert\widehat{K}-K\Vert_\infty+\lambda_n/\sqrt{m}\right).
\end{align}
\end{lem}

\begin{proof}
The stationary condition for $\widehat{\theta}_E$, observing that $\widehat{\theta}_{E^c}=\theta^*_{E^c}=0$, is given by

\begin{align}
	\widehat{\Gamma}_{EE}\widehat{\theta}_E+\widehat{K}_E + \lambda_n \widehat{Z}_E&=0.
\end{align}

Re-arranging and observing that $\Gamma_{EE}\theta^*_E=-K_E$, we have

\begin{align}
	\widehat{\Gamma}_{EE}\widehat{\theta}_E+\widehat{K}_E + \lambda_n \widehat{Z} &=\widehat{\Gamma}_{EE}\widehat{\theta}_E-\Gamma_{EE}\theta^*_E+\widehat{K}_E -K_E+ \lambda_n \widehat{Z}_E \\
	&=(\widehat{\Gamma}_{EE}-\Gamma_{EE})\widehat{\theta}+\Gamma_{EE}(\widehat{\theta}_E-\theta^*_E)+\widehat{K}_E-K_E+\lambda_n\widehat{Z}_E.
\end{align}

Consider the map

\begin{align}
	F(\Delta_E) = -\Gamma_{EE}^{-1}\left((\widehat{\Gamma}_{EE}-\Gamma_{EE})(\Delta_E+\theta^*_E)+\Gamma_{EE}\Delta_E+\widehat{K}_E-K_E+\lambda_n\widehat{Z}_E\right)+\Delta_E.
\end{align}

$F$ has a fixed point $F(\Delta_E)=\Delta_E$ at $\widehat{\Delta}_E=\widehat{\theta}_E-\theta^*_E$ for any solution $\widehat{\theta}$. Define $\tilde{r}:=2\kappa_\Gamma\left(2\kappa_{1,\theta}\Vert(\widehat{\Gamma}_{EE}-\Gamma_{EE})\Vert_{\max}+\Vert\widehat{K}_E-K_E\Vert_\infty+\lambda_n/\sqrt{m}\right)$. If we can show $\Vert F(\Delta)\Vert_\infty \leq\tilde{r}$ for each $\Vert\Delta\Vert_\infty\leq\tilde{r}$, from Brouwer's fixed point theorem \citep{ortega2000iterative}, it follows that some fixed point satisfies $\Vert\widehat{\Delta}\Vert_\infty\leq\tilde{r}$. For $\Vert\Delta\Vert_\infty\leq\tilde{r}$,

\begin{align}
	\Vert F_{ij}\Vert_2 &\leq \frac{\kappa_\Gamma}{\sqrt{m(d+\left| E\right|)}}\left(\Vert (\widehat{\Gamma}_{EE}-\Gamma_{EE})(\Delta_E+\theta^*_E)\Vert_2 + \Vert\widehat{K}_{E}-K_{E}\Vert_2 + \lambda_n\Vert\widehat{Z}_E\Vert_2\right) \\
	&\leq \kappa_\Gamma \left(\Vert (\widehat{\Gamma}_{ij,E}-\Gamma_{ij,E})(\Delta_E+\theta^*_E)\Vert_\infty+\Vert\widehat{K}_{E}-K_{E}\Vert_\infty + \lambda_n/\sqrt{m}\right).
\end{align}

Now,
\begin{align}
	\Vert(\widehat{\Gamma}_{EE}-\Gamma_{EE})\Delta\Vert_\infty &\leq \left\{\max_{i\in V}\sum_{j\in V, k\leq m} \left|\Delta_{ij}^k\right|\right\}\cdot  \Vert\widehat{\Gamma}_{EE}-\Gamma_{EE}\Vert_{\max} \\
	&\leq ms \Vert\Delta\Vert_\infty \Vert\widehat{\Gamma}_{EE}-\Gamma_{EE}\Vert_{\max},
\end{align}

and similarly,

\begin{align}
	\Vert(\widehat{\Gamma}_{EE}-\Gamma_{EE})\theta^*_E\Vert_\infty &\leq 2\kappa_{1,\theta}\Vert\widehat{\Gamma}_{EE}-\Gamma_{EE}\Vert_\infty.
\end{align}

Thus, if $\Vert\widehat{\Gamma}_{EE}-\Gamma_{EE}\Vert_{\max}\leq\frac{1}{2ms\kappa_\Gamma}$,
\begin{align}
	\Vert F\Vert_\infty \leq \max_{ij}\Vert F_{ij}\Vert_2 &\leq \kappa_\Gamma\left(\Vert\widehat{\Gamma}_{EE}-\Gamma_{EE}\Vert_{\max} (2\kappa_{1,\theta}+ms\tilde{r})+\Vert\widehat{K}-K\Vert_\infty+\lambda_n/\sqrt{m}\right) \\
	&\leq \frac{\tilde{r}}{2}+\frac{\tilde{r}}{2}\leq \tilde{r}.
\end{align}
\end{proof}

\begin{lem}\label{primaldual2}
Suppose that $\sqrt{m}\Vert\widehat{K}-K\Vert_\infty \leq\frac{\tau\lambda_n}{4}$,  $m^{1/2}\Vert\widehat{\Gamma}-\Gamma\Vert_{\max}(\kappa_\theta+s\sqrt{m}\lambda_n)\leq\frac{\tau\lambda_n}{4}$, and $\lambda_n/\sqrt{m}\geq 2\kappa_\Gamma(ms\kappa_\theta\Vert\widehat{\Gamma}_{EE}-\Gamma_{EE}\Vert_{\max}+\Vert\widehat{K}-K\Vert_\infty)$. Then for each $(i,j)\in E^c$,

\begin{align}
	\Vert\widehat{Z}_{ij}\Vert_2<1.
\end{align}
\end{lem}
\begin{proof}

For $(i,j)\in E^c$, the stationary conditions are

\begin{align}
	0&=\widehat{\Gamma}_{E^cE}\widehat{\theta}_E+\widehat{K}_{E^c}+\lambda_n\widehat{Z}_{E^c} \\
	&= \Gamma_{E^cE}(\widehat{\theta}_E-\theta^*_E) +(\widehat{\Gamma}_{E^cE}-\Gamma_{E^cE})\widehat{\theta}_E \\
	&\qquad+\widehat{K}_{E^c}-K_{E^c}+\lambda_n\widehat{Z}_{E^c},\notag
\end{align}

re-arranging and plugging in the stationary conditions for $\widehat{\theta}_E$, we have for $(i,j)\in E^c$,

\begin{align}
	\widehat{Z}_{ij} &= \frac{1}{\lambda_n} \bigg\{-\Gamma_{ij,E}\Gamma_{EE}^{-1}(-(\widehat{\Gamma}_{EE}-\Gamma_{EE})\widehat{\theta}_E-(\widehat{K}_E-K_E)-\lambda_n\widehat{Z}_E) \\
	&\qquad\qquad-(\widehat{\Gamma}_{ij,E}-\Gamma_{ij,E})\widehat{\theta}_E-\widehat{K}_{ij}+K_{ij}\bigg\}.\notag
\end{align}

Applying the $L_2$ norm,

\begin{align}
	\Vert\widehat{Z}_{ij}\Vert_2 &=\frac{1}{\lambda_n}\bigg\{\Vert\Gamma_{ij,E}\Gamma_{EE}^{-1}\Vert_2(\Vert(\widehat{\Gamma}_{EE}-\Gamma_{EE})\widehat{\theta}_E\Vert_2 + \Vert\widehat{K}_E-K_E\Vert_2+\lambda_n\Vert\widehat{Z}_E\Vert_2)\\
	&\qquad+\Vert\widehat{\Gamma}_{ij,E}-\Gamma_{ij,E}\widehat{\theta}_E\Vert_2+\Vert\widehat{K}_{ij}-K_{ij}\Vert_2 \bigg\}\notag\\
	&\leq\frac{1}{\lambda_n}\left\{ \sqrt{m}(2-\tau)(\Vert\widehat{K}-K\Vert_\infty+\Vert(\widehat{\Gamma}-\Gamma)\widehat{\theta}_E)\Vert_\infty\right\}+1-\tau.
\end{align}

Observe that since $\Vert\widehat{\theta}-\theta^*\Vert_{\max}\leq\tilde{r}$,

\begin{align}
	\Vert(\widehat{\Gamma}-\Gamma)\widehat{\theta}\Vert_\infty &
	\leq \Vert\widehat{\Gamma}-\Gamma\Vert_{\max}(\kappa_{1,\theta}+ms\tilde{r}) \\
	&\leq \Vert\widehat{\Gamma}-\Gamma\Vert_{\max}\left(\kappa_{1,\theta}+2s\sqrt{m}\lambda_n\right),
\end{align}

so $\Vert\widehat{Z}_{ij}\Vert_2$ is bounded by
\begin{align}
	\frac{1}{\lambda_n}\left\{(2-\tau)\sqrt{m}\Vert\widehat{K}-K\Vert_\infty+\sqrt{m}(2-\tau)\Vert\widehat{\Gamma}-\Gamma\Vert(\kappa_{1,\theta}+2s\sqrt{m}\lambda_n)\right\}+1-\tau.
\end{align}

if $\sqrt{m}\Vert\widehat{K}-K\Vert_\infty\leq\frac{\tau\lambda_n}{4}$ and $\sqrt{m}\Vert\widehat{\Gamma}-\Gamma\Vert_{\max}(2\kappa_{1,\theta}+2s\sqrt{m}\lambda_n)\leq\frac{\tau\lambda_n}{4}$, this is bounded by
\begin{align}
	(2-\tau)\left(\frac{\tau}{4}+\frac{\tau}{4}\right)+1-\tau\leq 1-\frac{\tau}{2}<1.
\end{align}
\end{proof}

\begin{proof}[Proof of Theorem \ref{modelselectthm}]
Using a concentration bound for $\widehat{K}_{ij}^u-K_{ij}^u$ and applying a union bound, we have that when $t\leq \nu_1$ for some $\nu$, $\Vert\widehat{K}-K\Vert_{\max}>t$ with probability no more than $\exp\{2\log(md)-c_2nt^2\}$ for a constant $c_2$. Similarly,
\begin{align}
	\Vert(\widehat{\Gamma}-\Gamma)\theta^*\Vert_{\max}&\leq 2\kappa_{1,\theta}\Vert\widehat{\Gamma}-\Gamma\Vert_{\max},
\end{align}

and for $t\leq \nu_2$, $\Vert\widehat{\Gamma}-\Gamma\Vert_{\max}>t$ with probability no more than $\exp\{2\log(md)-c_1nt^2\}$. Thus setting $\lambda_n= C\sqrt{\frac{m\kappa_{1,\theta}^2\log(md)}{n}}$ for sufficiently large $C$, and if $\sqrt{\frac{m^2s^2\log md}{n}}=o(1)$, the assumptions of lemma \ref{primaldual2} will be satisfied with probability approaching one. Further, assumption \eqref{pdscm1} is satisfied when $\Vert\widehat{\theta}-\theta^*\Vert_\infty\leq\frac{\rho^*}{2}$. Since $\Vert\widehat{\theta}-\theta^*\Vert_\infty = O(\lambda_n/\sqrt{m})$, we require $\frac{\lambda_n}{\rho^*\sqrt{m}}=o(1)$.

\end{proof}

\begin{lem} \label{bonnetlem}
Let $\phi_k$ be the $k$th orthonormal Legendre polynomial on $[0,1]$. then
\begin{align}
	\left| x(1-x)\frac{\partial \phi_k(x)}{\partial x}\right| &= O(k^{3/2}),\\
	\left| x(1-x)\frac{\partial^2\phi_k(x)}{\partial x^2}\right| &= O(k^{5/2}).
\end{align}

\end{lem}
\begin{proof}
	From Bonnet's recursion formula \citep{abramowitz1964handbook},
	\begin{align}
		x(1-x)\frac{d\phi_k(x)}{dx} &= \frac{k}{2}\left((2x-1)\phi_k(x)-\sqrt{\frac{2k+1}{2k-1}}\phi_{k-1}(x)\right), \label{bonnet}
	\end{align}
	so taking absolute values of each side, and using $\left|\phi_k\right| \leq \sqrt{2k+1}$,
	\begin{align}
		\left| x(1-x)\frac{d\phi_k}{dx}\right| &\leq k\left(\left|\phi_k\right|+\left|\phi_{k-1}\right|\right)\\
		&= O\left(k^{3/2}\right). \label{derivbound}
	\end{align}
	
	Now, using Legendre's differential equation \citep{abramowitz1964handbook},
	\begin{align}
		4x(1-x)\frac{d^2\phi_k}{dx^2}+2(2x-1)\frac{d\phi_k}{dx}-k(k+1)\phi_k &=0,
	\end{align}
	and using the fact that $\frac{d\phi_{k}}{dx}\leq \frac{k(k+1)\sqrt{2k+1}}{2}$, we find that
	\begin{align}
		\left|x(1-x)\frac{d^2\phi_k}{dx^2}\right| &\leq \frac{1}{2}\left|\frac{d\phi_k}{dx}\right|+\frac{1}{4}k(k+1)\left|\phi_k\right|\\
		&= O\left(k^{5/2}\right).
	\end{align}
	
\end{proof}

\begin{proof}[Proof of Theorem \ref{msnonp}]
The proof technique is essentially the same as Theorem \ref{modelselectthm} so we omit some details. The main difference is that here $K^*= - \Gamma^*\theta^*-\Gamma_{T}\theta_T$, so we must deal with one additional term in the analysis, the bias from truncation. Suppose $m_1=m_2$. We choose $\lambda_n$ so that with high probability,
\begin{align}
	\lambda_n &=\Omega\left( m_2\kappa_{1,\theta}\Vert\widehat{\Gamma}_{EE}-\Gamma_{EE}\Vert_{\max}+m_2\Vert\widehat{K}-K\Vert_{\max}+\kappa_Tm_2\Vert\theta_T\Vert_{\max}\right),\\
	\lambda_n &\rightarrow 0,
\end{align}

as well as requiring $n=\Omega(m_2^4 s^2 \log md)$.

Now, applying Lemma \ref{bonnetlem},  $\Vert A(x)\Vert_{\max} = O\left(m_2^4\right)$, so applying Hoeffding's inequality and a union bound as well as the boundedness assumption \ref{assump4}, $\Vert\widehat{\Gamma}_{EE}-\Gamma_{EE}\Vert_{\max}\leq C\sqrt{\frac{m_2^8 \log(m_2 d)}{n}}$ with probability approaching one, for sufficiently large constant $C$. Similarly, $\Vert\widehat{K}-K\Vert_{\max}\leq C'\sqrt{\frac{m_2^6\log(m_2d)}{n}}$ with probability approaching one. Furthermore, $\Vert\theta_T\Vert_{\max}=O(m_2^{-r-1/2})$. Thus, supposing $\kappa_{1,\theta}=O(m_2^2)$, we need
\begin{align}
	\lambda_n \asymp O\left(\sqrt{\frac{m_2^{12}\log(m_2d)}{n}} + m_2^{-r+1/2}\right).
\end{align}

Balancing the two terms, we choose $m_2\asymp n^{\frac{1}{2r+13}}$, so $\lambda_n\asymp \sqrt{\frac{\log(nd)}{n^{\frac{2r-1}{2r+13}}}}$. The stated sample complexity ensures that $\lambda_n\rightarrow 0$. Furthermore, $\Vert\widehat{\theta}_E-\theta^*_E\Vert_{\max}=O\left(\lambda_n/m_2\right)$, so we require $\frac{\lambda_n}{m_2\rho^*}\rightarrow 0$.

\end{proof}

\chapter{Software}
Software for Chapter 3 is available at https://github.com/geb5101h/trw. Software for Chapter 4 is available at https://github.com/geb5101h/quasr.
\bibliographystyle{apalike}
\singlespacing
\bibliography{trw.bib}


\end{document}